%% file: paper.tex
\documentclass[3p]{elsarticle}

\usepackage{inputenc}
\usepackage{amsmath}
\usepackage{amssymb, amsfonts}
\usepackage{graphicx,epstopdf,bm}
\usepackage[caption=false]{subfig}
\usepackage{dcolumn}%

\usepackage{lineno}
\usepackage[colorlinks,bookmarks,urlcolor=blue,citecolor=blue,linkcolor=blue]{hyperref}
\usepackage[capitalize]{cleveref}

\usepackage[T1]{fontenc}
\usepackage{bbm}
\usepackage{mathtools}
\usepackage{multirow}
\usepackage{xcolor}

\usepackage{algorithm}
\usepackage{algpseudocode}

\usepackage{amsthm}
\theoremstyle{plain}
\newtheorem{theorem}{Theorem}[section]

\newtheorem{remark}{Remark}[section]

\input{my_commands.tex}

\newcommand{\commenta}[1]{#1}
\newcommand{\commentb}[1]{#1}
\newcommand{\commentab}[1]{#1}

\begin{document}
\title{An Unstructured Mesh Convergent Reaction-Diffusion Master Equation for Reversible Reactions}

\author[bu]{Samuel A. Isaacson}
\ead{isaacson@math.bu.edu}
\author[bu]{Ying Zhang}
\ead{phzhang@bu.edu}


\address[bu]{Department of Mathematics and Statistics, Boston
  University, Boston, MA 02215}




\begin{abstract}
  The convergent reaction-diffusion master equation (CRDME) was
  recently developed to provide a lattice particle-based stochastic
  reaction-diffusion model that is a convergent approximation in the
  lattice spacing to an underlying spatially-continuous particle
  dynamics model. The CRDME was designed to be identical to the
  popular lattice reaction-diffusion master equation (RDME) model for
  systems with only linear reactions, while overcoming the RDME's loss
  of bimolecular reaction effects as the lattice spacing is taken to
  zero. In our original work we developed the CRDME to handle
  bimolecular association reactions on Cartesian grids. In this work
  we develop several extensions to the CRDME to facilitate the
  modeling of cellular processes within realistic biological
  domains. Foremost, we extend the CRDME to handle reversible
  bimolecular reactions on unstructured grids. Here we develop a
  generalized CRDME through discretization of the spatially continuous
  volume reactivity model, extending the CRDME to encompass a larger
  variety of particle-particle interactions. Finally, we conclude by
  examining several numerical examples to demonstrate the convergence
  and accuracy of the CRDME in approximating the volume reactivity
  model.

\end{abstract}


\maketitle

\section{Introduction}
The dynamics of many biological processes rely on an interplay between
spatial transport and chemical reaction. Examples of such processes
include gene regulation, where proteins undergo a diffusive search
throughout the nuclei of cells to reach specific DNA binding
sites~\cite{IsaacsonPNAS2011,IsaacsonBmB2013,BergSlidingI}; and signal
transduction, where phosphorylation and dephosphorylation of signaling
proteins can occur at spatially segregated locations leading to
spatial gradients of activated and inactivated
proteins~\cite{KholodenkoDiffControl}.  The interplay of diffusion and
reaction can also influence whether such signals are able to
successfully propagate from the cell membrane to nucleus
\cite{MunozGarcia:2009hd,Kholodenko:2010jv}. A challenge in developing
mathematical models of these processes is that they often occur within
spatially heterogeneous environments. Organelles, filamentous
structures, and macro-molecules can represent both steric barriers and
reactive targets for proteins undergoing diffusive transport within
cells. Even the basic geometrical property of cell shape can
dramatically alter information flow in signaling
pathways~\cite{IyengarCellShape08} and the ability of cells to sense
gradients~\cite{Spill:2016kc}.

An additional modeling challenge arises from experimental observations
that many biochemical processes within single cells have stochastic
dynamics~\cite{ArkinPNAS1997,CollinsNatureEukNoise,RaserScience2004,Ladbury:2012hq}.
To facilitate the study of reaction-diffusion processes at the scale
of a single cell it is therefore necessary to develop mathematical
models and numerical simulation methods that can accurately account
for the stochastic diffusion and reaction of molecules within
realistic cellular domains.  Particle-based stochastic
reaction-diffusion models are one popular approach to model the
stochastic diffusion and reactions between individual molecules within
cells~\cite{SmoluchowskiDiffLimRx,CollinsKimballPartialAdsorp,DoiSecondQuantA,DoiSecondQuantB,TeramotoDoiModel1967,GardinerRXDIFFME}. While
there are a plethora of particle-based stochastic reaction-diffusion
models and simulation methods that have been used for modeling
cellular processes~\cite{AndrewsBrayPhysBio2004,IsaacsonSJSC2006,
  Bartol2008,ErbanChapman2009,Arjunan:2009fd,SchultenWholCellEColi2011kk,HellanderURDME2012,DeSchutterSTEPS2012uh,IsaacsonCRDME2013,
  SchonebergNoeReaddy2013,DonevJCP2010, WoldeEgfrdPNAS2010,
  DonevSRBDPaper2017}, they all share a number of basic features.
Such models typically represent proteins, mRNAs, and other cellular
molecules as point particles or small spheres. In the most common form
that we will subsequently consider, individual molecules move by
diffusion processes, or by random walk approximations. Linear zeroth
order reactions, e.g.  $\varnothing \to \textrm{A}$, occur with a
fixed probability per time, while linear first order reactions, e.g.
$\textrm{A} \to \textrm{B}$, generally occur with a fixed probability
per time for each molecule of the reactant species (i.e. each
$\textrm{A}$ molecule). Nonlinear second order reactions,
e.g. $\textrm{A} + \textrm{B} \to \textrm{C}$, occur according to a
variety of mechanisms when an individual pair of \textrm{A} and
\textrm{B} molecules are sufficiently close. It is in the choice of
this association mechanism that the common particle-based stochastic
reaction-diffusion models tend to differ.

In this work we focus on the spatially continuous volume reactivity
model, which approximates molecules as point particles that move by
Brownian Motion. A bimolecular reaction of the form
$\textrm{A} + \textrm{B} \to \textrm{C}$ is modeled through an
interaction function that determines the probability density per time
an \textrm{A} molecule at $\vx$ and a \textrm{B} molecule at $\vy$ can
react to produce a \textrm{C} molecule at $\vz$. The Doi model is the
most common form of the volume reactivity model used in studying
cellular processes. It was popularized by
Doi~\cite{DoiSecondQuantA,DoiSecondQuantB}, who attributes the
original model to Termamoto and
Shigesada~\cite{TeramotoDoiModel1967}. In the Doi model, a bimolecular
reaction such as $ \textrm{A} + \textrm{B} \to \textrm{C}$ is modeled
for an individual pair of \textrm{A} and \textrm{B} molecules as
occurring with a fixed probability per time, $\lambda$, once their
separation is smaller than a specified reaction radius $\rb$, i.e.
$\abs{\vx - \vy} < \rb$.

A number of methods have been proposed for generating approximate
realizations of the stochastic process of diffusing and reacting
particles described by the Doi
model~\cite{ErbanChapman2009,ErbanChapman2011,DonevSRBDPaper2017,IsaacsonCRDME2013}.
Brownian Dynamics methods discretize this process in time using an
operator splitting approach applied to the underlying forward
Kolmogorov
equation~\cite{ErbanChapman2009,ErbanChapman2011,DonevSRBDPaper2017}. This
separates the diffusive and reactive processes into separate time
steps, allowing their individual approximation. For example,
in~\cite{DonevSRBDPaper2017} this splitting is coupled with exact
sampling of the diffusive motion of molecules during the diffusive
timestep (generally only possible in freespace or periodic domains),
and exact sampling of reactions between stationary molecules using a
variant of the Stochastic Simulation Algorithm
(SSA)~\cite{GillespieJPCHEM1977,KalosKMC75} during the reactive
timestep. This has allowed the simulation of two-dimensional pattern
formation systems with on the order of a million particles in square
domains.

In~\cite{IsaacsonCRDME2013} we developed an alternative convergent
approximation and corresponding numerical simulation method by
discretizing the Doi model in space instead of time. The resulting
convergent reaction-diffusion master equation (CRDME) represents a
spatial discretization of the forward Kolmogorov equation for the Doi
model, approximating the continuous Brownian motion and reactions of
molecules by a continuous time jump-process in which molecules hop and
react on an underlying spatial mesh. \commenta{This jump process gives
  the number of molecules of each chemical species at each mesh
  location at a given time. The CRDME corresponds to the set of ODEs
  describing the change in time of the probability the jump-process
  has a given value.}

The CRDME was designed with two major goals; providing a convergent
approximation to the Doi model, and being equivalent to the popular
lattice reaction-diffusion master equation (RDME) model in its
treatment of \emph{linear} reactions and spatial transport. In the
RDME molecules move by hopping between mesh voxels through
continuous-time random
walks~\cite{GardinerRXDIFFME,GardinerHANDBOOKSTOCH}. Within each mesh
voxel molecules are assumed to be well-mixed, i.e. uniformly
distributed. Zeroth and first order reactions are treated similarly to
the volume-reactivity model. Nonlinear second order reactions of the
form $\textrm{A} + \textrm{B} \to \textrm{C}$ occur with a fixed
probability per time for two molecules located within the \emph{same}
voxel. In this way the RDME can be formally interpreted as a
discretization of the volume reactivity model, where interaction
functions between molecules are given by delta functions in their
separation~\cite{IsaacsonRDMENote}. \commenta{The CRDME's primary
  difference from the RDME is in also allowing molecules within
  \emph{nearby} voxels to react, arising from \emph{direct}
  discretizion of the Doi model.}

The choice to maintain consistency with the RDME for linear reactions
and spatial transport is due to its many attractive features and
extensions that enable the modeling of cellular processes.  Foremost,
in the absence of nonlinear reactions the RDME can be interpreted as a
discretization of the forward equation for the volume reactivity
model. In the absence of any reactions this forward equation simply
corresponds to a high-dimensional diffusion equation in the combined
coordinates of all molecules.  By exploiting this connection standard
PDE discretization techniques can be used to extend the RDME to
include spatial transport mechanisms that are needed to model cellular
processes. These include extensions to include drift due to potential
fields~\cite{ElstonPeskinJTB2003,IsaacsonPNAS2011}, advection due to
underlying velocity fields~\cite{HellanderActTransport2010}, and
embedded boundary~\cite{IsaacsonSJSC2006} and unstructured
mesh~\cite{DeSchutterSTEPS2012uh,LotstedtFERDME2009} discretization
methods for deriving jump rates in meshes of complex domain
geometries. Such discretizations can be chosen to preserve important
physical properties of the underlying spatially-continuous transport
model, such as detailed balance of spatial fluxes and Gibbs-Boltzmann
equilibrium states for models that include drift due to potential
fields~\cite{ElstonPeskinJTB2003,IsaacsonPNAS2011}.


A second benefit to the RDME model is that in appropriate
large-population limits it is \emph{formally} expected to limit to
more macroscopic stochastic PDE (SPDE) models, and as the population
further grows, to reaction-diffusion PDE models for continuous
concentration fields~\cite{Arnold1980ww,Arnold1980dl}. These
connections have been exploited to develop computationally efficient
hybrid models that represent different chemical species or reactions
at different physical levels of
detail~\cite{DuncanMultiscaleStochChem2016gu,HellanderActTransport2010},
or that partition spatial domains into regions of low populations
represented by jump processes satisfying the RDME and regions of
higher populations represented through continuous concentration fields
satisfying reaction-diffusion SPDEs or deterministic
PDEs~\cite{FleggYatesMultiscale2015kd,YatesHarrisonMultiscaleCoup2016hx}.

Finally, for researchers interested in simply using the RDME to model
cellular systems, there are a number of optimized, publicly
available software packages that can simulate the jump process of
molecules reacting and diffusing within complex geometries. These
include the unstructured mesh
URDME/PYURDME/StochSS~\cite{HellanderURDME2012} and
STEPS~\cite{DeSchutterSTEPS2012uh} software packages, and the
GPU-based structured mesh Lattice Microbes software
package~\cite{SchultenWholCellEColi2011kk} (which has successfully
simulated systems containing hundreds of thousands to order of one
million molecules within complex geometries corresponding to
three-dimensional whole bacterial cells).

For these many reasons the RDME has become a popular tool for studying
cellular processes in which both noise in the chemical reaction
process and the diffusion of molecules may be
important. Unfortunately, there is a major practical difficulty when
using the RDME. The formal continuum limit of the RDME, the
volume-reactivity model with delta-function interactions, is only
correct in one-dimension. In two or more dimensions, in \commentb{the continuum
limit that the mesh spacing in the RDME approaches zero,} bimolecular
reactions are
lost~\cite{IsaacsonRDMELims,IsaacsonRDMELimsII,Hellander:2012jk}. This
loss of reaction occurs as in the RDME molecules are represented by
point particles that can only react when located within the same mesh
voxel.  In the continuum limit that the mesh spacing approaches zero
the RDME converges to a model in which molecules correspond to point
particles undergoing Brownian motion, which can only react when
located at the same position. In two or more dimensions the
probability of two molecules ever reaching the same position is zero,
and so the particles are never able to encounter each other and react.

This loss of bimolecular reactions in the limit that the mesh spacing
approaches zero is a challenge in using RDME-type models. In contrast
to numerically solving PDE models, or solving the RDME for systems
with only linear reactions, one can not expect that shrinking the mesh
spacing will eventually give better accuracy in approximating some
underlying spatially-continuous model. Moreover, for a given chemical
system it is not known how to determine an ``optimal'' mesh spacing
that minimizes the approximation error for convergent linear reaction
and spatial transport terms, while avoiding errors due to an
unphysical decrease in the occurrence of bimolecular reactions. For
specific chemical systems one may be able to numerically estimate an
optimal mesh spacing, but even then there is no explicit control on
how well the RDME approximates any particular spatially continuous
stochastic reaction-diffusion model.

Several methods have been proposed to overcome this challenge in using
the
RDME. In~\cite{ErbanChapman2009,ElfPNASRates2010,Hellander:2012jk,HellanderJumpRatesUnstructMesh2017ed}
bimolecular reaction rates in the RDME are renormalized, allowing the
more accurate approximation of statistics from spatially-continuous
particle models over a broader range of mesh spacings than the
standard RDME. While such methods still lose bimolecular reactions in
the continuum limit that the mesh spacing approaches zero, and are
hence non-convergent, they can provide accurate statistics over much
larger ranges of mesh spacings than the RDME (even approaching length
scales comparable to the Doi reaction-radius
$\epsilon$~\cite{Hellander:2012jk}). Such approaches have recently
been extended to unstructured grids, to allow for more accurate
RDME-like models in complex geometries corresponding to biological
domains~\cite{HellanderJumpRatesUnstructMesh2017ed}. A second approach
is to consider multiscale couplings, where the RDME model is replaced
with Brownian Dynamics or other continuous particle dynamics models in
regions where increased accuracy, and hence smaller mesh spacings, are
desired~\cite{FleggErban2011jl}.

In this work we take a different approach, developing a CRDME model
that \emph{converges} to the spatially-continuous volume-reactivity model,
but is consistent with the RDME in its handling of linear reactions
and spatial transport. In this way the CRDME can leverage both the
large body of extensions to the RDME to facilitate more general
spatial transport mechanisms in complex geometries, and the optimized
simulation methods and multiscale couplings developed for the
RDME. Here we extend the Cartesian grid CRDME developed
in~\cite{IsaacsonCRDME2013} for bimolecular association reactions to
reversible reactions on general unstructured grids. The new unstructured grid
CRDME can be used to simulate chemical processes in complex domain
geometries as needed for studying cellular systems. It also has
the appealing property of preserving pointwise detailed balance at
steady-state whenever the spatially continuous volume reactivity model
it approximates also satisfies pointwise detailed balance.

To construct an unstructured mesh CRDME for reversible reactions we
utilize a hybrid discretization approach. We begin in the next section
by approximating the continuous Brownian motion of a single molecule
within a bounded domain by a lattice jump process. We review the
method developed in~\cite{LotstedtFERDME2009}, which derives
transition rates for the hopping of one molecule between neighboring
mesh voxels. Here the mesh is given by polygonal voxels, representing
the dual mesh to a triangulation of the original domain. The diffusive
jump rates are derived by finite element discretization of the
diffusion operator on the triangulated
mesh~\cite{LotstedtFERDME2009}. In \cref{S:methodSect} we then
consider the approximation of reversible bimolecular reactions on the
same underlying polygonal dual mesh. We begin by introducing the
abstract spatially-continuous volume-reactivity model for reversible
reactions in \cref{S:absVolReactModel}. In \cref{S:volreactDoiModel}
we show which choice of interaction functions in the abstract
volume-reactivity model gives rise to the popular Doi reaction
model. \Cref{S:rxDiscretMethod} then develops our finite volume method
for deriving a jump process approximation to the association and
dissociation reaction terms in the abstract volume reactivity
model. Combining the finite element discretization of spatial
transport terms from \cref{S:diffApprox} with the finite volume
discretization of reaction terms from \cref{S:rxDiscretMethod}, in
\cref{S:unstructCRDME} we obtain a CRDME for a pair of molecules that
can undergo the reversible
$\textrm{A} + \textrm{B} \leftrightarrows \textrm{C}$
reaction. \commenta{Here the finite element discretization weights
  determine the probability per time that a molecule can hop from a
  given voxel to its neighbor, while the finite volume discretization
  weights determine the probability per time reactants in different voxels
  can react. Combined they define a jump process for the two molecules
  hopping and reacting on the unstructured mesh. The CRDME then
  describes the time evolution of the probability this jump process
  has a given value.}

We next explain \commenta{how this two-particle CRDME} can be
generalized to a system with an arbitrary number of molecules of each
species, and summarize the transitions comprising the general
multiparticle reaction-diffusion jump process in
\cref{tab:aAndBToCRxs}. We also summarize \cref{ap:multipartModel},
where it is shown how to formulate the general multiparticle abstract
volume-reactivity model, how to discretize this model to obtain a
multiparticle CRDME, and how this model can be rewritten in a form
that looks similar to the RDME. In \cref{S:crdmeTorRDME} we briefly
discuss the relationship between the reversible binding CRDME model
and the RDME, pointing out that the RDME can be interpreted as a
(non-convergent) approximation of the abstract volume-reactivity model
that is similar to the CRDME, but restricts reactions to molecules
within the same mesh voxel. In \cref{S:implementation} we develop
several methods for numerically evaluating the transition rates needed
to model reversible reactions in the CRDME model, considering both
general (smooth) interaction functions, and the discontinuous
indicator function that is used in the Doi model. Finally, in
\cref{S:numericExs} we consider a number of numerical examples to
demonstrate the convergence and accuracy of the CRDME in approximating
the volume-reactivity model, and to illustrate how the CRDME can be
used to study models for cellular processes within realistic domain
geometries arising from imaging data.

\section{Diffusion approximation on Unstructured
  Meshes} \label{S:diffApprox} We begin by deriving a lattice master
equation (equivalently continuous time random walk or jump process)
approximation to the Brownian motion of individual molecules using the
unstructured mesh method developed
in~\cite{LotstedtFERDME2009}. Spatial transition rates (i.e. hopping
rates) between lattice sites are obtained from a finite element
discretization of the Laplacian on triangulated meshes, giving rise to
a semi-discrete diffusion equation model with the form of a master
equation. In this section we summarize the method. Readers interested
in a more detailed discussion of this approach should
see~\cite{LotstedtFERDME2009}.

In the absence of chemical reactions, the Brownian motions of
individual molecules are independent processes. It is therefore
sufficient to derive a jump process (equivalently master equation)
approximation for a system in which there is only one diffusing
molecule.  Let $\Omega \subset \R^d$ be a bounded domain with boundary
$\partial \Omega$. We will denote by $p(\vx,t)$ the probability
density the molecule's position at time $t$ is $\vx$.  Assuming
reflecting boundary conditions on $\partial \Omega$, $p(\vx,t)$ then
satisfies the diffusion equation
\begin{equation}
  \begin{aligned}
    \PD{p}{t}(\vx,t) &= D\Delta p(\vx,t), &\forall \vx \in \Omega,\, t > 0 \\
    \nabla p(\vx,t) \cdot \veta(\vx) &= 0, &\forall \vx \in \partial\Omega, \, t > 0,
  \end{aligned}
  \label{eq:diffEqn}
\end{equation}
where $D$ denotes the molecule's diffusion constant (having units of
area per unit time), and $\veta(\vx)$ the outward normal at
$\vx \in \partial \Omega$.

\begin{figure}[tbp]
  \centering
  \includegraphics[scale=0.42]{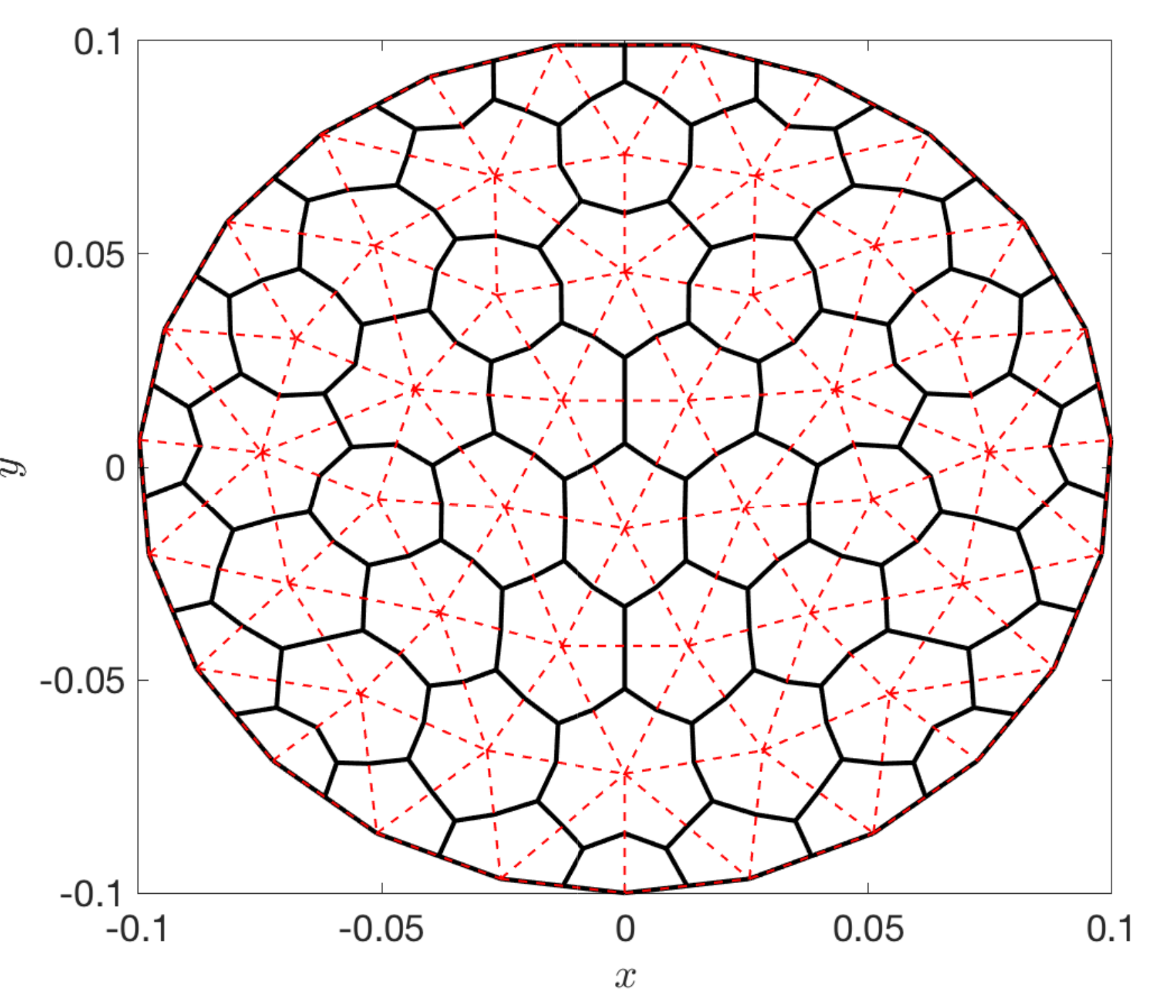}.
  \caption{An example of the dual mesh. The primal mesh is shown in
    dashed lines. The dual mesh is drawn in solid lines. Note, edges
    of triangles on the boundary are also within the primal mesh. }
  \label{fig:dualMeshEx}
\end{figure}
For simplicity, in discretizing~\cref{eq:diffEqn} to have the form of
a master equation we will focus on meshes in two--dimensions, $d = 2$.
Note, however, the final formulas we derive are also valid for $d=3$,
see~\cite{LotstedtFERDME2009}. We discretize $\Omega$ into a primal
mesh, given by a collection of triangles (red dashed lines in
\cref{fig:dualMeshEx}).  Let $\{\vx_i\}_{i = 1,\dots,K}$ label the
nodes of the mesh, corresponding to vertices of the triangles. We
define the dual mesh to consist of polygonal voxels
$\{V_{i}\}_{i = 1,\dots,K}$ in the interior of $\Omega$, with voxel
$V_{i}$ containing node $\vx_i$. Away from $\partial \Omega$, edges of
the polygonal voxel \commentb{containing} $\vx_i$ are given by lines connecting the
centroid of each triangle for which $\vx_i$ is a vertex to the
midpoint of each of that triangle's two edges that contain $\vx_i$
(black lines in \cref{fig:dualMeshEx}). For vertices that lie on the
boundary of the primal mesh, the corresponding polygonal voxel also
includes lines along the boundary connecting the vertex to the
midpoint of each triangle edge containing the vertex.  In 1D, the
primal mesh is a set of intervals with vertices at two ends. The
corresponding dual mesh is also a set of intervals, but shifted with
respect to the primal mesh so that the center of the interval is given
by a vertex. In 3D, the primal mesh is a set of tetrahedrons and the
corresponding dual mesh is a set of
polyhedrons~\cite{LotstedtFERDME2009}.

A standard finite element discretization of \cref{eq:diffEqn} on the
primal mesh using piecewise linear elements gives a linear system of
ODEs to solve for the set of nodal values,
$p_h(\vx_i,t)\approx p(\vx_i,t)$. Let
$\vp_h(t) = \brac{p_h(\vx_1,t),\dots,p_h(\vx_K,t)}^{T}$ denote the
vector of nodal values. The finite element discretization
of~\cref{eq:diffEqn} gives the semi-discrete diffusion equation
\begin{equation} \label{eq:semidisdiffEqn}
  M\D{\vp_h}{t}(t) = D S \, \vp_h(t),
\end{equation}
where $M$ denotes the mass matrix and $S$ the stiffness matrix. Under
suitable conditions on the mesh and domain this gives a second-order
discretization in space, with both matrices symmetric, $M$ positive
definite, and $S$ negative semi-definite.

The system \cref{eq:semidisdiffEqn} can be further simplified by
introducing mass lumping. $M$ is replaced by a diagonal lumped mass
matrix, $\Lambda$, where $\Lambda_{ii} = \sum_{j = 1}^KM_{ij}$. In
one-dimension, $\Lambda_{ii}$ gives the length of the dual mesh
element $V_i$. In two-dimensions, $\Lambda_{ii}$ gives the area of the
polygonal voxel $V_{i}$, while in three-dimensions, $\Lambda_{ii}$
gives the volume of polyhedral voxel
$V_{i}$~\cite{LotstedtFERDME2009}. Inverting the lumped mass matrix,
we obtain a simplified semidiscrete diffusion equation
\begin{equation}
  \D{\vp_h}{t}(t) = D \Delta_h \vp_h(t),
  \label{eq:lumpsemidisdiffEqn}
\end{equation}
where $\Delta_h$ can be interpreted as a discrete Laplacian,
\begin{equation}
  \Delta_h := \Lambda^{-1}S.
  \label{eq:disLapFEM}
\end{equation}

To arrive at a master equation approximation to the diffusion
equation, we define $P_{i}(t)$ to be the probability that the
diffusing molecule is in voxel $V_{i}$ at time $t$. 
We make the approximation that $P_{i}(t) = p_h(\vx_i,t) |V_{i}|$,
where $\abs{V_i}$ denotes the area of voxel $V_i$ (length in 1D or
volume in 3D). Let
$\vec{P}(t) = \brac{P_{1}(t),\dots,P_{K}(t)}^{T} = \Lambda \vp_h(t)$
denote the vector of probabilities to be in each voxel. As $S$ is
symmetric and $\Lambda^{-1}$ diagonal, \cref{eq:lumpsemidisdiffEqn}
then simplifies to the master equation
\begin{equation}
  \D{\vec{P}}{t} = D S \Lambda^{-1} \vec{P}(t) = D \Delta_h^{T} \vec{P}(t). 
  \label{eq:approxDiffEqn}
\end{equation}
For $i \neq j$, 
\begin{equation} \label{eq:diffHopRate}
  D (\Delta_h^{T})_{i j} = \frac{D S_{i j}}{\abs{V_j}} 
\end{equation}
then gives the probability per time, i.e. transition rate or hopping
rate, for a molecule in voxel $j$ to hop to voxel $i$, while
\begin{equation*}
  -D (\Delta_h^T)_{i i} = \frac{D}{\abs{V_{i}}} \sum_{j \neq i} S_{j i} 
\end{equation*}
gives the total probability per time for a molecule in $V_i$ to hop to
a neighboring voxel~\cite{LotstedtFERDME2009}. $D \Delta_h^T$ may then
be interpreted as the transition rate matrix for a continuous-time
random walk by the molecule between voxels of the lattice.

It is important to note that using an unstructured mesh to approximate
complex geometries can lead to negative values for
$(\Delta_h^T)_{ij}$, and hence negative transition rates, when using
piecewise linear finite elements. In~\cite{LotstedtFERDME2009} this
problem is resolved by modifying the transition matrix $\Delta_h^T$
when $(\Delta_h^T)_{ij} < 0$ so that $(\Delta_h^T)_{ij} = 0$ and
$(\Delta_h^T)_{ii} = -\sum_{j \neq i}(\Delta_h^T)_{ij}$.  Recently,
more accurate methods for eliminating negative discretization weights
were developed in~\cite{MeineckeEngblom2016}. For primal meshes given
by Delaunay triangulations in 2D, the transition rates
$(\Delta_h^T)_{i j}$ between voxels of the dual mesh defined by
\cref{fig:dualMeshEx} can be shown to always be non-negative, see
the discussion and references in~\cite{LotstedtFERDME2009}. For
simplicity, all examples subsequently considered in this work use
meshes that correspond to Delaunay triangulations in 2D to avoid this
potential complication.
\begin{remark}
  While in the remainder we shall assume that diffusive hopping rates
  are between elements of the dual mesh and given
  by~\cref{eq:diffHopRate}, there are a number of alternative methods
  one could use for determining spatial hopping rates in general
  geometries. These include the finite volume Cartesian grid cut-cell
  method of~\cite{IsaacsonSJSC2006} and the unstructured grid finite
  volume approach of~\cite{DeSchutterSTEPS2012uh}. The method we
  describe for approximating reversible bimolecular reactions in the
  next section can be used without modification with any of these
  alternative methods for determining spatial hopping rates.
\end{remark}

\section{Reversible Reactions on Unstructured
  Meshes} \label{S:methodSect} Having established how we will
approximate the spatial movement of molecules by a continuous-time
random walk (i.e. master equation), we now focus on developing a
convergent jump process (i.e. convergent master equation)
approximation for reversible bimolecular reactions. We begin by
formulating the general spatially-continuous volume-reactivity model
in \cref{S:absVolReactModel}, and the specific instance of this
model popularized by Doi~\cite{DoiSecondQuantA,DoiSecondQuantB}
in \cref{S:volreactDoiModel}. In \cref{S:rxDiscretMethod} we use
a finite volume discretization method to develop a master equation
approximation to the reaction terms on general unstructured polygonal
meshes.  The resulting discretization weights correspond to transition
rates for reactions to occur between two molecules within voxels of
the mesh. The finite volume discretization we present extends the
method we developed for Cartesian grids in~\cite{IsaacsonCRDME2013} to
reversible reactions on unstructured polygonal grids. In
\cref{S:unstructCRDME} we combine the spatial discretization
method from \cref{S:diffApprox} with the reaction
discretization method developed in \cref{S:rxDiscretMethod} to
derive the convergent reaction-diffusion master equation (CRDME) jump
process approximation to the volume--reactivity model. For simplicity,
we derive the CRDME for the reversible bimolecular reaction
$A + B \rightleftharpoons C$ in a system whose state is either one
molecule of \textrm{A} and one molecule of \textrm{B}, or one molecule
of the complex \textrm{C}. In \cref{ap:multipartModel}, we
show that knowing transition rates for this simplified model is
sufficient to fully determine a corresponding set of transition rates
for systems with arbitrary numbers of molecules.

\Cref{S:crdmeTorRDME} discusses the relationship between the CRDME and
the popular lattice reaction-diffusion master equation (RDME) model,
which can be interpreted as a coarse-mesh approximation to the
CRDME. Finally, in \cref{S:implementation} we discuss several
implementation details, including the numerical method we use to
evaluate the reactive transition rate formulas derived in
\cref{S:rxDiscretMethod} for reversible reactions in the CRDME.

\subsection{Abstract Volume--Reactivity  Model} \label{S:absVolReactModel} 
Consider the $\textrm{A} + \textrm{B} \leftrightarrows \textrm{C}$
reaction in a system with just one \textrm{A} and one \textrm{B}
molecule, or one \textrm{C} complex.  As in \cref{S:diffApprox}, we
assume the reaction is taking place within a $d$-dimensional bounded
domain $\Omega \subset \R^d$, with molecules experiencing a reflecting
Neumann boundary condition on the boundary $\partial \Omega$. Denote
by $\vx \in \Omega$ the position of the molecule of species
\textrm{A}, by $\vy \in \Omega$ the position of the molecule of
species \textrm{B}, and by $\vz \in \Omega$ the position of the
molecule of species \textrm{C}. The diffusion constants of the
molecules are given by $\DA$, $\DB$, and $\DC$ respectively. We let
$\kp(\vz\vert\vx,\vy)$ denote the probability density per unit time an
$A + B \rightarrow C$ reaction successfully occurs producing a
\textrm{C} molecule at $\vz$, given an \textrm{A} molecule at $\vx$
and a \textrm{B} molecule at $\vy$. With this definition, we denote by
$\kp(\vx,\vy)$ the probability per time an \textrm{A} molecule at
$\vx$ and a \textrm{B} molecule at $\vy$ react and create a \textrm{C}
molecule within $\Omega$. Then
\begin{equation} \label{eq:kpxyDef}
  \kp(\vx,\vy) := \int_{\Omega} \kp(\vz\vert\vx,\vy) \, d\vz.
\end{equation}

Similarly, we let $\km(\vx,\vy\vert\vz)$ denote the probability
density per unit time a reaction successfully occurs producing an
\textrm{A} molecule at $\vx$ and a \textrm{B} molecule at $\vy$ given
a \textrm{C} molecule at $\vz$. With this definition, we denote by
$\km(\vz)$ the probability per time a \textrm{C} molecule at $\vz$
unbinds and produces \textrm{A} and \textrm{B} molecules within
$\Omega$. Then
\begin{equation*}
  \km(\vz) := \int_{\Omega^2} \km(\vx,\vy\vert\vz) \, d\vx \, d\vy,
\end{equation*}
where $\Omega^2 := \Omega \times \Omega \subset \R^{2d}$.

Using the preceding definitions we can now formulate the general
volume-reactivity model. Let $p(\vx,\vy,t)$ denote the probability
density that system is in the state where the \textrm{A} and
\textrm{B} molecules are unbound, and located at positions $\vx$ and
$\vy$ at time $t$. Likewise, define $\pb(\vz,t)$ to be the probability
density that the molecules are bound together, and that the
corresponding \textrm{C} molecule is at position $\vz$ at time
$t$. Then $p(\vx,\vy,t)$ satisfies
\begin{equation}
  \label{eq:pPDE}
  \begin{aligned}
    \PD{p}{t}(\vx, \vy, t) = (\DA\Delta_{\vx} + \DB\Delta_{\vy})p(\vx,\vy,t) 
    &- \kp(\vx,\vy)p(\vx,\vy,t) 
    + \int_{\Omega}\km(\vx,\vy\vert\vz)\pb(\vz,t) \, d\vz,
  \end{aligned}
\end{equation}
and $\pb(\vz,t)$ satisfies
\begin{equation}
  \PD{\pb}{t}(\vz,t) = \DC\Delta_{\vz}\pb(\vz,t) - \km(\vz)\pb(\vz,t) 
  + \int_{\Omega^2}\kp(\vz\vert\vx,\vy)p(\vx,\vy,t) \, d\vx \, d\vy,
  \label{eq:pbPDE}
\end{equation}
together with appropriate initial conditions and no-flux reflecting
boundary conditions on $\partial \Omega$ for each molecule.

In the next section we present a special case of the preceding model,
where $\kp(\vz|\vx,\vy)$ is chosen such that two molecules which are
sufficiently close will react with a fixed probability per unit
time~\cite{TeramotoDoiModel1967,DoiSecondQuantA,DoiSecondQuantB,ErbanChapman2009}.
This ``standard'' form of the volume--reactivity model is also known
as the Doi model.  Note, the abstract formulation presented above also
encompasses alternative bimolecular reaction models, such as the
Gaussian--type interactions we used
in~\cite{IsaacsonGoyetteDuskek2016} to model reactions involving
unstructured tails of membrane bound proteins.

\subsection{Standard Volume--Reactivity Reaction Rate Functions (Doi Model)}
\label{S:volreactDoiModel}
In the standard volume--reactivity (Doi) model, \textrm{A} and
\textrm{B} are assumed to react with probability per unit time
$\lambda$ when within a distance $\rb$, commonly called the reaction
radius.  Let
$\Rset = \{ (\vx,\vy) \in \Omega^2 \,\vert\, \abs{\vx-\vy} < \rb\}
\subset \R^{2d}$
denote the effective reaction region, and $\ind_{\Rset}(\vx,\vy)$
denote the indicator function of $\Rset$. A common choice for
$\kp(\vz\vert\vx,\vy)$ is
\begin{align} 
  \kp(\vz\vert\vx,\vy) &= \lambda \ind_{\Rset}(\vx, \vy) \delta
                         \paren{\vz - (\gamma\vx + (1-\gamma)\vy)}, \label{eq:Doikp}
\end{align}
where $\gamma \in [0, 1]$ determines the placement of the newly
created \textrm{C} molecule relative to the locations of the
\textrm{A} and \textrm{B} molecules.  One simple choice is to take
$\gamma = \frac{1}{2}$, so that the \textrm{C} molecule is placed at
the midpoint between the \textrm{A} and \textrm{B} molecules. Another
common choice is to use the diffusion weighted center of
mass~\cite{AndrewsBrayPhysBio2004}
\begin{equation} \label{eq:diffWeightCentMass}
  \gamma = \frac{\DB}{\DA + \DB}.
\end{equation}
For $\gamma$ fixed, the probability per time that an \textrm{A} molecule
at $\vx \in \Omega$ and a \textrm{B} molecule at $\vy \in \Omega$ successfully
react is then
\begin{equation} \label{eq:Doikpxy}
  \kp(\vx,\vy) = \lambda \ind_{\Rset}(\vx, \vy) \ind_{\Omega}\paren{\gamma \vx + (1-\gamma) \vy)}.
\end{equation}
Here the second indicator function enforces that the reaction can only
occur if the location the product \textrm{C} molecule would be placed
at is within $\Omega$. If $\Omega$ is convex this is guaranteed. If
$\Omega$ is not convex, this association reaction model can be
interpreted as a two-step process; the molecules attempt to react with
probability per unit time $\lambda$ when within $\rb$, and if the
product location is within the domain the reaction is allowed to
proceed. If the product location is outside the domain the reaction
event is rejected.

The dissociation of the \textrm{C} molecule back into \textrm{A} and
\textrm{B} molecules is assumed to occur with probability per unit
time $\mu$.  Several different models have been used to specify the
placement of newly created \textrm{A} and \textrm{B} molecules when
dissociation occurs. The simplest choice would be to place them at the
location of the \textrm{C} molecule at the time of
unbinding~\cite{ErbanChapman2011}, which we call point unbinding. In
this case
\begin{align}
  \km(\vx,\vy\vert\vz) &= \mu \delta(\vx - \vy) \delta(\vy - \vz).
  \label{eq:ptUnbind}
\end{align}
The probability per time that a \textrm{C} molecule at
$\vz \in \Omega$ successfully dissociates is then
\begin{align*}
\km(\vz) = \mu \int_{\Omega^2} \delta(\vx - \vy) \delta(\vy - \vz) \, d\vx \, d\vy = \mu.
\end{align*}

In the remainder we focus on what we call the uniform unbinding
model. Here the position of the \textrm{A} molecule is sampled from a
uniform distribution within the ball of radius $(1-\gamma)\rb$ about
the position of the \textrm{C} molecule,
$B_{(1-\gamma)\rb}(\vz) := \{ \vx \in \R^d \,\vert \abs{\vz-\vx} <
(1-\gamma)\rb\}$.
The position of the \textrm{B} molecule is then chosen by reflection
so that $\vz = \gamma \vx + (1-\gamma) \vy$. We then have
\begin{align}
  \km(\vx,\vy\vert\vz) &= \mu \brac{\frac{1}{\abs{\ballunbind(\vO)}}\ind_{\ballunbind(\vz)}(\vx)}
                         \delta\paren{\vy - \frac{\vz - \gamma\vx}{1-\gamma}}. \label{eq:Doikm}
\end{align}
One complication with this choice is that when $\vz$ is sufficiently
close to $\partial \Omega$, the position of one or both of the
\textrm{A} and \textrm{B} molecules may end up outside the domain.  In
this case a natural choice that is consistent with the preceding
definitions is to simply reject the dissociation
event~\cite{IsaacsonZhang17}. $\km(\vz)$ is therefore
reduced relative to the point unbinding case,
\begin{align*}
  \km(\vz) = \frac{\mu}{\abs{\ballunbind(\vO)}} \int_{\Omega} \ind_{\ballunbind(\vz)}(\vx)
  \ind_{\Omega} \paren{\frac{\vz - \gamma\vx}{1-\gamma}} \, d\vx \leq \mu.
\end{align*}
For points $\vz$ that are sufficiently far from the boundary, or if
$\Omega = \R^d$, this simplifies to $\km(\vz) = \mu$ and unbinding
events are always successful. The combination of the standard (Doi)
association model with rejection of unbinding events that produce
molecules outside the domain can be shown to imply point-wise
detailed balance of the resulting reversible binding
reaction~\cite{IsaacsonZhang17}.

\begin{remark}
  With the choices~\cref{eq:Doikp}, \cref{eq:diffWeightCentMass}
  and~\cref{eq:Doikm}, $\gamma = 0$ indicates that the \textrm{B}
  molecule is not diffusing. Upon binding, the \textrm{C} molecule is
  therefore placed at $\vy$. On the other hand, $\gamma = 1$ indicates
  that the \textrm{A} molecule is not diffusing, and one needs to
  interchange $\vx$ and $\vy$ in~\cref{eq:Doikm}. Such choices would
  be appropriate if one of the \textrm{A} or \textrm{B} molecules
  represents a stationary target.
\end{remark}

\subsection{Discretization of Reaction Terms to Master Equation} 
\label{S:rxDiscretMethod}
We now develop a master equation approximation to the reaction terms
of the general volume--reactivity model on polygonal unstructured
meshes. This is achieved by developing a finite volume discretization
of~\cref{eq:pPDE} and~\cref{eq:pbPDE} that has the general form of a
master equation for a jump process.  We discretize $\Omega$ into a
polygonal mesh of $K$ voxels labeled by $V_{i}$,
$i \in \{1,\dots,K\}$, with corresponding centroids
$\{\vx_i\}_{i=1,\dots,K}$.  As we will often need to consider the
phase-space voxels that pairs or triplets of molecules are located
within, we let $V_{ij} = V_{i} \times V_{j}$ and
$V_{ijk} = V_{i} \times V_{j} \times V_{k}$, with corresponding
centroids labeled by $(\vx_{i}, \vy_{j})$ and
$(\vx_{i}, \vy_{j}, \vz_k)$. With these definitions we make the
well-mixed approximation that the probability densities,
$p(\vx,\vy,t)$ and $\pb(\vz,t)$, are piecewise constant within each
mesh voxel, $V_{i j}$ and $V_k$ respectively.  The probability the
system is in the unbound state with the \textrm{A} molecule in $V_{i}$
and the \textrm{B} molecule in $V_{j}$ at time $t$ is then
approximated by
\begin{equation} \label{eq:PijDef}
  P_{ij}(t) = \int_{V_{ij}} p(\vx,\vy,t) \, d\vx \, d\vy
  \approx p(\vx_{i}, \vy_{j},t) \abs{V_{ij}}.
\end{equation}
Similarly, we denote by $\vz_{k}$ the centroid of $V_{k}$. The
probability density the system is in the bound state with the
\textrm{C} molecule in $V_{k}$ at time $t$ is then approximated by
\begin{equation} \label{eq:PbkDef}
  P_{\textrm{b}k}(t) = \int_{V_{k}} \pb(\vz,t) \, d\vz \approx \pb(\vz_{k},t) \abs{V_{k}}.
\end{equation}

In what follows we drop the diffusive terms in~\cref{eq:pPDE}
and~\cref{eq:pbPDE} as we will ultimately approximate them through
the finite element method of \cref{S:diffApprox}. In the next
section we illustrate the final combined model with both spatial
(diffusive) transport and chemical reactions. With this
simplification, we construct a finite volume discretization
of~\cref{eq:pPDE} and~\cref{eq:pbPDE} by integrating both sides
of~\cref{eq:pPDE} and~\cref{eq:pbPDE} over $V_{ij}$ and $V_{k}$
respectively. \cref{eq:pPDE} is approximated by
\begin{equation}   \label{eq:CRDMEfwdRxs}
  \begin{aligned}
    \D{P_{ij}}{t} &= 
    -\int_{V_{ij}}\kp(\vx,\vy)p(\vx,\vy,t) \, d\vx \, d\vy 
    + \int_{\Omega}\brac{\int_{V_{ij}}\km(\vx,\vy\vert\vz) \, d\vx d\vy}\pb(\vz,t) \, d\vz \\
    &\approx -\frac{1}{\abs{V_{ij}}}P_{ij}(t)\int_{V_{ij}}\kp(\vx,\vy) \, d\vx \, d\vy 
    + \sum_{k}\frac{1}{\abs{V_{k}}} 
    P_{\textrm{b}k}(t)\int_{V_{ijk}}\km(\vx,\vy\vert\vz) \, d\vx \, d\vy \, d\vz \\
    &= -\kpij P_{ij}(t) + \sum_{k} \kmijk P_{\textrm{b}k}(t),
  \end{aligned}
\end{equation}
where
\begin{align}
  \kpij &:= \frac{1}{\abs{V_{ij}}} \int_{V_{ij}}\kp(\vx,\vy) \, d\vx \, d\vy \label{eq:kpijDef} \\ 
  \kmijk &:= \frac{1}{\abs{V_{k}}} \int_{V_{ijk}}\km(\vx,\vy\vert\vz) \, d\vx \, d\vy \, d\vz. \label{eq:kmijkDef}
\end{align}
One can interpret $\kpij$ as the probability per unit time that given
an \textrm{A} molecule in $V_{i}$ and a \textrm{B} molecule in
$V_{j}$, they react to produce a \textrm{C} molecule in
$\Omega$. Similarly, $\kmijk$ gives the probability per unit time that
given a \textrm{C} molecule in $V_{k}$, it dissociates into an
\textrm{A} molecule in $V_{i}$ and a \textrm{B} molecule in $V_{j}$.

The reaction terms of~\cref{eq:pbPDE} are approximated by
\begin{equation}   \label{eq:CRDMEbwdRxs}
  \begin{aligned}
    \D{P_{\textrm{b}k}}{t} &= 
    -\int_{V_{k}}\km(\vz)\pb(\vz,t) \, d\vz 
    + \int_{\Omega^2}\brac{\int_{V_{k}}\kp(\vz\vert\vx,\vy) \, d\vz}p(\vx,\vy,t) \, d\vx \, d\vy\\
    &\approx -\frac{1}{\abs{V_{k}}}P_{\textrm{b}k}(t) \int_{V_{k}}\km(\vz) \, d\vz 
    + \sum_{i, j}\frac{1}{\abs{V_{ij}}}P_{ij}(t) 
    \int_{V_{ijk}} \kp(\vz \vert \vx,\vy) \, d\vx \, d\vy \, d\vz \\
    &= -\kmk P_{\textrm{b}k}(t) + \sum_{i, j} \kpijk P_{ij}(t),
  \end{aligned}
\end{equation}
where
\begin{align}
  \kmk &:= \frac{1}{\abs{V_{k}}} \int_{V_{k}}\km(\vz) \, d\vz \label{eq:kmkDef} \\
  \kpijk &:= \frac{1}{\abs{V_{ij}}} \int_{V_{ijk}}\kp(\vz \vert \vx,\vy) \, d\vx \, d\vy \, d\vz. \label{eq:kpijkDef}
\end{align}
One can interpret $\kmk$ as the probability per unit time that given a
\textrm{C} molecule in $V_{k}$, it dissociates into \textrm{A} and
\textrm{B} molecules within $\Omega$.  Similarly, $\kpijk$ gives the
probability per unit time that given an \textrm{A} molecule in
$V_{i}$ and a \textrm{B} molecule in $V_{j}$, they react to
produce a \textrm{C} molecule in $V_{k}$.

Using the definitions of $\kp(\vx,\vy)$ and $\km(\vz)$, we have 
\begin{equation*}
  \int_{V_{ij}}\kp(\vx,\vy) \, d\vx \, d\vy 
  = \int_{V_{ij}}\brac{\int_{\Omega}\kp(\vz\vert\vx, \vy) \, d\vz}  d\vx \, d\vy 
  = \sum_{k}\kpijk \abs{V_{i j}},
\end{equation*}
and
\begin{equation*}
  \int_{V_{k}}\km(\vz) \, d\vz 
  = \int_{V_{k}}\brac{\int_{\Omega^2}\km(\vx, \vy\vert\vz) \, d\vx \, d\vy} \, d\vz 
  = \sum_{i, j}\kmijk \abs{V_{k}},
\end{equation*}
which gives that 
\begin{subequations} \label{eq:kpkmsums}
  \begin{align}
    \kpij &= \sum_{k}\kpijk \label{eq:kpkmsumsKpij}\\
    \kmk &= \sum_{i, j}\kmijk. \label{eq:kpkmsumsKmk}
  \end{align}
\end{subequations}
With these identities, we can identify the probability of placing a newly
created \textrm{C} molecule in 
$V_{k}$ given that an \textrm{A} molecule in $V_{i}$ and a \textrm{B} molecule in
$V_{j}$ react,
\begin{equation} \label{eq:kpkijDef}
  \kpkij := \frac{\kpijk}{\kpij}.
\end{equation}
Similarly, the probability of placing a newly created \textrm{A}
molecule in $V_{i}$ and a \textrm{B} molecule in $V_{j}$ given that
a \textrm{C} molecule in $V_{k}$ dissociated is
\begin{equation} \label{eq:kmkijDef}
  \kmkij := \frac{\kmijk}{\kmk}.
\end{equation}

The semi-discrete equations~\cref{eq:CRDMEfwdRxs}
and~\cref{eq:CRDMEbwdRxs} have the form of a master equation
(i.e. forward Kolmogorov equation) for a jump process corresponding to
the positions of the molecules and the current chemical state of the
system (unbound or bound). The transition rates (i.e. propensities)
$\{\kpij, \kmk, \kpijk, \kmijk\}$ and placement probabilities
$\{\kpkij, \kmkij\}$ then allow for the simulation of this process
using any of the many stochastic simulation algorithm (SSA)-based
methods, for
example~\cite{GillespieJPCHEM1977,GibsonBruckJPCHEM2002}. Note, there
are two statistically equivalent approaches one can take to use the
reversible binding model we've derived within simulations. In the
first approach one separates the association and dissociation
reactions from the placement of reaction products:
\begin{enumerate}
\item Given one \textrm{A} molecule in $V_{i}$ and one \textrm{B}
  molecule in $V_{j}$, the reaction $A_{i} + B_{j} \rightarrow C$
  occurs with transition rate $\kpij$.  Similarly, given one
  \textrm{C} molecule in $V_{k}$, the reaction
  $C_{k} \rightarrow A + B$ occurs with transition rate $\kmk$.
\item If $A_{i}$ and $B_{j}$ molecules react, place a \textrm{C}
  molecule in $V_{k}$ with probability $\kpkij$.  If a $C_{k}$
  molecule dissociates apart, place an \textrm{A} molecule in $V_{i}$
  and a \textrm{B} molecule in $V_{j}$ with probability $\kmkij$.
\end{enumerate}
Equivalently, the second approach expands the set of reactions to
include product placement within the transition rates:
\begin{enumerate}
\item Given one \textrm{A} molecule in $V_{i}$ and one \textrm{B}
  molecule in $V_{j}$, the reaction $A_{i} + B_{j} \rightarrow C_{k}$
  occurs with transition rate $\kpijk$.  Similarly, given one
  \textrm{C} molecule in $V_{k}$, the reaction
  $C_{k} \rightarrow A_{i} + B_{j}$ occurs with transition rate
  $\kmijk$.
\end{enumerate}
The first approach requires two sampling steps: selection of the
reaction to execute, and then placement of newly created molecules. In
contrast, the second approach requires only one sampling step but has
many more possible reactions. In the remainder, we use the first
algorithm for all reported simulations.

\subsection{Unstructured Mesh CRDME for Reversible Reactions}
\label{S:unstructCRDME}
To arrive at a final unstructured mesh CRDME for simulating the
reversible $\textrm{A} + \textrm{B} \leftrightarrows \textrm{C}$
reaction, we combine the finite element discretization for spatial
(diffusive) transport from \cref{S:diffApprox} with the finite volume
discretization of the reversible binding process developed in the
previous section. \emph{Both} discretizations are constructed on the
(dual) polygonal mesh of a triangulated primal mesh, see the
discussion in \cref{S:diffApprox}. Applying the finite element
discretization~\cref{eq:semidisdiffEqn} to each Laplacian
in~\cref{eq:pPDE} and~\cref{eq:pbPDE}, and using the reaction term
discretizations~\cref{eq:CRDMEfwdRxs} and~\cref{eq:CRDMEbwdRxs}, we
obtain the final master equation model
\begin{subequations} \label{eq:twopartCRDME}
  \begin{equation}     \label{eq:twopartCRDMEfwd}     
    \begin{alignedat}{2}
      \D{P_{ij}}{t} &= D^{\textrm{A}} \sum_{i'=1}^{K} \brac{(\Delta_h^T)_{i i'} P_{i'j}(t) 
        - (\Delta_h^T)_{i' i} P_{ij}(t)} 
      + D^{\textrm{B}} &\sum_{j'=1}^{K} \brac{ (\Delta_h^T)_{j j'} P_{ij'}(t) - (\Delta_h^T)_{j' j} P_{ij}(t)} \\
      &&-\kpij P_{ij}(t) + \sum_{k=1}^{K} \kmijk P_{\textrm{b}k}(t),
    \end{alignedat}
  \end{equation}
  \begin{equation}   \label{eq:twopartCRDMEbwd}
    \D{P_{\textrm{b}k}}{t} = 
    D^{\textrm{C}} \sum_{k'=1}^{K} \brac{(\Delta_h^T)_{k k'} P_{\textrm{b}k'}(t) 
      - (\Delta_h^T)_{k' k} P_{\textrm{b}k}(t)}
    -\kmk P_{\textrm{b}k}(t)  + \sum_{i,j=1}^{K} \kpijk P_{ij}(t).
  \end{equation}
\end{subequations}
Here $P_{ij}(t)$ gives the probability for the \textrm{A} and
\textrm{B} molecules to be in $V_{ij}$ at time $t$, and
$P_{\textrm{b}k}(t)$ the probability for the \textrm{C} molecule to be
in voxel $V_k$ at time $t$, see~\cref{eq:PijDef}
and~\cref{eq:PbkDef}. We call~\cref{eq:twopartCRDME} the convergent
reaction-diffusion master equation (CRDME). 

\begin{table}[tbp]
  \centering  
  \setlength\tabcolsep{3.5pt}
  \bgroup
  \renewcommand{\arraystretch}{1.2}
  \begin{tabular}{|c|c|c|c|}
    \hline
    & Transitions & Transition Rates & Upon Transition Event\\
    \hline
    \multirow{3}{*}{\parbox{2cm}{Diffusive hopping:}} 
    & $\textrm{A}_{j} \to \textrm{A}_i$ & $D^{\textrm{A}} (\Delta_h^T)_{ij} \, a_j$ 
                                   & $\textrm{A}_i := \textrm{A}_i + 1$, $\textrm{A}_j := \textrm{A}_j - 1$, \\ \cline{2-4}
    & $\textrm{B}_{j} \to \textrm{B}_i$ & $D^{\textrm{B}} (\Delta_h^T)_{ij} \, b_j$ 
                                   & $\textrm{B}_i := \textrm{B}_i + 1$, $\textrm{B}_j := \textrm{B}_j - 1$, \\ \cline{2-4}
    & $\textrm{C}_{j} \to \textrm{C}_i$ & $D^{\textrm{C}} (\Delta_h^T)_{ij} \, c_j$ 
                                   & $\textrm{C}_i := \textrm{C}_i + 1$, $\textrm{C}_j := \textrm{C}_j - 1$, \\ \cline{2-4}
    \hline 
    \multirow{6}{*}{\parbox{2cm}{Chemical Reactions:}}
    &\multirow{3}{*}{$\textrm{A}_i + \textrm{B}_j \to \textrm{C}$} & \multirow{3}{*}{$\kpij a_i b_j$} 
                                   & $\textrm{A}_i := \textrm{A}_i - 1$, $\textrm{B}_j := \textrm{B}_j - 1$.\\ 
    & & & Sample  $k$ from $\{\kpkij\}_{k=1,\dots,K}$.\\
    & & & Set $\textrm{C}_k := \textrm{C}_k + 1$. \\ \cline{2-4}
    &\multirow{3}{*}{$\textrm{C}_k \to \textrm{A} + \textrm{B}$} & \multirow{3}{*}{$\kmk c_k$} 
    & $\textrm{C}_k := \textrm{C}_k - 1$. \\
    & & & Sample $(i,j)$ from $\{\kmkij\}_{i,j=1,\dots,K}$. \\
    & & & Set $\textrm{A}_i := \textrm{A}_i + 1$, $\textrm{B}_j := \textrm{B}_j + 1$.\\
    \hline
  \end{tabular}
  \egroup
  \caption{ 
    Summary of diffusive and chemical transitions for the jump process approximation of
    the general multi-particle $\textrm{A} + \textrm{B} \leftrightarrows \textrm{C}$ reaction. 
    The statistics of this process are given by the corresponding forward Kolmogorov equation for the
    probability distribution, the multiparticle CRDME~\cref{eq:vPEq}. Here
    $a_i$ denotes the number of \textrm{A} molecules in voxel $V_i$, with $b_j$ and $c_k$ defined
    similarly. Transition rates give the probability per time for a transition to occur, often called
    propensities in the chemical kinetics literature.
    The final column explains how to update the system state upon occurrence of a transition event.
  }
  \label{tab:aAndBToCRxs}
\end{table}
While~\cref{eq:twopartCRDME} is specialized to a system containing
one \textrm{A} and one \textrm{B} molecule, or one \textrm{C} molecule
when the two are bound, it is straightforward to generalize the
equation to systems that include arbitrary numbers of each species. In
\cref{ap:multipartModel} we develop the corresponding
continuous particle dynamics equations for such systems, generalizing
the two molecule system given by~\cref{eq:pPDE}
and~\cref{eq:pbPDE}. The structure of the resulting
equation~\cref{eq:multipartABtoCEqs} includes only two-body
interactions, allowing the discretization method we used to
derive~\cref{eq:twopartCRDME} to be applied
to~\cref{eq:multipartABtoCEqs} to derive a general CRDME for systems
with arbitrary numbers of molecules~\cref{eq:vPEq}. The resulting set
of diffusive and chemical reactions, along with associated transition
rates (i.e. propensities), are summarized in
\cref{tab:aAndBToCRxs}. Notice, the only difference between the
general multi-particle system and the two-particle system is that the
transition rates are multiplied by the number of possible ways a given
transition can occur. Let $a_i$ denote the number of molecules of
species \textrm{A} in $V_i$, with $b_j$ and $c_k$ defined
similarly. For the forward reaction there are $a_i b_j$ possible pairs
of species \textrm{A} molecules in $V_i$ and species \textrm{B}
molecules in $V_j$ that can react. The new transition rate for the
$\textrm{A}_i + \textrm{B}_j \to \textrm{C}$ reaction is therefore
$\kp_{ij} a_i b_j$. Similarly there are $c_k$ possible dissociation
reactions for species \textrm{C} molecules in voxel $V_k$, giving a
new transition rate of $\km_k c_k$. Likewise, there are $a_j$ possible
hopping transitions of a molecule of species \textrm{A} from voxel
$V_j$ to $V_i$, giving a new diffusive transition rate of
$D^{\textrm{A}} (\Delta_h^T)_{i j} a_j$.

The set of transitions in \cref{tab:aAndBToCRxs} collectively
define a vector jump process for the number of molecules of each
species and their locations on the mesh. Let $A_i(t)$ represent the
stochastic process for the number of molecules of species \textrm{A}
in voxel $V_i$ at time $t$, and define $B_j(t)$ and $C_k(t)$
similarly. We denote by
\begin{equation*}
 \vec{W}(t) = \paren{A_1(t),\dots,A_K(t), B_1(t),\dots, B_K(t), C_1(t),\dots, C_K(t)} 
\end{equation*}
the stochastic process for the total system state at time $t$, and by
$\vec{w}$ a value of $\vec{W}(t)$, i.e. $\vec{W}(t)=\vec{w}$. The
master equation~\cref{eq:vPEq} then gives the probability that
$\vec{W}(t) = \vec{w}$. Implicit equations for the stochastic
processes that are components of $\vec{W}(t)$ can also be written,
which are equivalent in distribution to the master
equation~\cite{AndersonKurtzReview2011,AndersonKurtzBook2015}.

The coupled system of ODEs that correspond to the master
equation~\cref{eq:vPEq} for the
$\textrm{A} + \textrm{B} \leftrightarrows \textrm{C}$ reaction with
arbitrary numbers of molecules is too high-dimensional to solve
directly. Instead, the well-known Stochastic Simulation Algorithm
(SSA), also known as the Gillespie method or Kinetic Monte Carlo
method, and its many variants can be used to generate exact
realizations of
$\vec{W}(t)$~\cite{GillespieJPCHEM1977,KalosKMC75,GibsonBruckJPCHEM2002}.
For all numerical examples we subsequently consider in
\cref{S:numericExs} we use this method to directly simulate the
hopping of molecules between voxels and their chemical interactions.

  
\subsection{Relation to Reaction-Diffusion Master Equation (RDME)}
\label{S:crdmeTorRDME}
The CRDME is quite similar to the corresponding reaction-diffusion
master equation (RDME)
model~\cite{GardinerHANDBOOKSTOCH,GardinerRXDIFFME}, which also
satisfies~\cref{eq:twopartCRDME} but with the redefined transition
rates 
\commenta{
\begin{align*}
  \kpijk &= \frac{\beta^{+} \delta_{ij} \delta_{ik}}{\abs{V_i}}, & 
  \kmijk &= \beta^{-} \delta_{ik} \delta_{jk}.
\end{align*}
}
In the RDME \commenta{$\beta^{+}$} is usually taken to be the well-mixed rate
constant for the $\textrm{A} + \textrm{B} \to \textrm{C}$ reaction
(with units of volume per time in three-dimensions), \commenta{$\beta^{-}$} is
the dissociation rate (with units of inverse time), and $\delta_{ij}$
is the Kronecker delta function~\cite{IsaacsonRDMENote}. With these
choices, using~\cref{eq:kpkmsums} we find that
\commenta{
\begin{align*}
  \kpij &= \frac{\beta^{+} \delta_{ij}}{\abs{V_i}}, & 
  \kmk &= \beta^{-}.
\end{align*}
}
As such, the primary difference between the CRDME and the RDME is that
the latter only allows for chemical reactions between molecules within
the same voxel, while the former allows for reactions between
molecules located in nearby voxels.

It has been shown that this difference has the unfortunate drawback of
causing the RDME to lose bimolecular reactions in the continuum limit
that the voxel sizes approach zero (in two or more
dimensions)~\cite{IsaacsonRDMELims,IsaacsonRDMELimsII,Hellander:2012jk}. The
RDME is therefore not a convergent approximation to any reasonable
continuous particle dynamics model, though in practice it may give a
good approximation for voxel sizes that are neither too large nor
small~\cite{IsaacsonRDMELims,IsaacsonRDMELimsII,Hellander:2012jk}. As
explained in the introduction, the loss of bimolecular reactions is
due to the representation of molecules as point particles, and their
restriction to only react when in the same voxel. As the voxel size
approaches zero we recover a system of point particles moving by
Brownian motion, for which bimolecular reactions are only possible
when one \textrm{A} molecule and one \textrm{B} molecule are located
at the same point. The probability of the latter event occurring is
zero, so that the molecules simply never find each other to
react~\cite{Hellander:2012jk,IsaacsonRDMELims}.

While the RDME loses bimolecular reactions in the continuum limit that
the voxel size approaches zero, we demonstrated for the irreversible
$\textrm{A} + \textrm{B} \to \textrm{C}$ reaction on Cartesian meshes
that as the voxel size is \emph{coarsened} the RDME approaches the
corresponding CRDME for the standard volume-reactivity (Doi)
model~\cite{IsaacsonCRDME2013}. As such, we may interpret the RDME as
an approximation to the CRDME that is only valid for sufficiently
large voxel sizes (roughly corresponding to voxel sizes that are
significantly larger then the reaction radii for any bimolecular
reactions).

\subsection{Numerical Evaluation of Transition Rates}
\label{S:implementation}
To use the SSA to generate realizations of the jump process
corresponding to the CRDME~\cref{eq:twopartCRDME}, or its
multiparticle generalization~\cref{eq:vPEq}, requires the numerical
evaluation of the diffusive and reactive transition rates. The former
require the calculation of the matrix with entries
$(\Delta_h^T)_{i j} = (S \Lambda^{-1})_{i j}$. For all simulations
reported in this work we used the MATLAB linear finite element
implementation of~\cite{Alberty1999ff} to calculate the stiffness
($S$) and mass ($M$) matrices, from which the matrix $\Delta_h^T$
is then easily calculated.

The transitions we use to model chemical reactions, see
\cref{tab:aAndBToCRxs}, require the transition rates $\kpij$ and
$\kmk$, along with the reaction probabilities $\kpkij$ and
$\kmkij$. When $\kp(\vx,\vy)$ is a sufficiently smooth function, we
found that $\kpij$ could be easily evaluated by nesting MATLAB's
built-in two-dimensional numerical integration
routine~\texttt{integral2}
to evaluate the four-dimensional integral~\cref{eq:kpijDef}. While
the Doi model, in which $\kp(\vx,\vy)$ is discontinuous~\cref{eq:Doikp}, 
is perhaps the most popular volume-reactivity model, smooth interactions
do arise in applications. For example, in~\cite{IsaacsonGoyetteDuskek2016} 
we used the Gaussian interaction 
\begin{equation}
  \kp(\vx, \vy) = \lambda
  \paren{\frac{3}{2\pi}}^{3/2}\frac{1}{\rb^3}e^{-\frac{3\abs{\vx-\vy}^2}{2\rb^2}}
  \label{eq:GaussInteract}
\end{equation}
to model bimolecular reaction rates between membrane-bounded tethered
signaling molecules with unstructured tails. (Here $\lambda$
corresponds to a catalytic rate constant, with units of volume per
time, and $\rb$ defines the width of the Gaussian interaction.)

In the Doi model variant of the volume-reactivity model $\kp(\vx,\vy)$
is given by~\cref{eq:Doikpxy}.
The corresponding association reaction transition rate in the CRDME is
then
\begin{equation*}
  \kpij = \frac{\lambda}{\abs{V_{i j}}} \int_{\Rset \cap V_{i j}} 
  \ind_{\Omega}(\gamma \vx + (1-\gamma) \vy) \, d\vx \, d \vy. 
\end{equation*}
In the special case that $\Omega$ is convex, $\kpij$ simplifies to
\begin{equation} \label{eq:kpijConvex}
  \kpij = \frac{\lambda \abs{\Rset \cap V_{i j}}}{\abs{V_{i j}}} 
= \frac{\lambda}{\abs{V_{i j}}} \int_{V_i} \abs{B_{\rb}(\vx) \cap V_j} \, d\vx,
\end{equation}
the same formula we derived for Cartesian grids
in~\cite{IsaacsonCRDME2013}. Here $\abs{B_{\rb}(\vx) \cap V_j}$
denotes the area of intersection between a disk of radius $\rb$ about
$\vx$ and the voxel $V_j$.  In \cref{ap:hypervolCalc} we
describe how we evaluate the hyper-volume $\abs{\Rset \cap V_{i j}}$
in practice by using this representation as the two-dimensional
integral of an area of intersection. For domains in which $\Omega$ is
not convex, we found it easiest to numerically evaluate $\kpijk$
directly and then use~\cref{eq:kpkmsumsKpij} to calculate $\kpij$.

Evaluating $\kpijk$ for the Doi volume-reactivity model requires the
numerical evaluation of the integral
in~\cref{eq:kpijkDef}. Using~\cref{eq:Doikp} we find
\begin{align*}
  \kpijk &= \frac{\lambda}{\abs{V_{i j}}} \int_{\Rset \cap V_{i j}} 
  \ind_{V_k} \paren{\gamma \vx + (1-\gamma) \vy} \, d\vx \, d\vy \\
         &= \frac{\lambda}{\abs{V_{i j}}} \int_{V_i} \abs{B_{\rb}(\vx) 
           \cap V_j \cap \hat{V}_k(\vx)} \, d\vx,
\end{align*}
where $\hat{V}_k(\vx)$ denotes the translated and dilated set
\begin{equation} \label{eq:VhatDef}
  \hat{V}_k(\vx) = \left\{ \frac{\vy - \gamma \vx}{1 - \gamma} \middle|\, \vy \in V_k \right\}. 
\end{equation}
We evaluated $\kpijk$ through this representation as the
two-dimensional integral of an area of intersection function. Since
both $\hat{V}_k(\vx)$ and $V_j$ are polygons, their intersection is
also one or more polygon(s), and as such the integrand
$\abs{B_{\rb}(\vx) \cap V_j \cap \hat{V}_k(\vx)}$ can be reduced to a
sum of areas of intersections between the disk $B_{\rb}(\vx)$ and
polygons. This allows the direct reuse of the code we developed for
evaluating $\abs{B_{\rb}(\vx) \cap V_j}$. The details of our method
for evaluating the integral are described in
\cref{ap:hypervolCalc}. Knowing both $\kpijk$ and $\kpij$ then allowed
the evaluation of the placement probability $\kpkij$
using~\cref{eq:kpkijDef}.

There are a number of equivalent methods one could use to generate
samples of the jump process for the dissociation reaction
$\textrm{C}_k \to \textrm{A} + \textrm{B}$ with rate $\kmk$ and
placement probabilities $\kmkij$. One approach would be to numerically
evaluate the integral~\cref{eq:kmijkDef}, use~\cref{eq:kpkmsumsKmk} to
evaluate $\kmk$ and use~\cref{eq:kmkijDef} to evaluate $\kmkij$. In
practice we found it simpler to sample a possible time for the next
unbinding reaction using the dissociation rate, $\mu$, and then
exploit the well-mixed approximation for placing reaction
products. The domain boundary is ignored initially, and the \textrm{A}
and \textrm{B} molecules are placed at sampled locations $\vx$ and
$\vy$.  A given reaction event is then rejected if one of $\vx$ or
$\vy$ is outside the domain. If both molecules are placed inside the
domain, the voxel $V_i$ containing $\vx$ and voxel $V_j$ containing
$\vy$ are determined, and both $\textrm{A}_i$ and $\textrm{B}_j$ are
updated. Our precise sampling method is given in
\cref{alg:CkToAandBRx}.
\begin{algorithm}[tbp]
  \caption{Sampling next \emph{possible}
    $\textrm{C}_k \to \textrm{A} + \textrm{B}$ reaction time, $\tau$,
    and product voxel locations, $(V_i,V_j)$.}
  \label{alg:CkToAandBRx}
  \begin{algorithmic}[1]
    \State{Sample \emph{candidate} next reaction time $\tau$ from an
      exponential distribution with rate $\mu$,
      \begin{equation*}
        \tau := \frac{-1}{\mu} \ln\paren{\mathcal{U}_{[0,1)}},
      \end{equation*}      
      where $\mathcal{U}_{[0,1)}$ denotes a uniform random number on $[0,1)$. 
    }     
    \State{Sample the position $\vz$ of the $\textrm{C}_k$ molecule
      within $V_k$ using the well-mixed assumption; i.e. from
      \begin{equation*}
        \frac{1}{\abs{V_k}} \ind_{V_k}(\vz).
      \end{equation*}
    } \State{Given $\vz$, sample the position $\vx$ of the \textrm{A}
      molecule from a uniform distribution within the ball of radius
      $(1-\gamma) \rb$ about $\vz$; i.e. from
      \begin{equation*}
        \frac{1}{\abs{B_{(1-\gamma)\rb}(\vO)}} \ind_{B_{(1-\gamma)\rb}(\vz)}(\vx).
      \end{equation*}
    }   
    \If{$\vx \in \Omega$}
    \State{Given $\vx$ and $\vz$, the position of the \textrm{B} molecule
      is $\vy := (1-\gamma)^{-1} (\vz - \gamma \vx)$.}
    \If{$\vy \in \Omega$}
    \State{Determine which $V_i$ and $V_j$ contain $\vx$ and $\vy$.}
    \State{\Return{$V_i$, $V_j$, and $\tau$.}}
    \EndIf              
    \EndIf
    \State{\Return that no reaction occurs.}
  \end{algorithmic}
\end{algorithm}
In the following theorem we prove that this sampling procedure is
equivalent to directly sampling $\kmijk$ (which in turn is equivalent
to sampling $\kmk$ and $\kmkij$).
\begin{theorem}
  The probability per time a \textrm{C} molecule located in $V_k$
  reacts to produce an \textrm{A} molecule in $V_i$ and a \textrm{B}
  molecule in $V_j$ in \cref{alg:CkToAandBRx} is $\kmijk$.
\end{theorem}
\begin{proof}
  In \cref{alg:CkToAandBRx}, the probability density per time
  a reaction is successful, with the new \textrm{A} molecule placed at
  $\vx$ and the new \textrm{B} molecule at $\vy$ given the \textrm{C}
  molecule is in $V_k$ is
  \begin{align*}
    \rho(\vx,\vy) &= \mu \ind_{\Omega}(\vx) \ind_{\Omega}(\vy) 
                    \int_{V_k} \delta \paren{\vy - \frac{\vz - \gamma \vx}{1 - \gamma}}
                    \paren{\frac{\ind_{B_{(1-\gamma)\rb}(\vz)}(\vx)}{\abs{B_{(1-\gamma)\rb}(\vO)}}}
                    \paren{\frac{\ind_{V_k}(\vz)}{\abs{V_k}}} \, d\vz \\
    &= \frac{1}{\abs{V_k}} \ind_{\Omega}(\vx) \ind_{\Omega}(\vy) \int_{V_k} \km(\vx,\vy\vert\vz) \, d\vz.
  \end{align*}
  The probability per time the reaction successfully occurs producing
  $\vx \in V_i$ and $\vy \in V_j$ is then
  \begin{align*}
    \int_{V_{i j}} \rho(\vx,\vy) \, d\vx \, d\vy 
    &= \frac{1}{\abs{V_k}} \int_{V_{i j k}} \km(\vx,\vy\vert\vz) \, d\vx \, d\vy \, d\vz \\
    &= \kmijk,
  \end{align*}
  where the last line follows by definition~\cref{eq:kpijkDef}.
\end{proof}

Finally, we note that there is a third method one can use to determine
the $\kmijk$'s, and hence $\kmk$ and $\kmkij$, when detailed balance
is satisfied. As we discuss in~\cite{IsaacsonZhang17}, the
volume-reactivity model with choices~\cref{eq:Doikp}
and~\cref{eq:Doikm} satisfies the pointwise detailed balance
condition
\begin{equation} \label{eq:detailedBalanceCondit}
  \kp\paren{\vz\vert\vx,\vy} \bar{p}(\vx,\vy) = \km\paren{\vx,\vy\vert\vz} \bar{p}_{\textrm{b}}(\vz),
\end{equation}
where $\bar{p}(\vx,\vy)$ and $\bar{p}_{\textrm{b}}(\vz)$ denote the
equilibrium solutions to~\cref{eq:pPDE}
and~\cref{eq:pbPDE}. Combining~\cref{eq:detailedBalanceCondit} with
the reflecting domain boundary conditions we derived
in~\cite{IsaacsonZhang17} that
\begin{equation} \label{eq:equlibSoluts}
  \begin{aligned}
    \bar{p}(\vx,\vy) &= \frac{1}{\abs{\Omega}} \frac{K_d}{1 + K_d \abs{\Omega}}, 
    &\bar{p}_{\textrm{b}}(\vz) &= \frac{1}{\abs{\Omega}}\frac{1}{1 + K_d \abs{\Omega}},
  \end{aligned}
\end{equation}
where $K_d$ corresponds to the equilibrium dissociation constant of
the reaction.  Substituting into~\cref{eq:detailedBalanceCondit} and
integrating over $V_{i j k}$ we find
\begin{equation*}
 K_d \abs{V_{i j}} \kpijk = \abs{V_k} \kmijk.
\end{equation*}
Therefore, once $\kpijk$ is evaluated we may calculate $\kmijk$ using
\begin{equation*}
  \kmijk = \frac{K_d \abs{V_{i j}}}{\abs{V_k}} \kpijk.
\end{equation*}

\begin{remark}
  Using the proceeding equation, by direct substitution it follows
  that whenever the spatially continuous volume reactivity model
  satisfies the detailed balance
  condition~\cref{eq:detailedBalanceCondit}, the
  CRDME~\cref{eq:twopartCRDME} has the equilibrium solutions
  \begin{align*}
    \bar{P}_{i j} &= \bar{p} \abs{V_{i j}} = \frac{\abs{V_{i j}}}{\abs{\Omega}} \frac{K_d}{1 + K_d \abs{\Omega}}, 
    &\bar{P}_{\textrm{b}k} &= \bar{p}_{\textrm{b}} \abs{V_k} = \frac{\abs{V_{k}}}{\abs{\Omega}}\frac{1}{1 + K_d \abs{\Omega}},
  \end{align*}
  and satisfies the discrete detailed balance
  condition
  \begin{equation*}
    \kpijk \bar{P}_{i j} = \kmijk \bar{P}_{\textrm{b}k}.
  \end{equation*}  
\end{remark}
\commenta{
\begin{remark}
  If the reaction-rate functions in the volume reactivity model do not
  satisfy detailed balance, the dissociation transition rates,
  $\kmijk$, can still be evaluated by quadrature (i.e. numerically
  evaluating~\eqref{eq:kmijkDef}). Alternatively, one can modify the
  sampling procedure given in~\cref{alg:CkToAandBRx} for the chosen
  $\km(\vx,\vy\vert\vz)$.
\end{remark}
}

\section{Numerical Examples} \label{S:numericExs} We now illustrate
the convergence and accuracy of the unstructured mesh CRDME with
several examples. For all simulations we generate exact realizations
of the jump process $\vec{W}(t)$ associated with the CRDME, defined in
\cref{S:unstructCRDME}, using the next reaction method
SSA~\cite{GibsonBruckJPCHEM2002}.  We begin in
\cref{S:numConvUnstructMesh} by demonstrating that several
reaction time statistics converge to finite values as the mesh size
approaches zero for the two-particle
$\textrm{A} + \textrm{B} \to \varnothing$ annihilation reaction within
a circle. We examine two different association functions
$\kp(\vx,\vy)$, the smooth Gaussian
interaction~\cref{eq:GaussInteract} and the standard discontinuous
Doi interaction~\cref{eq:Doikpxy}.

With convergence established for the forward reaction approximation,
we then confirm in \cref{S:numConvRevRx} that statistics of the
two-particle reversible
$\textrm{A} + \textrm{B} \leftrightarrows \textrm{C}$ reaction
converge to the solution of the Doi model by comparison with Brownian
Dynamics (BD) simulations. Finally, in \cref{S:CRDMEapplications} we
consider several multiparticle systems. We first consider an example
from~\cite{ElfPNASRates2010}, and show that our method is consistent
with results from both Brownian Dynamics simulations and the
renormalized RDME approach of~\cite{ElfPNASRates2010}. To show the
flexibility of our method, we conclude by looking at a simplified
version of a signal propagation model \cite{MunozGarcia:2009hd} in the
complex two-dimensional geometry given by the cytosol of a human B
cell (corresponding to a slice plane from a full three-dimensional
reconstruction).

\subsection{$\textrm{A} + \textrm{B} \rightarrow \varnothing$ Annihilation Reaction}
\label{S:numConvUnstructMesh}
We begin by examining the
$\textrm{A} + \textrm{B} \rightarrow \varnothing$ annihilation
reaction in a system with just one \textrm{A} molecule and one
\textrm{B} molecule.  We consider both the CRDME with both the
discontinuous Doi interaction~\cref{eq:Doikp} and the smooth Gaussian
interaction~\cref{eq:GaussInteract}. \commentab{We also include an
  RDME model for comparison, illustrating the lack of convergence of
  the RDME even for this simple example (in contrast to the two CRDME
  models).}

Molecules are assumed to diffuse within a disk centered at the origin
of radius $R =0.1 \mu$m, i.e. $\Omega = B_{R}(\vO)$. A reflecting
Neumann boundary condition is assumed on the circle
$\partial B_{R}(\vO)$, so that molecules can not leave the domain.
The circle is approximated by a set of $122$ line segments, and mesh
refinement is restricted to the interior of the circle
(\cref{fig:dualMeshEx}). We discretize the circle into a primal
triangular mesh using MATLAB's \texttt{delaunayTriangulation}
function, specifying an edge constraint on the boundary to ensure the
triangulation is strictly in the interior of the domain. Starting from
this initial mesh, we subsequently create a series of refined meshes
by repeatedly dividing each triangle into four congruent
triangles. Repeating this step throughout the entire initial Delaunay
triangular mesh produces a consistent refined mesh that preserves
Delaunay properties~\cite{Carey1997}. A dual polygonal mesh on which
molecules diffuse and react is constructed at the final stage of the
refinement. That is, the CRDME and RDME are defined on this polygonal
dual mesh as described in previous sections. In what follows, we
denote by $h$ the maximum diameter of all polygons within a given dual
mesh.

In the remainder of the paper, unless otherwise stated, spatial units
of all parameters are micrometers and time is seconds. For all
simulations of the annihilation reaction we chose the \textrm{A} and
\textrm{B} molecules' diffusion constants to be $10\mu$m$^2$s$^{-1}$.
For CRDME simulations using the Doi reaction
mechanism~\cref{eq:Doikp}, we choose the reaction radius $\rb$ to be
$10^{-3}\mu\text{m}$ and $\lambda = 10^9\text{s}^{-1}$. In the case of
the CRDME with Gaussian interaction~\cref{eq:GaussInteract}, we choose
$\rb = 0.025\mu\text{m}$, corresponding to a typical interaction
distance for the tethered enzymatic reactions we studied
in~\cite{IsaacsonGoyetteDuskek2016}. For such interactions, the
catalytic rate $\lambda$ is set to be
$2.55459\times10^{7}\mu\text{m}^3\text{s}^{-1}$, which is calibrated
so that the mean reaction time between the two molecules matches that
when using the Doi reaction mechanism. \commentab{Finally, in the RDME
  model we choose the (well-mixed) association rate defined in
  \cref{S:crdmeTorRDME} to be $\beta^+ = \lambda \pi \rb^2$. This
  choice is consistent with the effective well-mixed reaction rate one
  would expect from the volume-reactivity model with Doi
    interaction when $\rb \sqrt{\lambda/D}$ is a small parameter,
  see~\cite{IsaacsonCRDME2013}.}

\begin{figure}[!tbp]
  \centering
  \subfloat[\commentab{RDME}]{
    \label{fig:survProbCircRDME}
    \hspace{-10pt}
    \includegraphics[width=.32\textwidth]{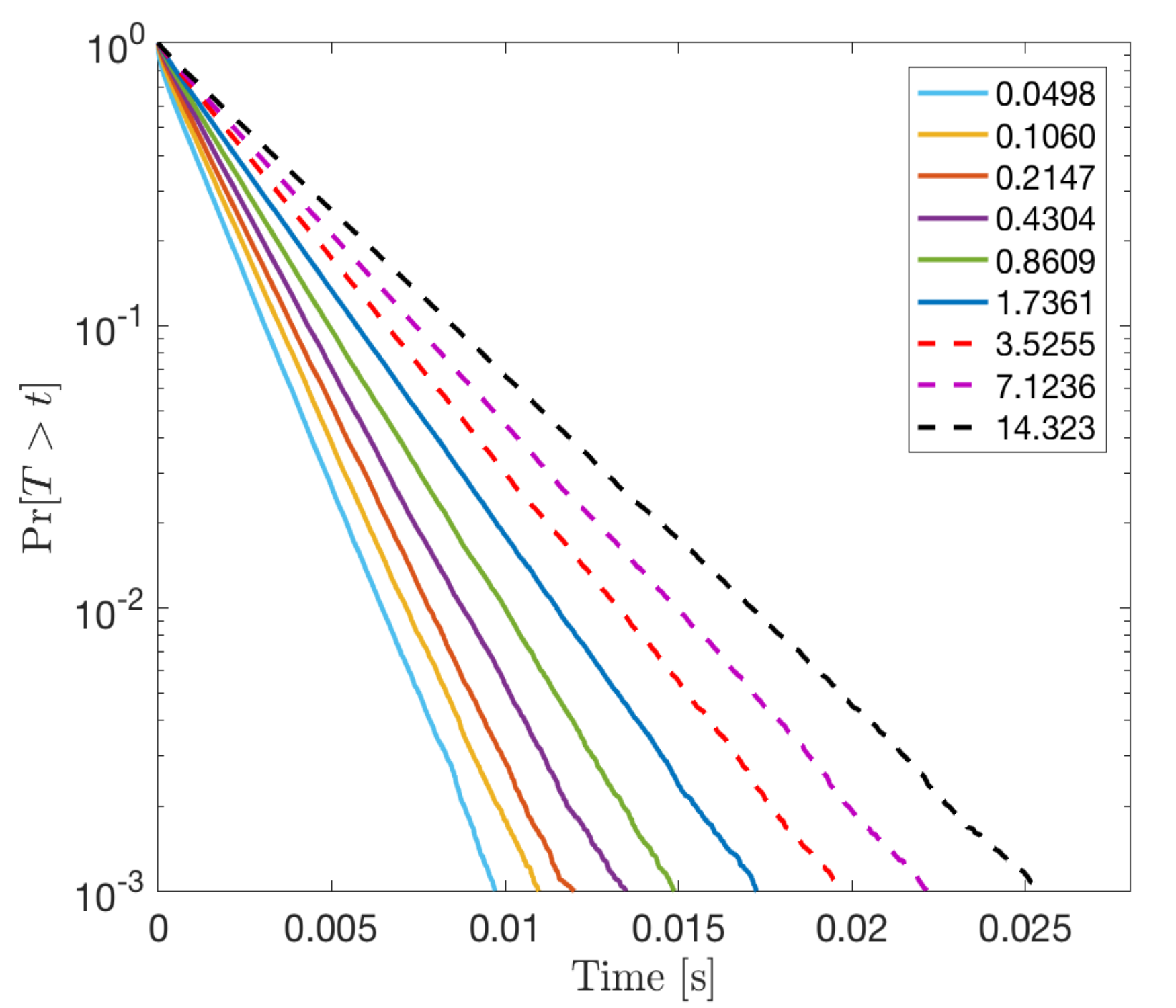}
  }
  \subfloat[CRDME, Doi interaction]{
    \label{fig:survProbCirc}
    \includegraphics[width=.32\textwidth]{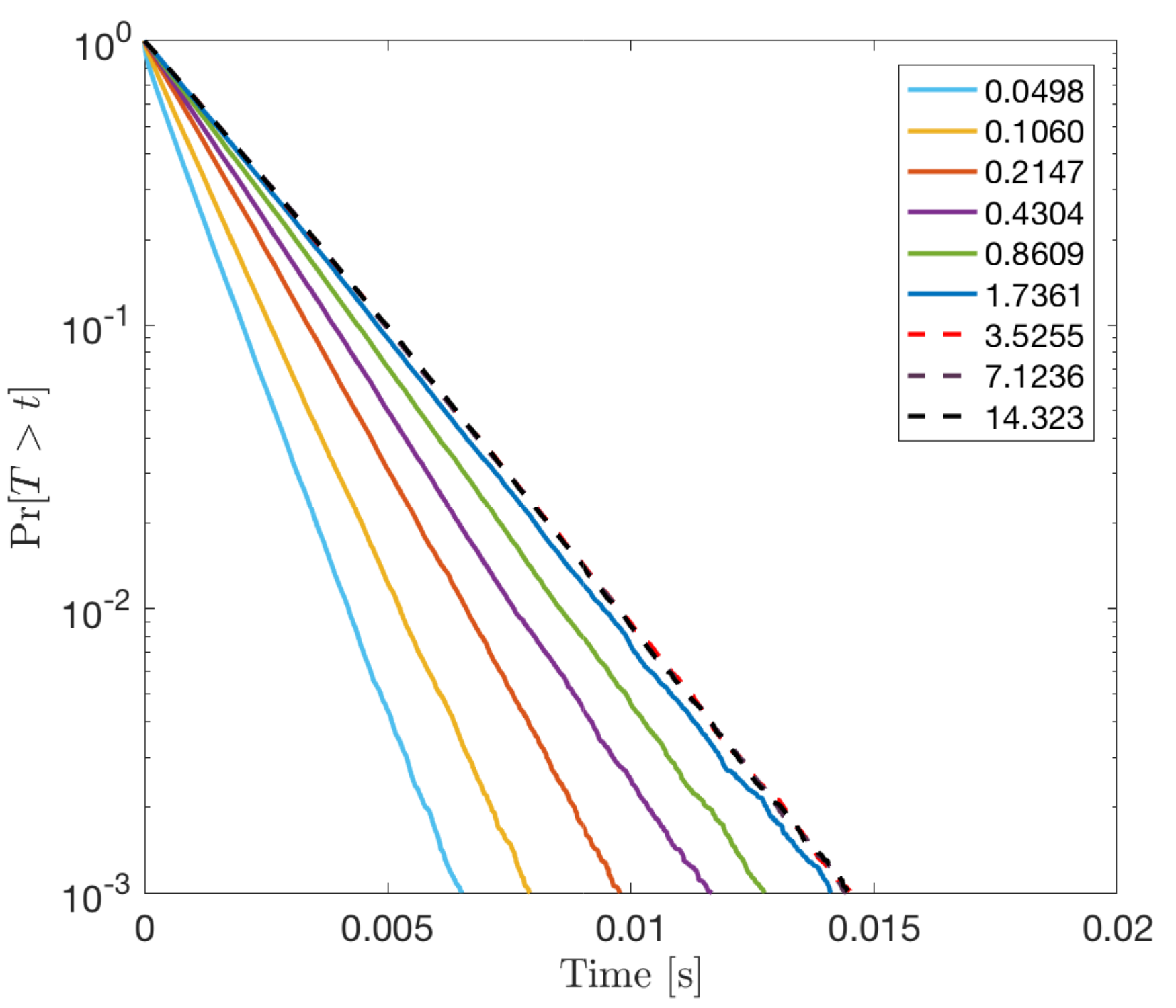}
  }
  \subfloat[CRDME, Gaussian interaction]{
    \label{fig:GaussSurvProb}
    \includegraphics[width=.32\textwidth]{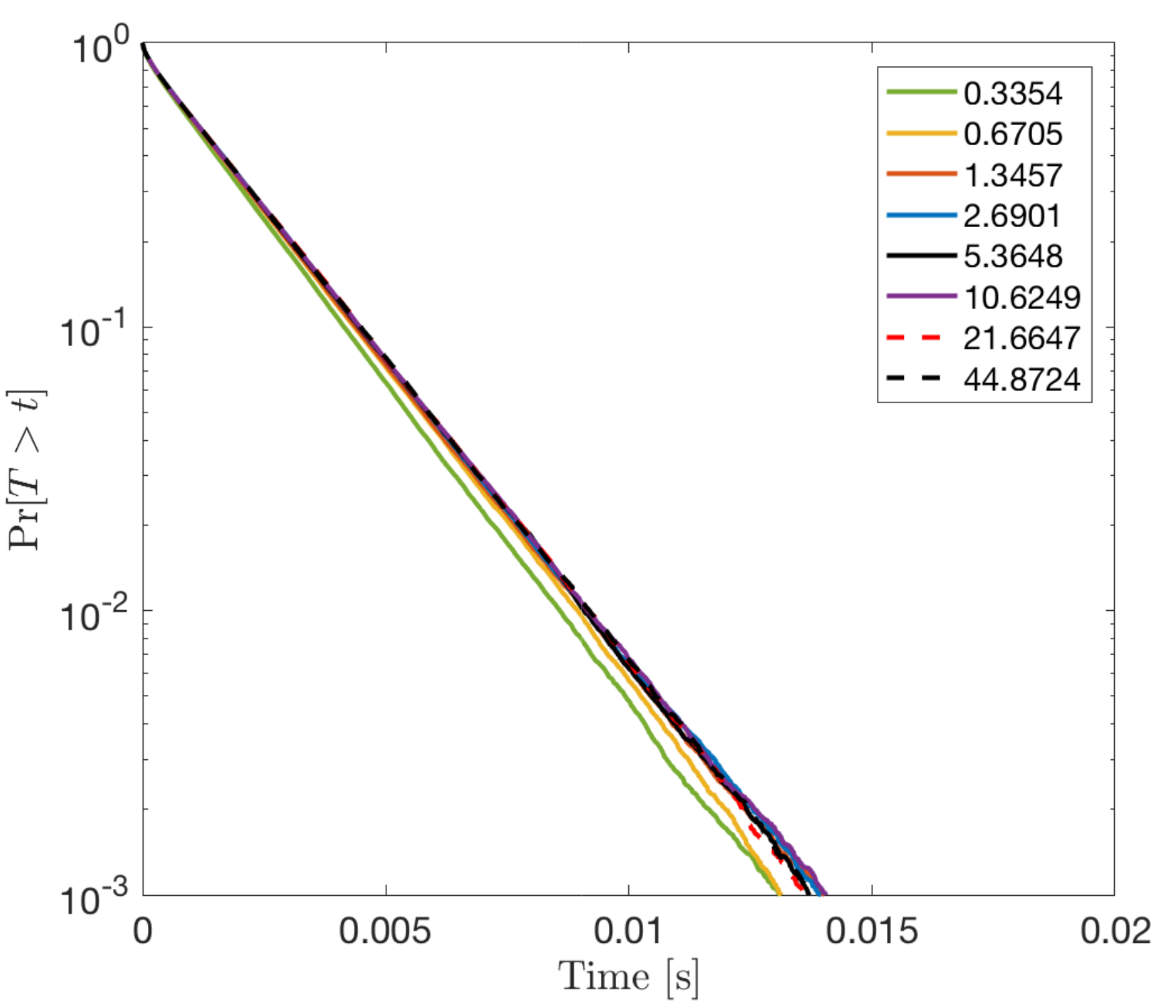}    
}
\caption{Survival time distributions vs. $t$ from the two-particle
  $\textrm{A} + \textrm{B} \to \varnothing$ reaction \commentab{for
    (a) (non-convergent) RDME model; (b) CRDME Doi reaction
    model~\cref{eq:Doikpxy}; (c) CRDME Gaussian interaction
    model~\cref{eq:GaussInteract}.} In each case the domain is a
  disk. \commentab{For the RDME each curve was estimated from 128000
    simulations, for the CRDME with Doi interaction from 128000
    simulations, and for the CRDME with Gaussian interaction from
    100000 simulations.} The legends give the ratio, $\rb$/$h$ as the
  mesh is refined ($h$ is approximately successively halved). See
  \cref{S:numConvUnstructMesh} for other parameter values and
  details. \commentab{We see that the survival time distribution for
    the RDME never converges as $h$ is reduced, while for both
    reaction models the CRDME survival time distributions converge.}}
\end{figure} 
Let $\Tb$ denote the random time for the two molecules to react when
each starts uniformly distributed in $\Omega$. The corresponding
survival time distribution is given by
\begin{equation*}
  \prob\brac{\Tb > t} = \int_{\Omega}p(\vx, \vy, t) \, d\vx d\vy ,
\end{equation*}
where $\vx$ and $\vy$ are the locations of the $A$ and $B$ molecules
respectively, and $p(\vx,\vy,t)$ satisfies~\cref{eq:pPDE} with
$\km(\vx,\vy\vert\vz) = 0$. We estimate the survival time distribution
from the numerically sampled reaction times using the \texttt{ecdf}
command in MATLAB.  \commentab{ \Cref{fig:survProbCircRDME}
  demonstrates the divergence of the estimated survival time
  distribution for the RDME as the mesh width is reduced. Note, there
  is no range of mesh widths over which the survival distribution
  from the RDME is robust to changes in $h$. In contrast,
  \Cref{fig:survProbCirc} demonstrates the convergence (to within
  sampling error) of the estimated survival time distribution of the
  unstructured mesh CRDME using a Doi interaction. Similarly,
  \cref{fig:GaussSurvProb} demonstrates the convergence (to within
  sampling error) of the estimated survival time distribution of the
  unstructured mesh CRDME using a Gaussian interaction.} For both
CRDME models, the survival time distributions are seen to converge as
the maximum mesh width $h \to 0$.

\begin{figure}[!tbp]
  \centering
  \subfloat[Convergence of $\avg{\Tb}$]{
    \label{fig:convergMRT}
    \includegraphics[width=.48\textwidth]{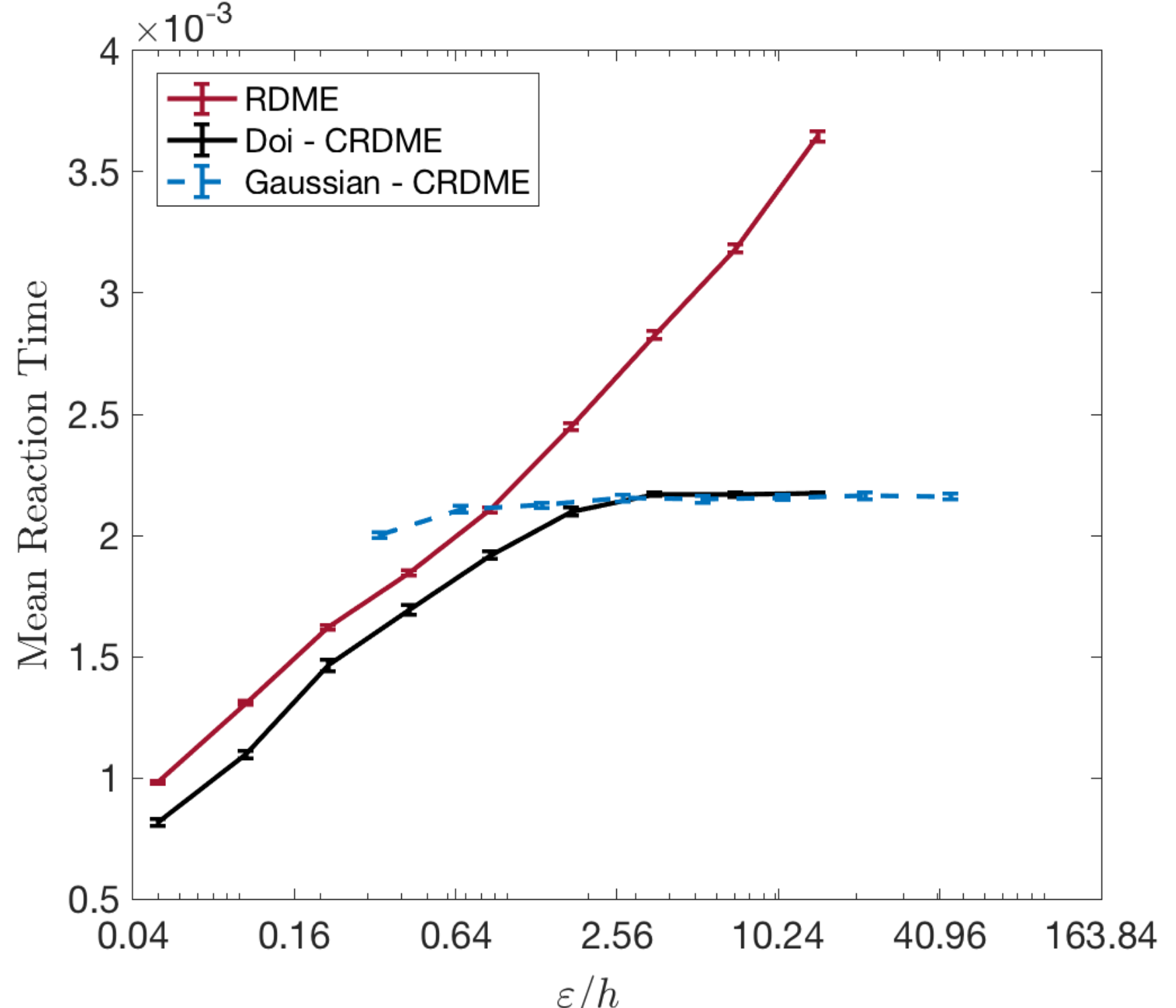}
  }
  \subfloat[Rate of convergence of $\avg{\Tb}$]{
    \label{fig:sDiffCirc}
    \includegraphics[width=.48\textwidth]{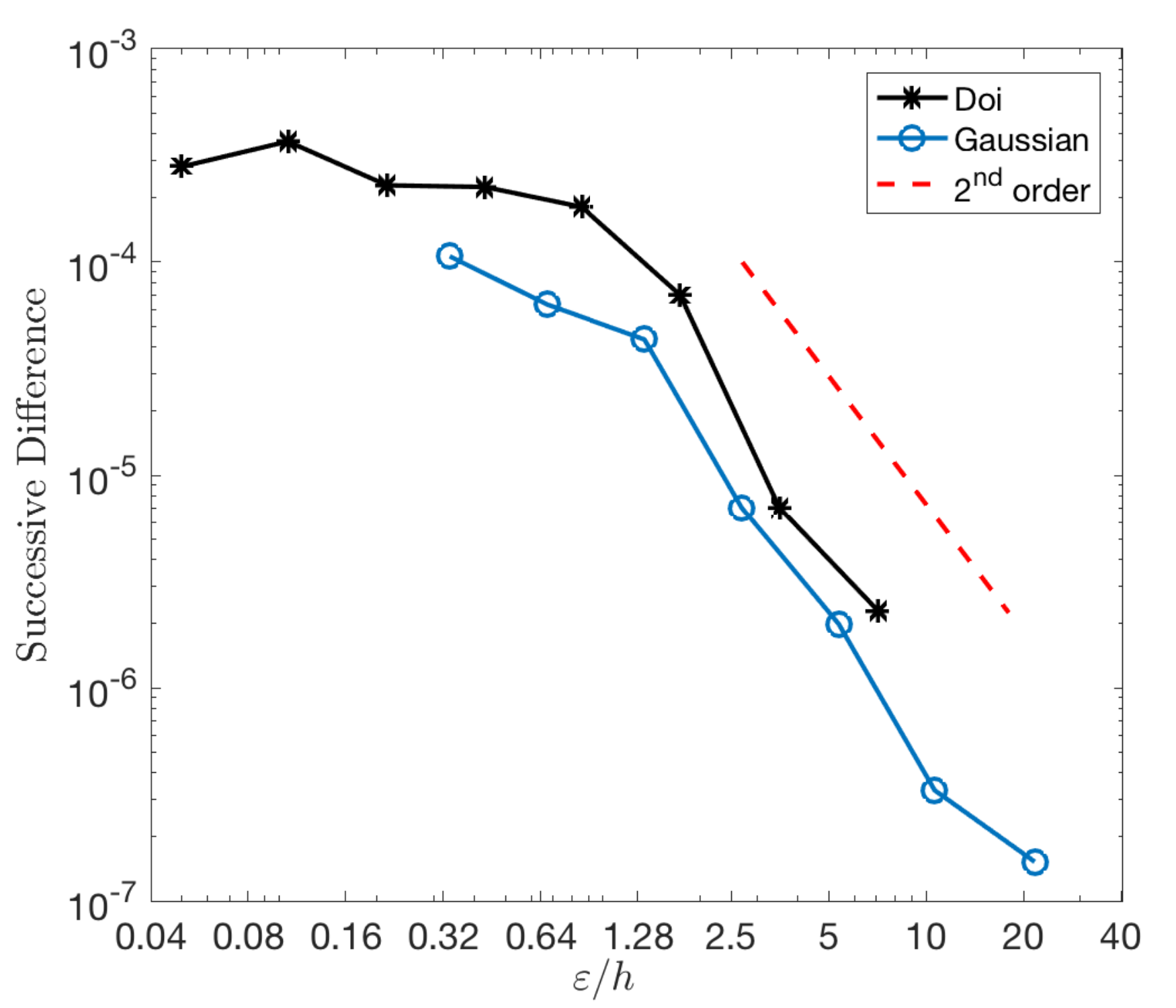}    
  }
  \caption{\commentab{Mean reaction time $\avg{\Tb}$ for the
      two-particle $\textrm{A} + \textrm{B} \to \varnothing$ reaction
      as the mesh width, $h$, is reduced. In panel (a) we plot the
      mean reaction time $\avg{\Tb}$ vs. $\rb$/$h$ as $h$ is
      (approximately) successively halved. Each mean reaction time for
      the RDME was estimated from 128000 simulations; for the CRDME
      with Doi interaction from 128000 simulations; and for the CRDME
      with Gaussian interaction from 100000 simulations.} Note, 95$\%$
    confidence intervals are drawn on each data point, but for some
    points are smaller than the marker labeling the point. See
    \cref{S:numConvUnstructMesh} for parameter values and further
    details. In panel (b) we demonstrate the rate of convergence when
    using the CRDME with Doi or Gaussian reaction models by plotting
    the difference between successive points on the corresponding
    $\avg{\Tb}$ vs $\rb/h$ curves from \cref{fig:convergMRT}. The
    smaller of the two $h$ values is used for labeling. \commentab{The
      effective convergence rate of the successive differences to zero
      for the CRDME with either reaction model scales roughly like
      $O(h^2)$.}}
    \label{fig:convergCRDME}
\end{figure}
To study the rate of convergence we examined the mean reaction time
$\avg{\Tb}$, defined by
\begin{equation*}
  \avg{\Tb} = \int_{0}^{\infty} \prob \brac{\Tb > t} \, dt.
\end{equation*}
We estimated the mean reaction time from the numerically sampled
reaction times by calculating the sample mean. In
\cref{fig:convergMRT} we show the sample mean reaction times for the
three choices of reaction mechanisms as $\rb$/$h$ is
varied. \commentab{We see that as $\rb$/$h$ $\to \infty$ (i.e.
  $h \to 0$) the sample mean reaction times for both CRDME reaction
  models converge to a finite value, while the sample mean for the
  RDME reaction model diverges.}  \Cref{fig:sDiffCirc} illustrates the
rate of convergence for the CRDME Doi and Gaussian interaction models
by plotting the successive difference of the estimated mean reaction
times as $h$ is decreased (approximately halved). For $h$ sufficiently
small, the empirical rate of convergence for both reaction mechanisms
is roughly second order.

\subsection{$A + B \rightleftharpoons C$  Reversible Binding Reaction}
\label{S:numConvRevRx}
We now consider the reversible bimolecular
$A + B \rightleftharpoons C$ reaction in a system that initially
contains just one $C$ molecule. The corresponding volume reactivity
model is then given by~\cref{eq:pPDE} and~\cref{eq:pbPDE}. It is
assumed all three molecules have the same diffusion constant,
$\DA = \DB = \DC = 0.01 \mu\text{m}^2 \text{s}^{-1}$. The domain
$\Omega$ is chosen to be a square with sides of length
$L = 0.2 \mu\text{m}$, and we assume a reflecting Neumann boundary
condition on $\partial \Omega$ in each of the $\vx$, $\vy$, and $\vz$
coordinates. We use the Doi reaction model~\cref{eq:Doikp} for the
forward $\textrm{A} + \textrm{B} \to \textrm{C}$ reaction, with
reaction radius $\rb = 10^{-3}\mu\text{m}$, and consider two
dissociation mechanisms: the point unbinding model~\cref{eq:ptUnbind}
introduced in~\cite{ErbanChapman2011}, and the uniform unbinding
model~\cref{eq:Doikm}. For the association reaction, the product
\textrm{C} molecule is placed at the diffusion weighted center of
mass~\cref{eq:diffWeightCentMass}, so that $\gamma = \frac{1}{2}$.
For all simulations the \textrm{C} molecule was initially placed
randomly within $\Omega$, corresponding to the initial conditions that
\begin{align*}
p(\vx,\vy,0) = 0, && \pb(\vz,0) = \frac{1}{\abs{\Omega}}.
\end{align*}

To confirm that the unstructured mesh CRDME converges to the solution
of the Doi volume-reactivity model, we compare statistics from SSA
simulations of the jump processes corresponding to the CRDME against
statistics calculated from Brownian Dynamics (BD) simulations using
the method of~\cite{ErbanChapman2009,ErbanChapman2011} (with a fixed
time step of \commenta{$dt = 10^{-10}\textrm{s}$}). Unless otherwise stated, for
all simulations the association rate constant in~\cref{eq:Doikp} was
chosen to be $\lambda = 9.3662 \times 10^7s^{-1}$, and the
dissociation rate constant in~\cref{eq:ptUnbind} and~\cref{eq:Doikm}
was chosen to be $\mu = 9.2735 \times 10^{5}s^{-1}$. Here $\lambda$
was determined by matching the mean association time $\Tb$
($\Tb = 1.9328$s) for the irreversible
$\textrm{A} + \textrm{B} \to \varnothing$ reaction to occur in
$\Omega$, given a uniform initial distribution for the \textrm{A} and
\textrm{B} molecules, to the corresponding time found in Figure~2
of~\cite{ElfPNASRates2010}. $\mu$ was then determined by matching the
equilibrium constant ($K = 3.1730\times 10^{-4} \mu \text{m}^2$)
with that in~\cite{ElfPNASRates2010}.
\begin{figure}[!tbp]
  \centering
  \subfloat[Point unbinding]{
    \label{fig:PointUnbind}
    \includegraphics[width=.48\textwidth]{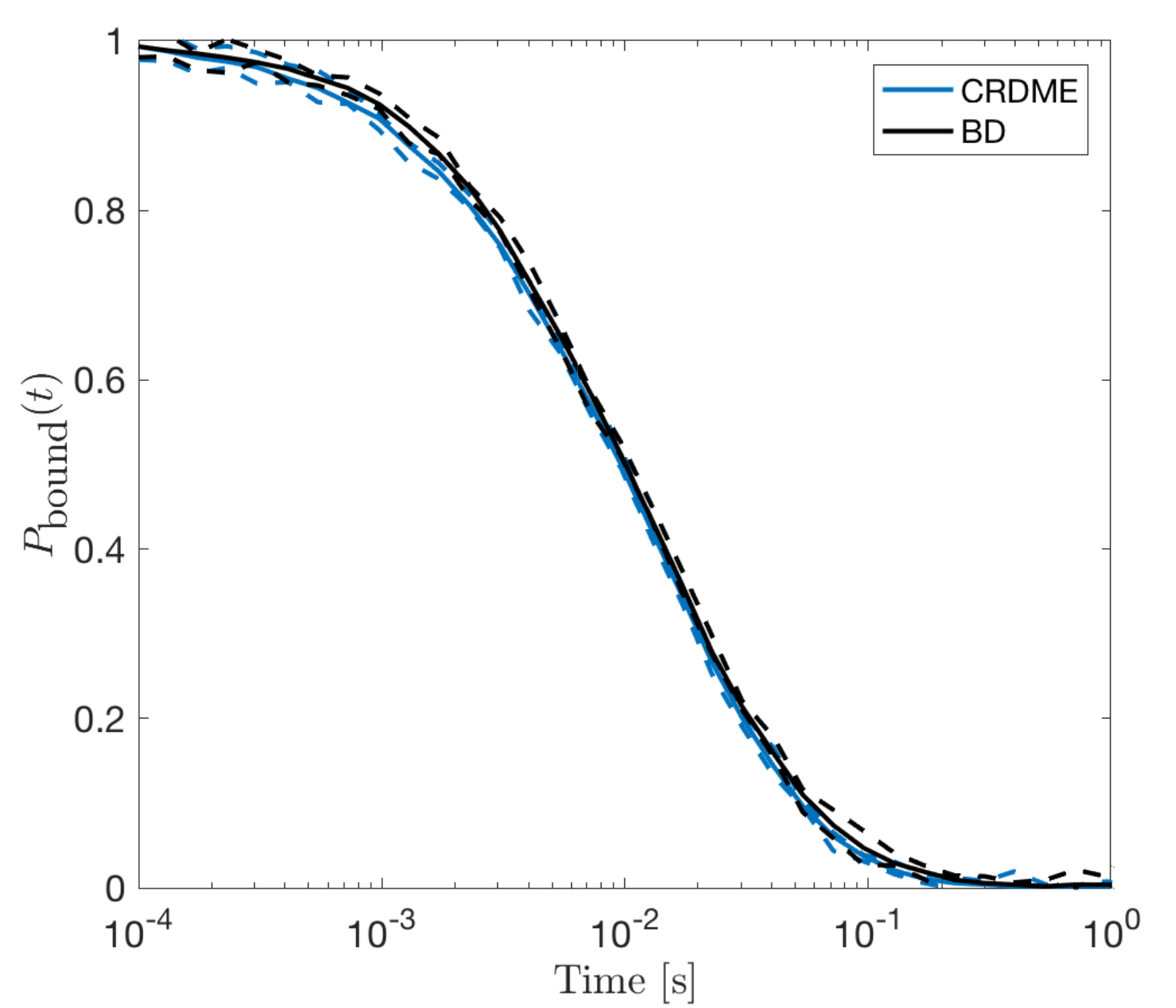}
  }
  \subfloat[Uniform unbinding]{
    \label{fig:UniformUnbind}
    \includegraphics[width=.48\textwidth]{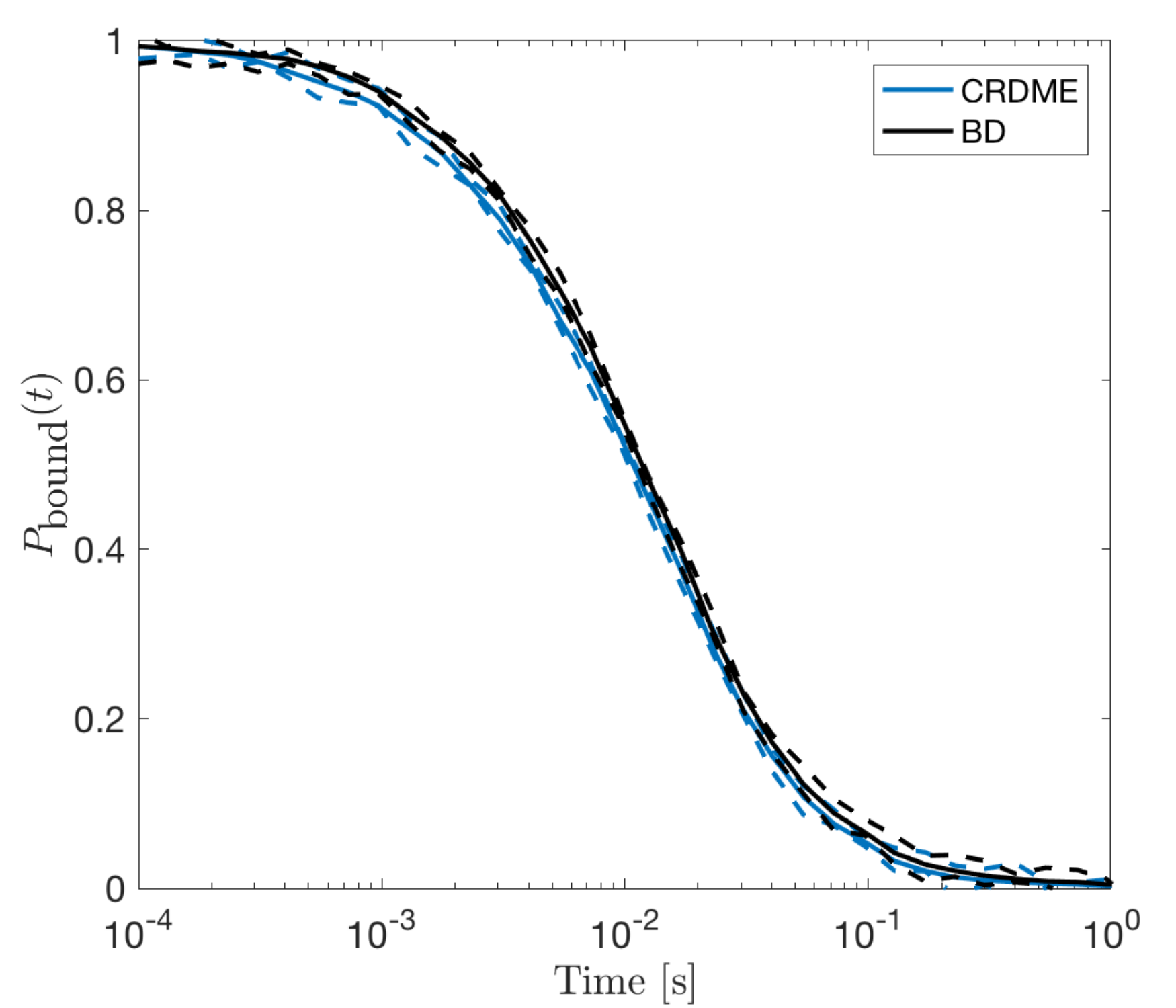}    
  }
  \caption{Probability molecules are bound $P_{\textrm{bound}}(t)$
    vs. time for CRDME SSA simulations and BD simulations. Blue curves
    correspond to the CRDME simulations and black to the BD
    simulations. Corresponding 95$\%$ confidence intervals are drawn
    with dashed lines in the same color. Each curve was estimated from
    128000 simulations. The left panel corresponds to the point
    unbinding model~\cref{eq:ptUnbind}, while the right corresponds to 
    the uniform unbinding model~\cref{eq:Doikm}. The domain $\Omega$
    was a square with sides of length $L = 0.2\mu\text{m}$, and was
    discretized into $N = 263169$ polygonal voxels with maximum mesh
    width approximately $h = 3.5334\times 10^{-5} \mu m$. The
    polygonal mesh was constructed as the dual mesh to a uniform
    triangulation of the square, that was itself obtained from a
    Cartesian mesh by dividing each square into two triangles. For
    remaining parameters see \cref{S:numConvRevRx}.}
    \label{fig:compareBD}
\end{figure}
\begin{figure}[!tbp]
  \centering
  \subfloat[Convergence of $P_{\textrm{bound}}(t)$]{
    \label{fig:convergPb}
    \includegraphics[width=.48\textwidth]{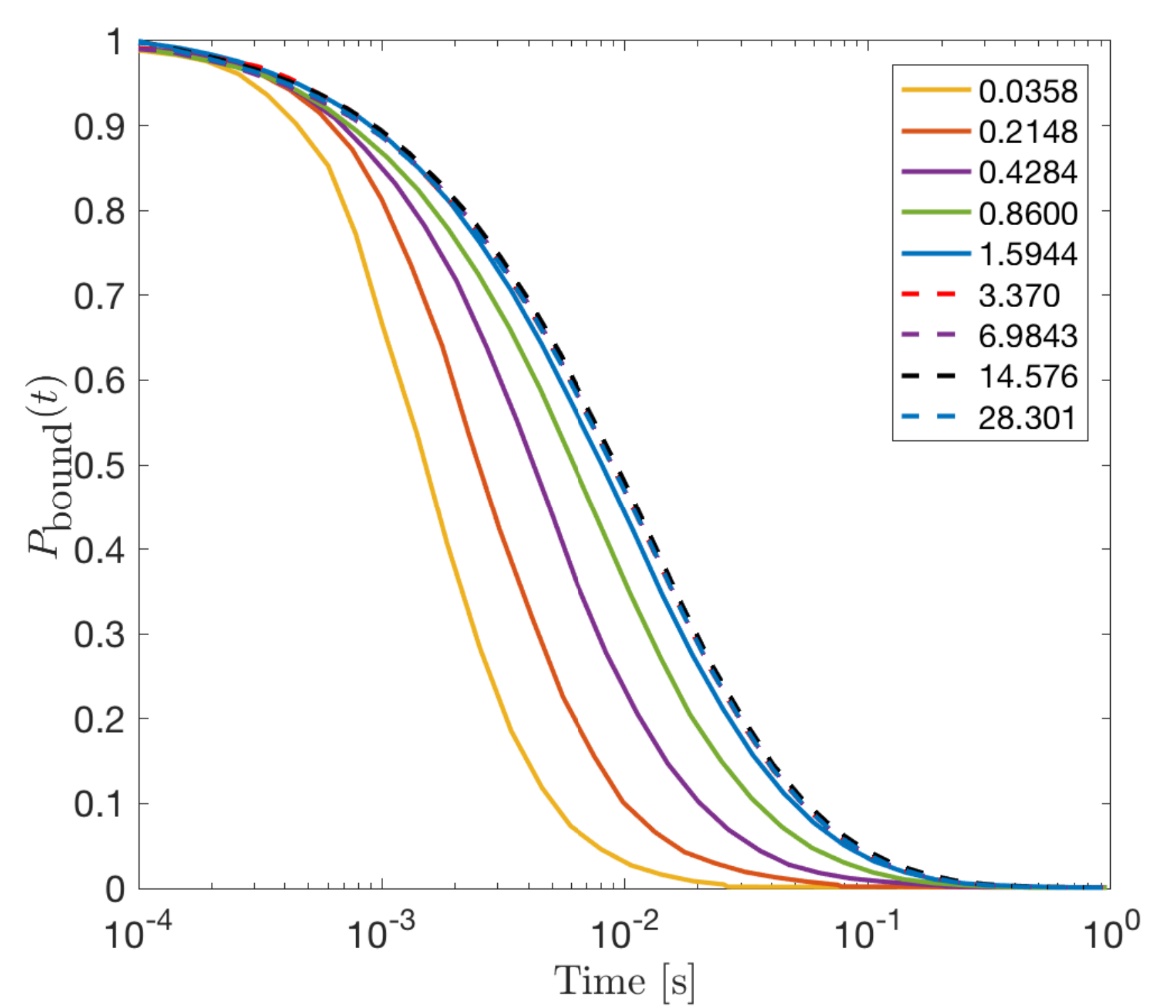}
  }
  \subfloat[Rate of convergence of $P_{\textrm{bound}}(t)$]{
    \label{fig:convergTeq}
    \includegraphics[width=.48\textwidth]{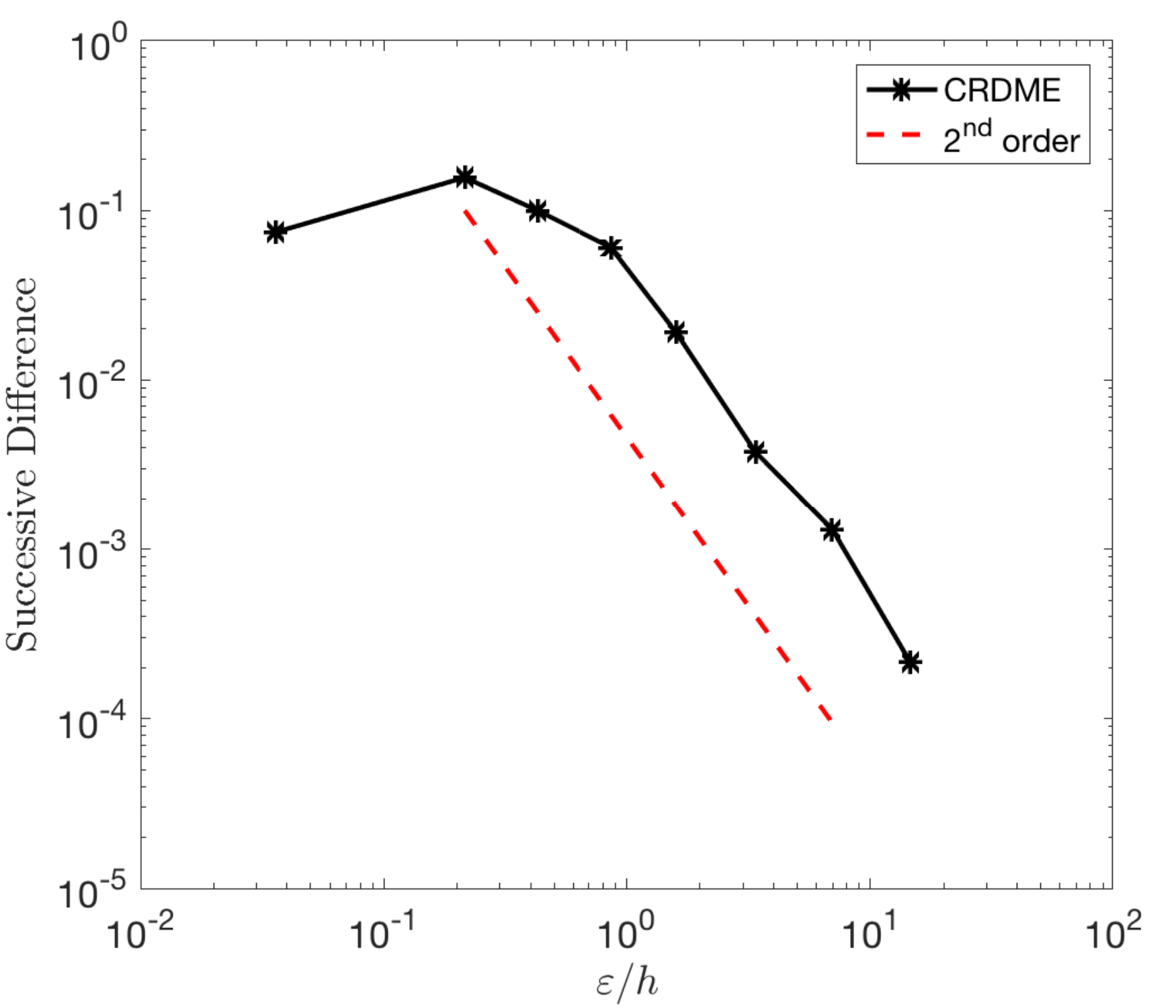}    
  }
  \caption{Convergence of the probability the molecules are in the
    bound state, $P_\textrm{bound}(t)$, as $h \to 0$. In panel (a) we
    plot $P_{\textrm{bound}}(t)$ vs time as $\rb / h$ is varied for
    $\rb = 10^{-3}\mu\text{m}$. Each curve was estimated from 100000
    simulations. We see convergence as the mesh width $h$ goes to 0
    (i.e.  $\rb /h \to \infty$). Legend gives the ratio, $\rb/h$, for
    each curve. For remaining parameters see \cref{S:numConvRevRx}. In
    panel (b) we demonstrate the rate of convergence at $t = 0.01$s by
    plotting the difference between successive points of
    $P_\textrm{bound}(t)$ vs. $\rb/h$. The smaller of the two $h$
    values is used for labeling. The effective convergence rate to
    zero scales like $O(h^2)$.}
    \label{fig:convergRevPb}
\end{figure}
Let $P_\textrm{bound}(t)$ denote the probability the \textrm{A} and
\textrm{B} molecules are bound together in the \textrm{C} state at
time $t$. That is
\begin{equation*}
  P_\textrm{bound}(t) = \int_{\Omega} \pb(\vz,t) \, d\vz.  
\end{equation*}
We estimate $P_\textrm{bound}(t)$ numerically by averaging the number
of C molecules at a fixed time in the system over the total number of
CRDME (resp. BD) simulations. \Cref{fig:compareBD} demonstrates that
for each of the unbinding models, $P_\textrm{bound}(t)$ from the
unstructured mesh CRDME with Doi interaction agrees to statistical
error with $P_\textrm{bound}(t)$ from BD
simulations. \Cref{fig:convergPb} shows the convergence of
$P_\textrm{bound}(t)$ from the CRDME as the mesh width $h \to 0$. To
illustrate the rate of convergence of $P_\textrm{bound}(t)$, in
\cref{fig:convergTeq} we plot the successive difference of the
estimated $P_\textrm{bound}(t)$ at $t = 0.01s$ as $h$ is approximately
halved. In the limit that $\rb/h \to \infty$, the empirical rate of
convergence is roughly second order.

\subsection{CRDME Applications}
\label{S:CRDMEapplications}
In the previous subsections we demonstrated the convergence of the
CRDME for two basic bimolecular chemical reactions involving at most
two molecules ($A + B \to \emptyset$ and
$A + B \rightleftharpoons C$). We now demonstrate that the CRDME is
capable of accurately resolving more general multiparticle reaction
systems, considering the example given by equation~3
of~\cite{ElfPNASRates2010}. The domain $\Omega$ is chosen to be a
square with sides of length $L = 1 \mu\text{m}$, allowing us to
directly compare with the results of~\cite{ElfPNASRates2010}.

\begin{figure}[tbp]
  \centering
  \includegraphics[scale=0.42]{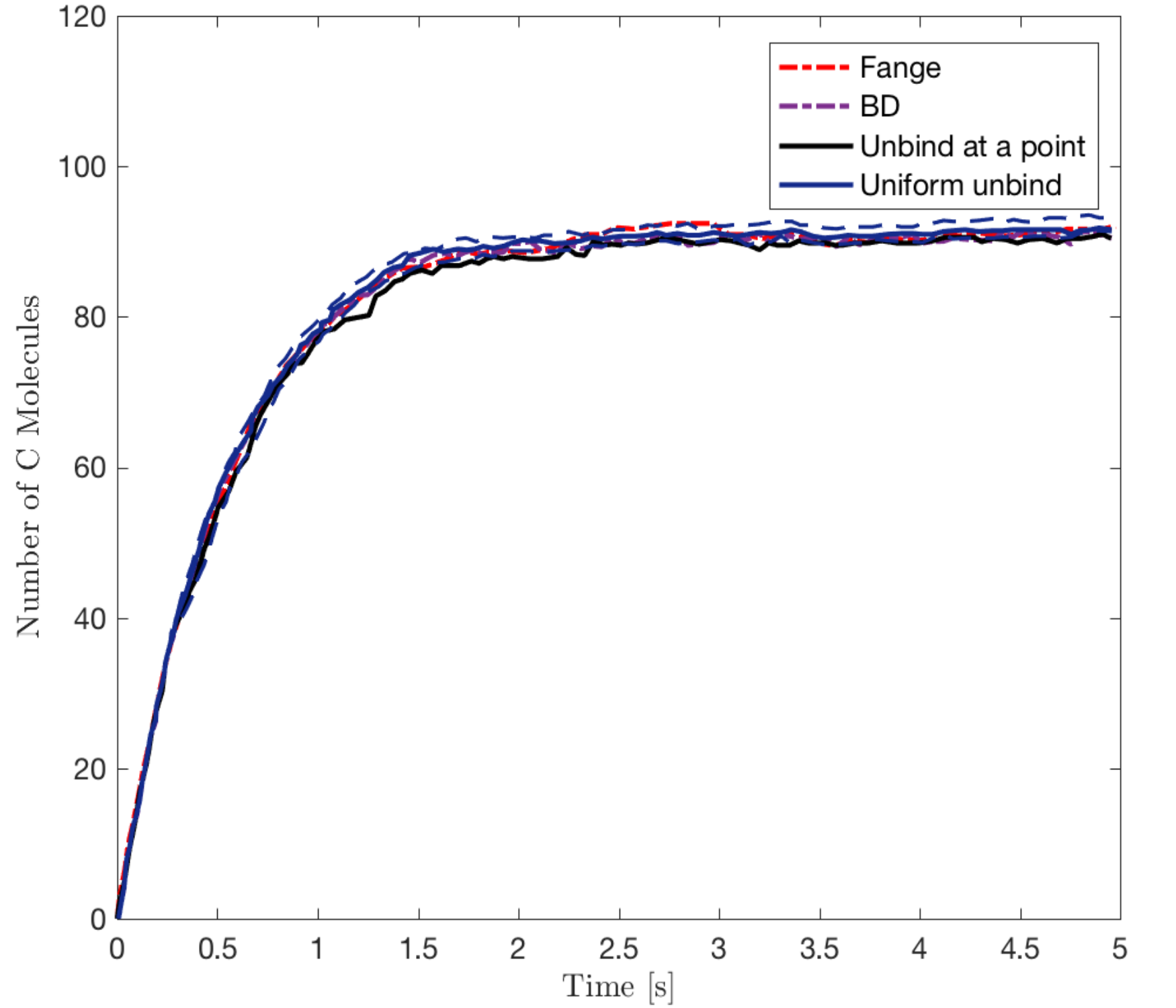}
  \caption{Mean number of \textrm{C} molecules vs. time when $\rb$ =
    $10^{-3}\mu$m for the reaction system~\cref{eq:chemSysm} in a
    square with sides of length $1 \mu\text{m}$. The diffusion
    constant of all species is $D = 0.01 \mu$m$^{2}$s$^{-1}$. The
    production rate of C molecules is $\kprod$ = 180$s^{-1}$ and the
    first-order rate constant for degradation of A and B molecules is
    $\kdecay$ = 10$s^{-1}$. The blue line corresponds to CRDME
    simulations of the point unbinding model~\cref{eq:ptUnbind}, while
    the solid black line corresponds to CRDME simulations of the
    uniform unbinding model~\cref{eq:Doikm}. The dash-dot red line
    gives the finest mesh resolution result obtained
    in~\cite{ElfPNASRates2010} using their modified RDME model, while
    the purple dash-dot line indicates the result from BD simulations
    of the uniform unbinding model~\cref{eq:Doikm}. The dashed black 
    and blue lines correspond to a 95$\%$ confidence interval for the
    mean in the point and uniform unbinding CRDME simulations
    respectively. The CRDME and BD curves were estimated from 100
    simulations, while the red line was generated by estimating the
    data points in Fig~4B of~\cite{ElfPNASRates2010}.  The mesh used
    for all CRDME simulations was the same as described in
    \cref{fig:compareBD}. All BD simulations used a time-step of
    $dt = 10^{-10}$s.}
  \label{fig:chemSys}
\end{figure}
\begin{figure}[!tbp]
  \centering
  \subfloat[Point unbinding]{
    \label{fig:histPointUnbind}
    \includegraphics[width=.48\textwidth]{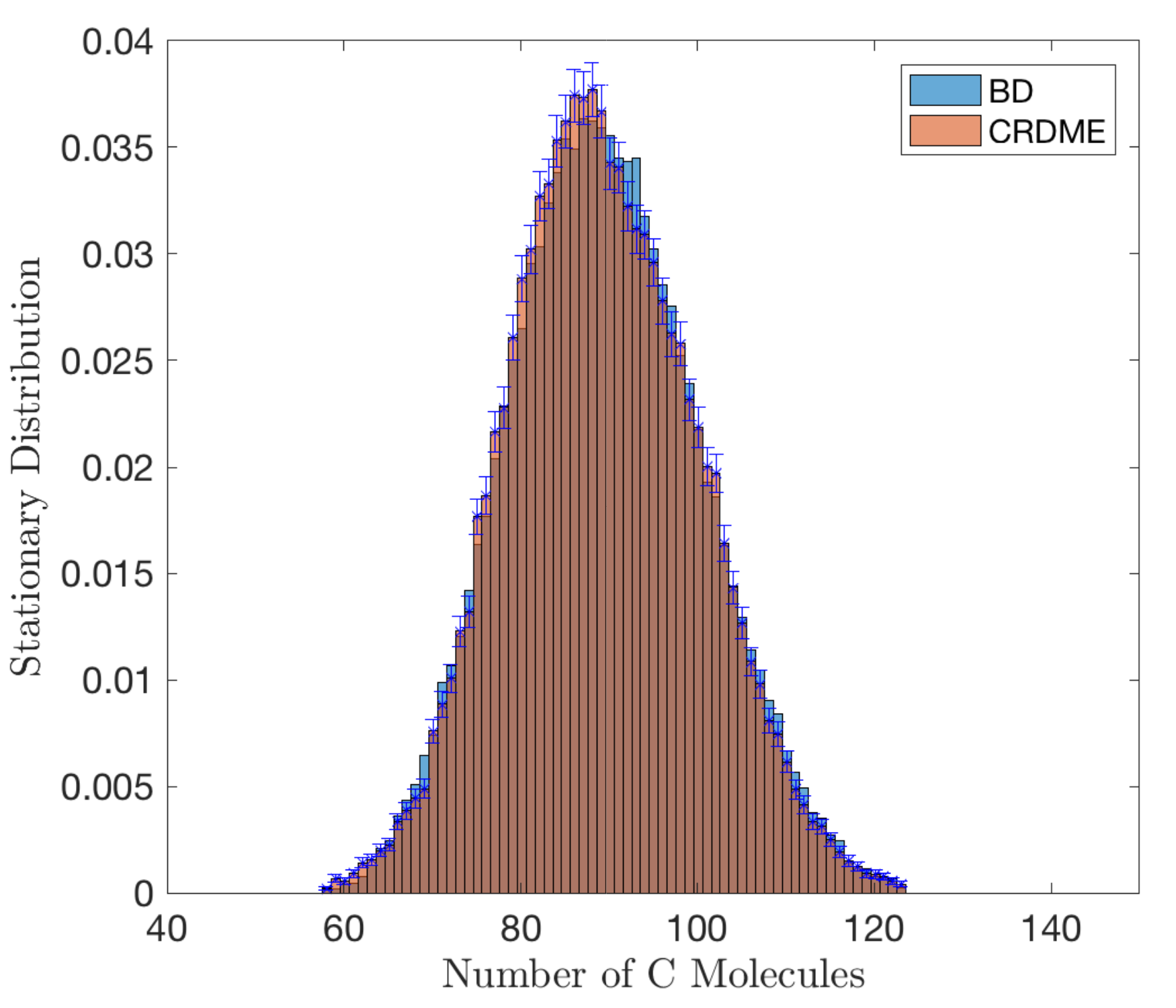}
  }
  \subfloat[Uniform unbinding]{
    \label{fig:histUniformUnbind}
    \includegraphics[width=.48\textwidth]{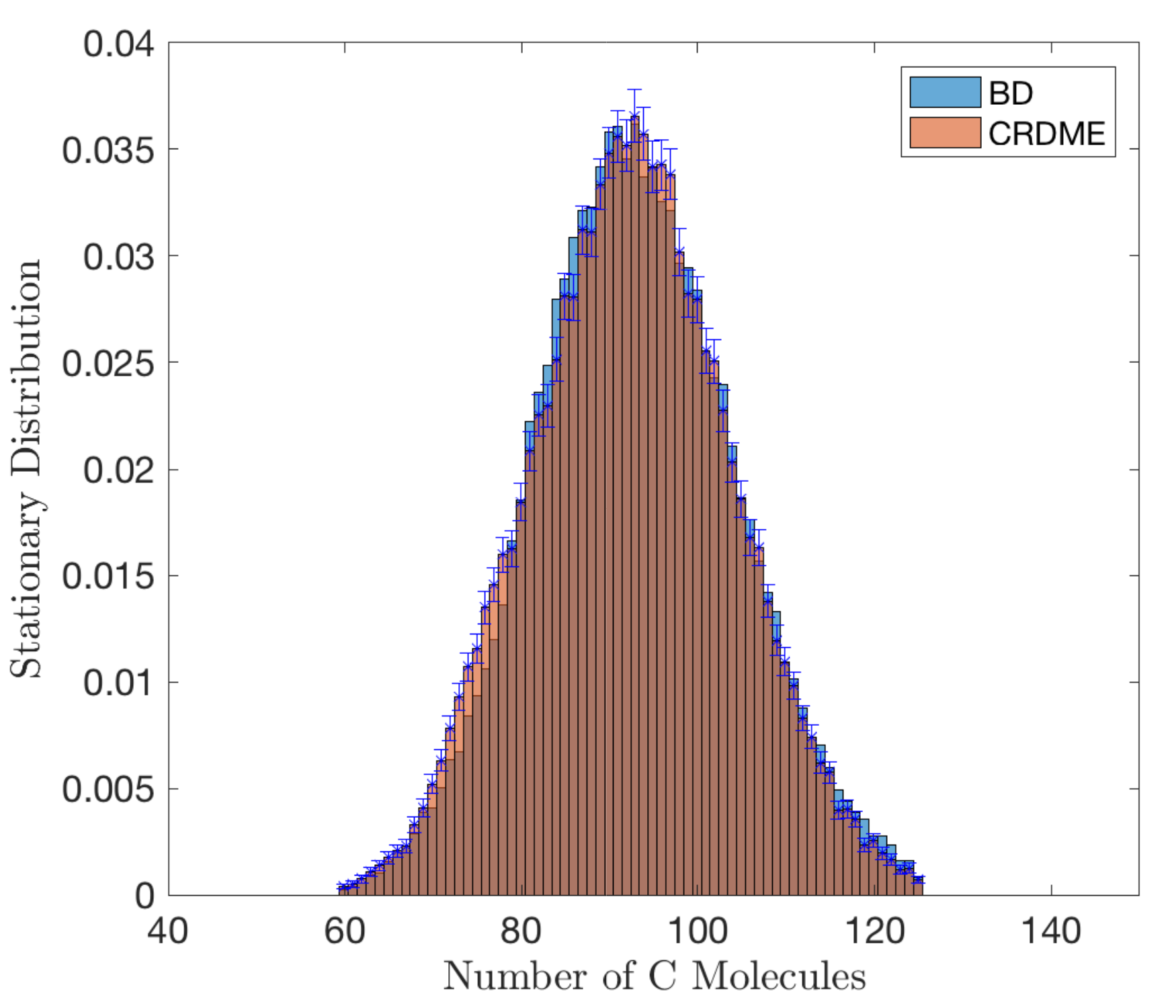}    
  }
  \caption{Histogram of the empirical stationary distribution for the
    number of \textrm{C} molecules obtained from 60000 CRDME and BD
    simulations of~\cref{eq:chemSysm}. Parameters are the same as in
    \cref{fig:chemSys}. Panel (a) corresponds to using the point
    unbinding model~\cref{eq:ptUnbind} in both the CRDME and BD
    simulations. Panel (b) corresponds to using the uniform unbinding
    model~\cref{eq:Doikm} in both the CRDME and BD simulations. 95$\%$
    confidence intervals for CRDME are drawn in blue. For each
    unbinding model, the CRDME simulations agree with the Brownian
    dynamics simulations to statistical error.}
    \label{fig:histCRDME}
\end{figure}
The reaction system described by equation~3 of~\cite{ElfPNASRates2010}
is
\begin{equation}
  A + B \underset{\mu}{\stackrel{\lambda}{\rightleftharpoons}} C,\:\:\:\:\:\:\: 
  \emptyset \xrightarrow{k_1} C,\:\:\:\:\:\:\: A \xrightarrow{k_2} 
  \emptyset,\:\:\:\:\:\:\: B \xrightarrow{k_2} \emptyset.
  \label{eq:chemSysm}
\end{equation}
Here the reaction radius, $\rb$, is again chosen to be $10^{-3}\mu$m,
the reaction rate was chosen to be
$\lambda = 1.0056 \times 10^8s^{-1}$ and the dissociation rate was
chosen to be $\mu = 3.1621 \times 10^4s^{-1}$. Parameters are
calibrated as described in the preceding subsection using the
parameter relations established in~\cite{ElfPNASRates2010}.  For our
CRDME simulations, the domain $\Omega$ was discretized into 263169
mesh voxels with maximum mesh width $h = 3.5334\times 10^{-5} \mu m$,
as detailed in \cref{fig:compareBD}. In \cref{fig:chemSys}, we plot
the time evolution of the average number of C molecules as found
in~\cite{ElfPNASRates2010}, as determined from BD simulations using
the uniform unbinding mechanism~\cref{eq:Doikm}, and from CRDME
simulations using both the point~\cref{eq:ptUnbind} and
uniform~\cref{eq:Doikm} unbinding mechanisms. The estimated average
number of \textrm{C} molecules agreed quite well between all four
methods (when averaged over 100 simulations). \Cref{fig:histCRDME}
shows the corresponding stationary distribution of the number of
\textrm{C} molecules in the CRDME and BD simulations for each
unbinding model. In both cases the stationary distributions agree to
the level of statistical error.

We conclude by demonstrating the ability of our method to handle
complex geometries by considering a simplified version of a signaling
cascade model~\cite{MunozGarcia:2009hd} within a two-dimensional
domain corresponding to the cytosol of a human B cell (reconstructed
from an X-ray tomogram~\cite{IsaacsonBmB2013}). The simplified
(three-level) signaling cascade model is given by
\begin{equation}
  \begin{aligned}
    \text{first level: } &c_1^u \xrightarrow{k_1^a} c_1^p,  &\text{(cell membrane)} \\
      &c_1^p \xrightarrow{k_1^i} c_1^u, \\
    \text{second level: } &c_2^u + c_1^p \xrightarrow{k_2^a} c_2^p + c_1^p, \\
                          &c_2^p \xrightarrow{k_2^i} c_2^u, \\
    \text{third level: } &c_3^u + c_2^p \xrightarrow{k_3^a} c_3^p + c_2^p, \\
                          &c_3^p \xrightarrow{k_3^i} c_3^u,
  \end{aligned}
  \label{eq:sigCasd}
\end{equation}
where the first level phosphorylation reaction occurs only in the cell
membrane. The diffusion of molecules and all other reactions occur
within the cytosol.

\begin{figure}[tbp]
  \centering
  \subfloat[Reconstructed boundaries]{
    \label{fig:bcell_boundary}
    \includegraphics[width=.48\textwidth]{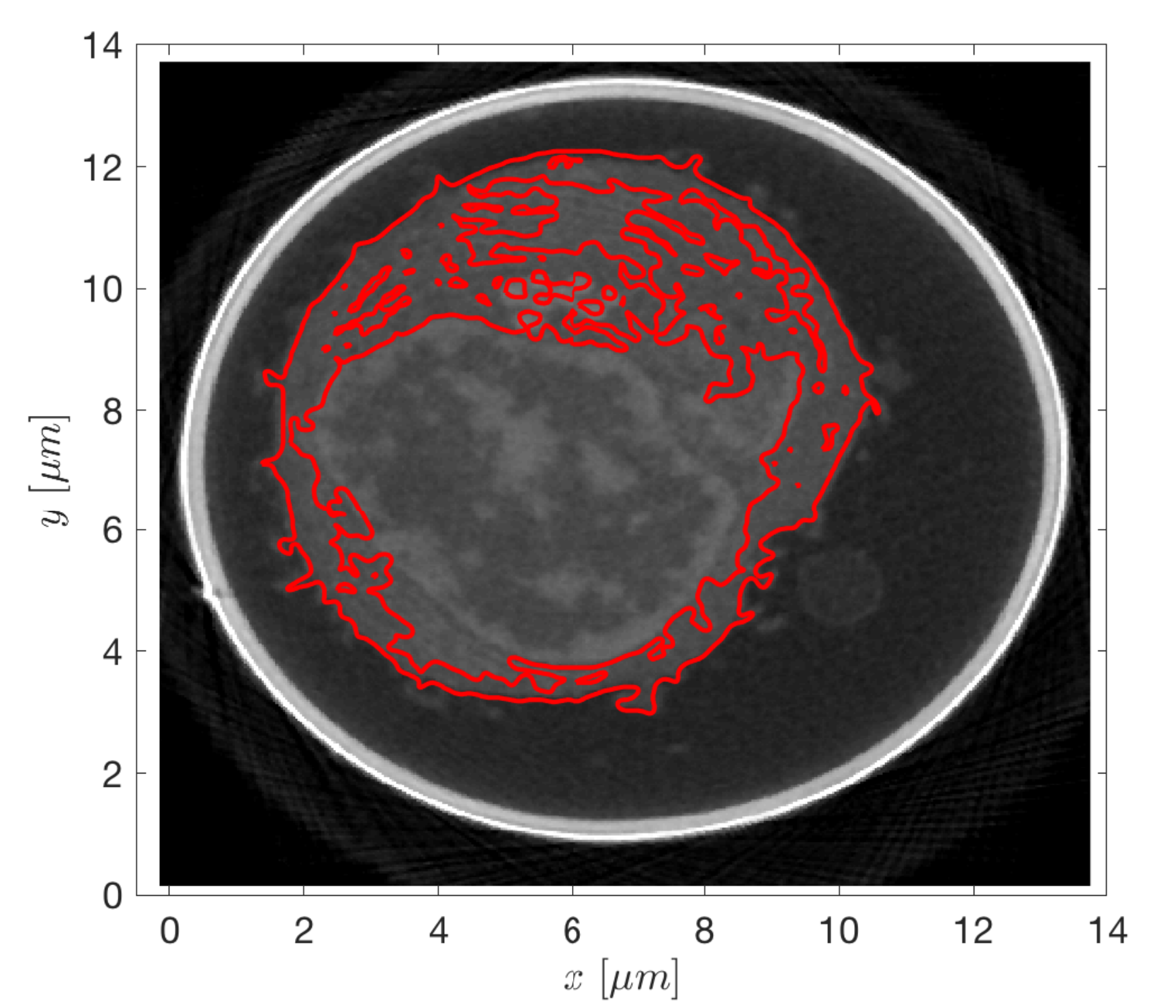}
  }
  \subfloat[The mesh of the human B cell]{
    \label{fig:bcellmesh}
    \includegraphics[width=.48\textwidth]{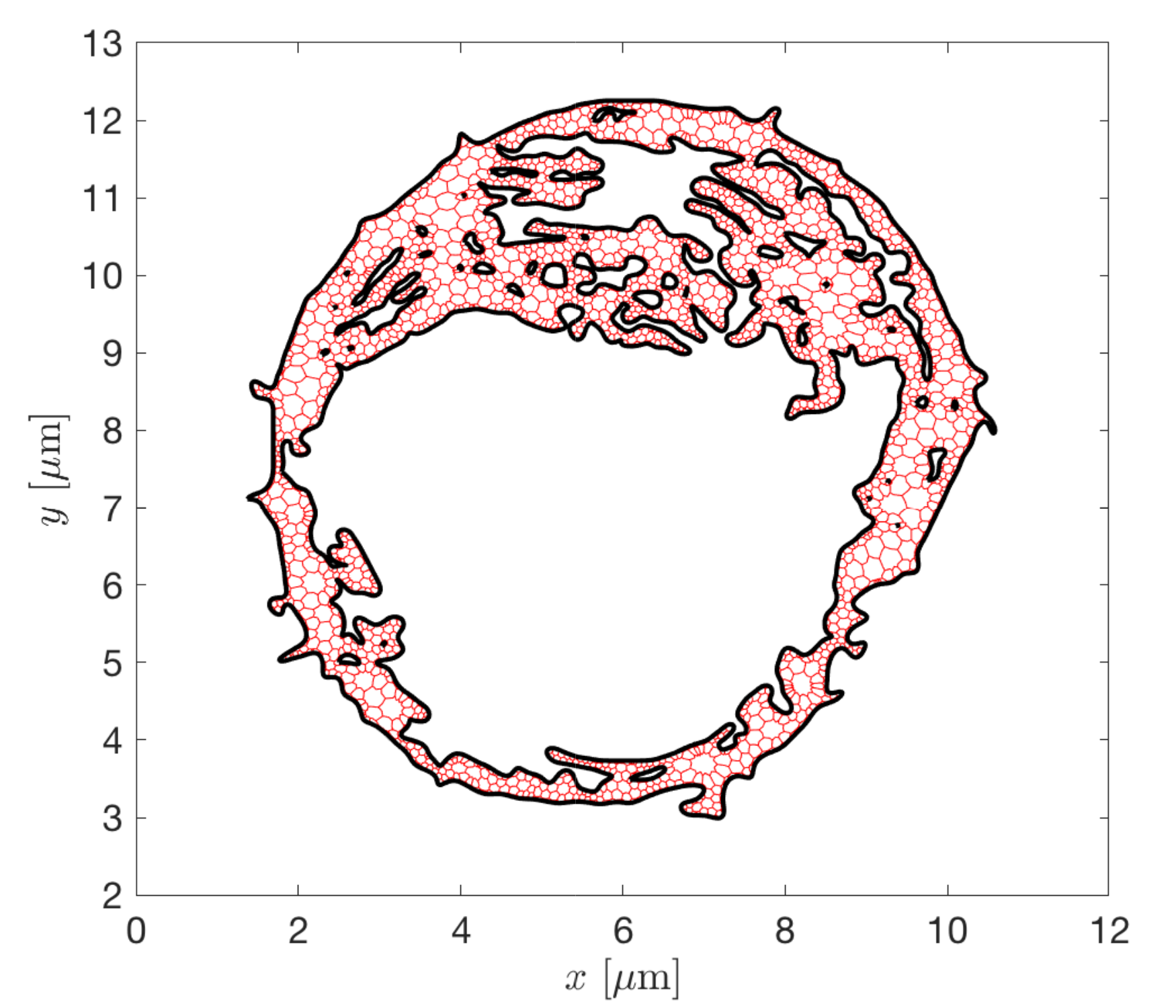}     
  }
  \caption{A two-dimensional slice of an X-ray tomogram of a human B
    cell. In panel (a) we plot the reconstructed boundaries (red solid
    lines) on top of the original imaging data. In panel (b) we show
    the mesh of the human B cell used in the CRDME simulation of the
    signaling cascade model ~\cref{eq:sigCasd}. The maximum connected
    region is discretized into $6915$ polygonal elements.}
    \label{fig:humanBcell}
\end{figure}
We extracted a two-dimensional slice of a human B cell from a
(labeled) three-dimensional X-ray tomogram provided by the National
Center for X-ray Tomography (NCXT)~\cite{NCXT}. The boundaries of the
cell, the nucleus, and cytosolic organelles were then segmented using
the \texttt{bwboundaries} command in MATLAB. The maximum connected
region from the resulting segmented cytosol was determined
(\cref{fig:bcell_boundary}), and that domain was then triangulated
in Gmsh~\cite{GMSH2009}. The corresponding dual mesh was then
calculated, providing a polygonal mesh approximation to the cytosol we
used in CRDME simulations (\cref{fig:bcellmesh}). For
demonstrative purposes, we choose a coarse mesh with a maximum dual
mesh diameter of $121.4$nm and a average dual mesh diameter of
$15.9$nm.

\begin{figure}[tbp]
  \centering
  \subfloat[$t=0.01s$]{
    \label{fig:sigCasdBcell1}
    \includegraphics[scale=0.55]{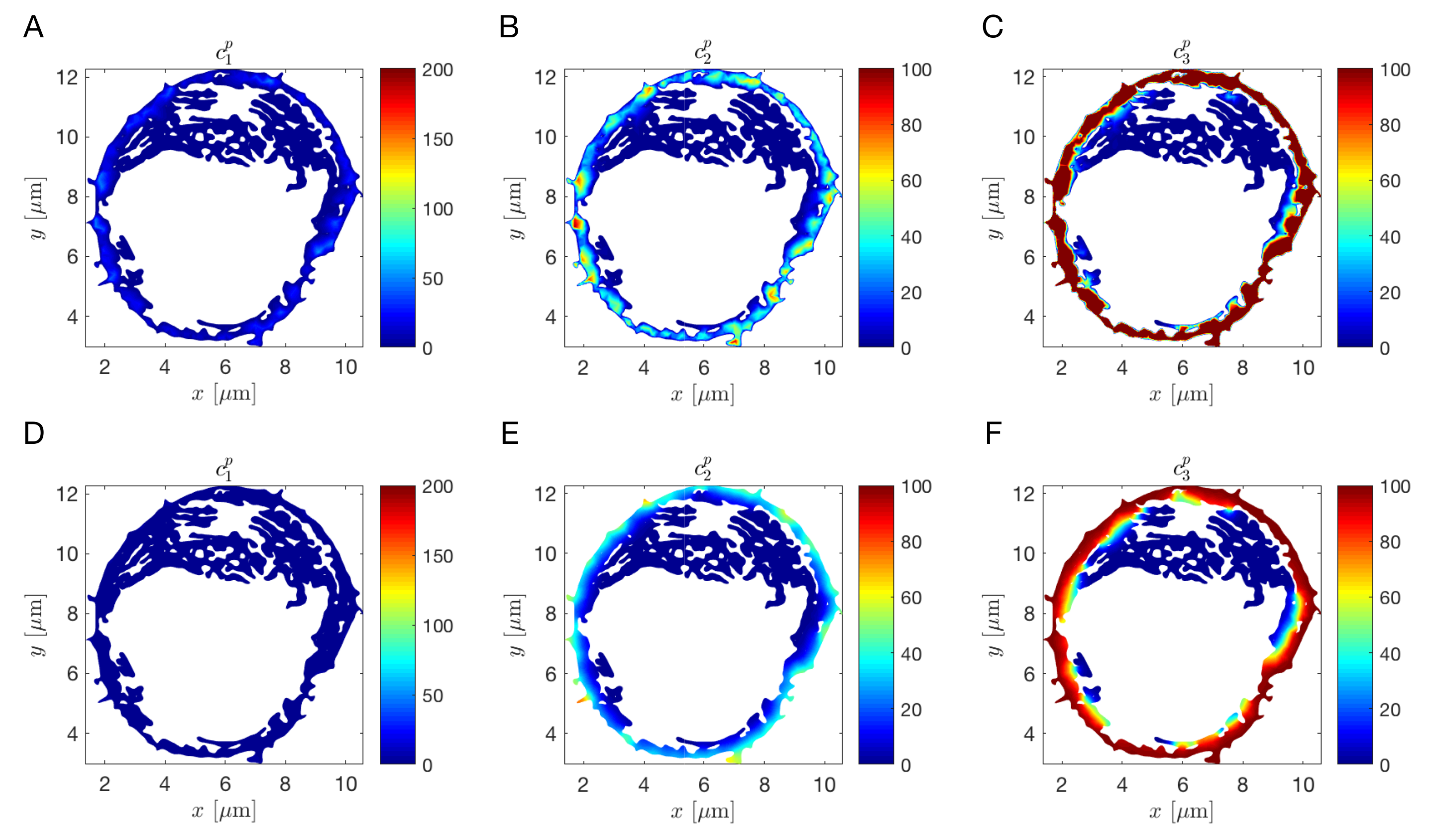}
  }
  \\
  \subfloat[Approximate steady-state, $t=0.4s$]{
    \label{fig:sigCasdBcellss} 
    \includegraphics[scale=0.55]{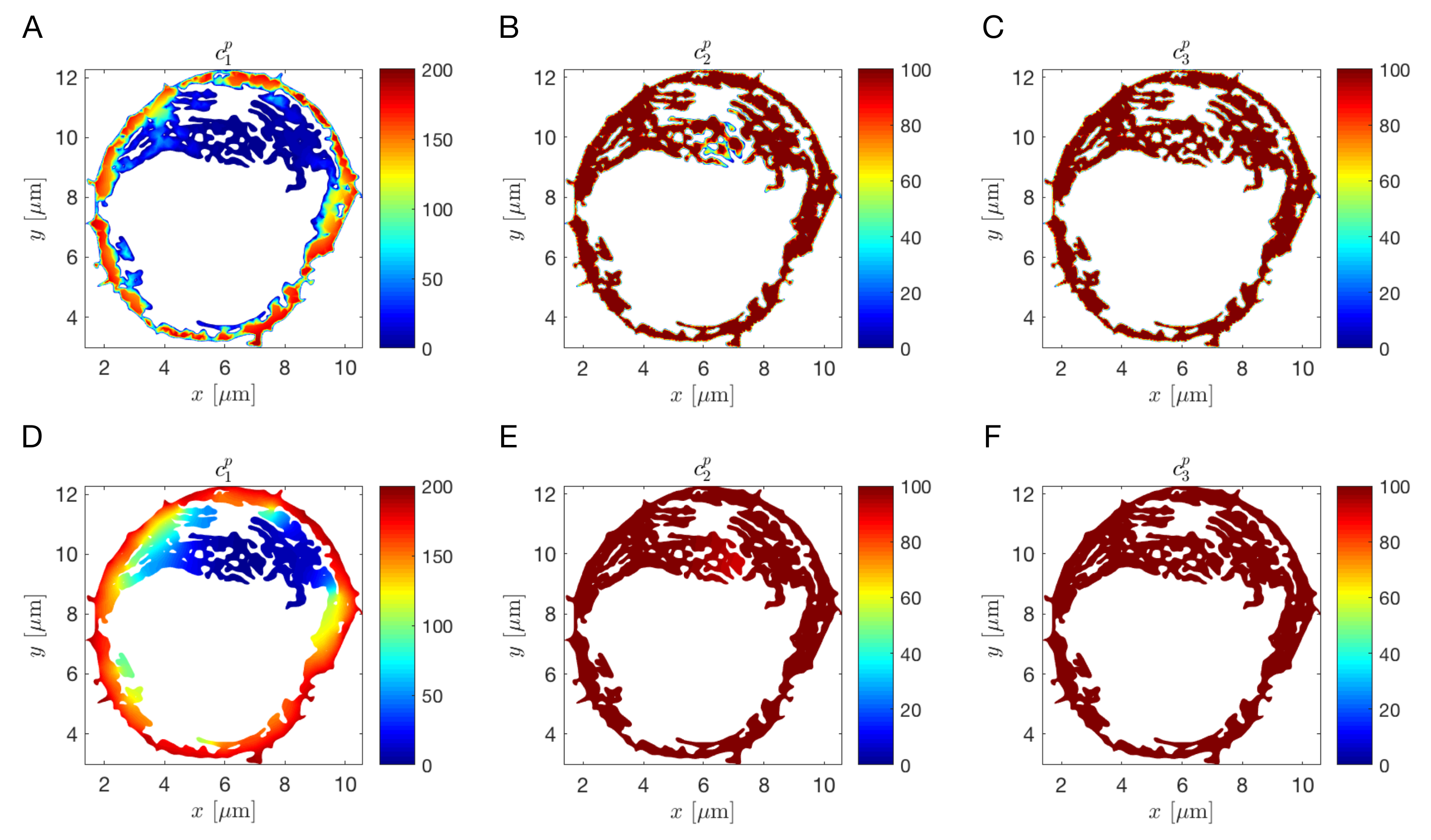}
  }
  \caption{Phosphorylated form ($c_n^p$, $n = 1, 2, 3$) profiles at
    $t = 0.01s$ and at steady-state ($t=0.4s$). In panels A, B, and C
    we plot the result from one CRDME simulation. In panels D, E, and
    F we plot the result by numerically solving the PDE corresponding
    to \cref{eq:sigCasd} using the finite element method in space and
    backward-Euler method in time.}
\end{figure}
We assumed no-flux Neumann boundary conditions on the boundary of the
cell, the nucleus, and all organelles within the cytosol, with a
diffusion constant of $D = 5 \mu$m$^{2}s^{-1}$ for each species.  The
bimolecular reaction radius $\rb$ was chosen to be $10\text{nm}$. The
reaction rates were chosen to be consistent with those used
in~\cite{MunozGarcia:2009hd}. The dephosphorylation rates, $k_n^i$,
are set to be $5s^{-1}$ for $n = 1, 2, 3$. The phosphorylation rates,
$k_n^a$, are set to be $50s^{-1}$ for $n = 2, 3$, and the
phosphorylation rate for the first species, $k_1^a$, is set to be
$5s^{-1}$. We note that the first-level kinase phosphorylation only
occurs in the cell membrane. In the simulation, this reaction process
was restricted to voxels bordering the boundary of the cell. Finally,
we initialized the system with no $c_n^p$ for $n = 1, 2, 3$, $200$
$c_1^u$ molecules and $100$ $c_2^u$ and $c_3^u$ molecules in each
voxel.

In \cref{fig:sigCasdBcell1,fig:sigCasdBcellss} we show the
phosphorylated form ($c_n^p$, $n = 1, 2, 3$) profiles at $t = 0.01s$
and $t = 0.4s$ (steady state) of the system from one CRDME simulation
in comparison to the numerical solution of the PDE corresponding to
\cref{eq:sigCasd}. In both cases, we observe initially a rapid
infusion of $c_2^p$ and $c_3^p$ starting from the cell membrane
(\cref{fig:sigCasdBcell1}), and eventually $c_2^p$ and $c_3^p$ become
uniformly distributed in the cell.

\section{\commentb{Discussion}}
By using a finite volume discretization of reaction terms, we have
developed a convergent lattice jump process approximation, the
convergent reaction-diffusion master equation (CRDME), to the abstract
volume-reactivity model for reversible reactions. The final CRDME can
handle general bimolecular interaction functions on both structured
and unstructured polygonal grids, including the popular Doi reaction
model. The flexibility of the CRDME approach allows the reuse of a
variety of methods for approximating the spatial diffusion of
molecules developed for the RDME model. These include the unstructured
grid finite element approach we used here (developed
in~\cite{LotstedtFERDME2009}), but also the Cartesian grid cut cell
finite volume method of~\cite{IsaacsonSJSC2006} and the unstructured
grid finite volume method of~\cite{DeSchutterSTEPS2012uh}.  As we
demonstrated in \cref{S:numericExs}, this enables the use of the CRDME
to study complex particle-based stochastic reaction-diffusion models
within realistic domain geometries arising from cellular imaging data.

One benefit to the CRDME approach is that it is equivalent to the
popular RDME model in its treatment of spatial transport
(e.g. diffusion) and linear reactions, while converging to an
underlying spatially-continuous particle model in the limit that the
mesh spacing approaches zero (whereas in two or more dimensions the
RDME loses bimolecular reactions in such limits). This enables
CRDME-based models to immediately reuse many of the extensions to the
RDME that have been developed to optimize simulation performance, and
extend the RDME to more general spatial transport
mechanisms. Moreover, as the RDME approaches the CRDME as the mesh is
coarsened~\cite{IsaacsonCRDME2013}, this suggests a possible adaptive
mesh refinement method. The CRDME could be used in regions where fine
meshes are required and the RDME loses bimolecular reaction effects,
while the RDME (or a renormalized
RDME~\cite{HellanderJumpRatesUnstructMesh2017ed}), is used in regions
where coarse meshes are acceptable. Constructing such hybrid models is
relatively straightforward since the RDME and CRDME differ only in the
set of jump process transitions they use to model bimolecular
reactions.

There are still a number of directions in which the CRDME could be
improved. Foremost is the development of optimized simulation methods
for the jump process associated with the CRDME; in this work we used a
simple variant of the method for well-mixed reactions developed
in~\cite{GibsonBruckJPCHEM2002}. \commentab{Once such optimized
  simulation methods are available, it would be interesting to compare
  the computational work to achieve a given accuracy with the CRDME to
  that required by Brownian Dynamics methods for solving the Doi
  model, and potentially to that required by renormalized RDMEs that
  approximate the Doi model. Note, the latter would require first
  developing renormalized RDME approximations to the Doi model as
  current approaches have focused on approximating the
  Smoluchowski-Collins-Kimball
  model~\cite{HellanderJumpRatesUnstructMesh2017ed}.}

We have also not discussed how to evaluate the reactive jump rate
integrals for $\kpij$ and $\kpijk$ in three-dimensional domains. While
the method we developed in \cref{ap:hypervolCalc} for evaluating such
integrals in two-dimensions is straightforward, it is more complicated
in three-dimensions, where one must rapidly and accurately evaluate
volumes of intersection between spheres and polyhedra. In limited
testing we have found that the recently developed primitive
intersection library from~\cite{StroblVolIsect2016bh} offers good
performance and accuracy in many cases. It can then be wrapped within
adaptive quadrature routines to evaluate integrals of volumes of
intersection over polyhedra in a similar manner to how we wrapped our
2D area of intersection method to evaluate integrals over polygons in
\cref{ap:hypervolCalc}.


\appendix

\section{General Multi-particle $\textrm{A} + \textrm{B} \leftrightarrows \textrm{C}$ Reaction}
\label{ap:multipartModel}
In this appendix we consider the general multiparticle
$\textrm{A} + \textrm{B} \leftrightarrows \textrm{C}$ reaction within
a bounded domain $\Omega \subset \R^d$. We will show that the natural
generalization of the discretization procedure used in
\cref{S:diffApprox,S:rxDiscretMethod} to derive the
master equation approximation~\cref{eq:twopartCRDME} also leads to a
master equation approximation for the general multiparticle case. This
new master equation corresponds to the forward Kolmogorov equation for
the jump process given in \cref{tab:aAndBToCRxs}, demonstrating
that the multiparticle system's transition rates are simply the
two-particle transition rates multiplied by the number of possible
ways each transition event can occur given the current system state.

We first formulate the general multiparticle abstract volume
reactivity model, using a similar notation
to~\cite{DoiSecondQuantA,DoiSecondQuantB,Lachowicz2011}.  Denote by
$A(t)$ the stochastic process for the number of species \textrm{A}
molecules in the system at time $t$, with $B(t)$ and $C(t)$ defined
similarly. Values of $A(t)$, $B(t)$ and $C(t)$ will be given by $a$,
$b$ and $c$ (i.e. $A(t)=a$). When $A(t)=a$, we will let
$\vQa_l(t) \in \Omega$ label the stochastic process for the position
of the $l$th molecule of species \textrm{A} within the domain
$\Omega$. $\vqa_l$ will denote a possible value of $\vQa_l(t)$.  The
species \textrm{A} position vector when $A(t)=a$ is then given by
\begin{equation*}
  \vQa(t) = (\vQa_1(t), \dots, \vQa_a(t)) \in \Omega^{a},
\end{equation*}
where $\Omega^{a} = \Omega \times \dots \times \Omega \subset \R^{da}$.
Similarly, $\vqa$ will denote a possible value of $\vQa(t)$,
\begin{equation*}
  \vQa(t) = \vqa = (\vqa_1,\dots,\vqa_a).
\end{equation*}
$\Omega^b$, $\Omega^c$, $\vQb(t)$, $\vQb_m(t)$, $\vQc(t)$,
$\vQc_n(t)$, $\vqb_m$, $\vqc_n$, $\vqb$ and $\vqc$ will all be defined
analogously. The state of the system is then a hybrid
discrete--continuous state stochastic process given by
$\paren{A(t), B(t), C(t),\vQ^{A(t)}, \vQ^{B(t)}, \vQ^{C(t)}}$.

With this notation, denote by $\fabc(\vqa,\vqb,\vqc,t)$ the
probability density that $A(t)=a$, $B(t)=b$ and $C(t)=c$ with
$\vQa(t)=\vqa$, $\vQb(t)=\vqb$ and $\vQc(t) = \vqc$. We assume that
molecules of the same species are indistinguishable, that is for
$1 \leq l< l' \leq a$ fixed
\begin{multline*}
\fabc \paren{\vqa_1,\dots,\vqa_{l-1},\vqa_l,\vqa_{l+1},\dots,\vqa_{l'-1},\vqa_{l'},\vqa_{l'+1},
  \dots,\vqa_a,\vqb,\vqc,t} \\
= \fabc \paren{\vqa_1,\dots,\vqa_{l-1},\vqa_{l'},\vqa_{l+1},\dots,\vqa_{l'-1},\vqa_{l},\vqa_{l'+1},
  \dots,\vqa_a,\vqb,\vqc,t},
\end{multline*}
with similar relations holding for permutations of the molecule
orderings within $\vqb$ and $\vqc$.  With this assumption the
$f^{(a,b,c)}$ are chosen to be normalized so that
\begin{equation*}
  \sum_{a=0}^{\infty} \sum_{b=0}^{\infty} \sum_{c=0}^{\infty} \brac{ \frac{1}{a!\, b!\, c!}
  \int_{\Omega^{a}} \int_{\Omega^{b}} \int_{\Omega^{c}} \fabc\paren{\vqa,\vqb,\vqc,t} \, d\vqc \, d\vqb \, d\vqa
  } = 1.
\end{equation*}
Here the bracketed term corresponds to the probability of having a
given number of each species, i.e.
\begin{equation*}
  \prob \brac{A(t) = a, B(t) = b, C(t) = c}  = \frac{1}{a!\, b!\, c!}
  \int_{\Omega^{a}} \int_{\Omega^{b}} \int_{\Omega^{c}} \fabc\paren{\vqa,\vqb,\vqc,t} \, d\vqc \, d\vqb \, d\vqa.
\end{equation*}

We denote by $\vf(t)$ the overall probability density vector,
\begin{equation*}
  \vf(t) = \{ f^{(a,b,c)}(\vq^a,\vq^b,\vq^c)\}_{a,b,c},
\end{equation*}
so that the component of $\vf(t)$ indexed by $(a,b,c)$ is
$f^{(a,b,c)} (\vq^a,\vq^b,\vq^c)$. The density vector satisfies the
forward Kolmogorov equation
\begin{equation}
  \label{eq:multipartABtoCEqs}
  \PD{\vf}{t}(t) = (\diffop + \Rp + \Rm) \vf(t).
\end{equation}
We assume molecules can not leave $\Omega$, so that each component
$\fabc(\vqa,\vqb,\vqc,t)$ also satisfies a reflecting Neumann boundary
condition on the appropriate domain boundary,
$\partial \paren{\Omega^{a+b+c}}$.  Here the linear operators
$\diffop$, $\Rp$ and $\Rm$ correspond to diffusion, the forward
association reaction and the reverse dissociation reaction
respectively. The $(a,b,c)$ component of the diffusion operator is
given by
\begin{equation}
  \label{eq:multipartDiffop}
  (\diffop \vf(t))_{a,b,c} = \paren{D^{\textrm{A}} \sum_{l=1}^{a} \lap_{\vqa_l} 
    + D^{\textrm{B}} \sum_{m=1}^{b} \lap_{\vqb_m}
    + D^{\textrm{C}} \sum_{n=1}^{c} \lap_{\vqc_n}} \fabc(\vqa,\vqb,\vqc,t),
\end{equation}
where $\lap_{\vqa_l}$ denotes the $d$-dimensional Laplacian acting on
the $\vqa_l$ coordinate, and $\lap_{\vqb_m}$ and $\lap_{\vqc_n}$ are
defined similarly. To define the reaction operators, $\Rp$ and $\Rm$,
we introduce notations for adding or removing a molecule from a given
state, $\vqa$. Let
\begin{align*}
  \vqa \cup \vx &= \paren{ \vqa_1, \dots, \vqa_a, \vx}, & 
  \vqa \setminus \vqa_l &= \paren{ \vqa_1,\dots,\vqa_{l-1}, \vqa_{l+1}, \dots, \vqa_a},
\end{align*}
which correspond to adding a molecule to species \textrm{A} at $\vx$,
and removing the $l$th molecule of species \textrm{A} respectively.
With these definitions, the reaction operator for the association
reaction $\textrm{A} + \textrm{B} \to \textrm{C}$ is given by
\begin{equation} \label{eq:multipartRp}
  \begin{aligned} 
    (\Rp \vf(t))_{a,b,c} = 
    &-\paren{\sum_{l=1}^a \sum_{m=1}^b \kp \paren{\vqa_{l}, \vqb_{m}}} \fabc(\vqa,\vqb,\vqc,t) \\
    &+ \sum_{n=1}^{c} \!\brac{\int_{\Omega^2} \kp(\vqc_n \vert \vx, \vy) 
      f^{(a+1,b+1,c-1)}(\vqa \cup \vx, \vqb \cup \vy, \vqc \setminus \vqc_n, t) d\vx d\vy}\!\!.
  \end{aligned}
\end{equation}
Similarly, the reaction operator for the dissociation reaction
$\textrm{C} \to \textrm{A} + \textrm{B}$ is given by
\begin{multline} \label{eq:multipartRm}
  (\Rm \vf(t))_{a,b,c} = - \paren{ \sum_{n=1}^{c} \km(\vqc_n)} \fabc(\vqa, \vqb, \vqc, t) \\
  + \sum_{l=1}^a \sum_{m=1}^b \brac{\int_{\Omega} \km\paren{\vqa_l,\vqb_m\vert\vz}
    f^{(a-1,b-1,c+1)}\paren{\vqa\setminus \vqa_l, \vqb\setminus \vqb_m, \vqc \cup \vz, t} d\vz}.
\end{multline}

To approximate~\cref{eq:multipartABtoCEqs} by a master equation model
we reuse the jump process discretizations for diffusion and reaction
developed in \cref{S:diffApprox,S:methodSect}. The
former is possible since the diffusive motion of each molecule is
independent, so we can apply the finite element discretization of
\cref{S:diffApprox} to each individual molecule's position
coordinates. Similarly, each reaction term in the $\Rp$ and $\Rm$
definitions involves sums of independent one or two-body interactions
that are identical to those in~\cref{eq:pPDE}
and~\cref{eq:pbPDE}. As such, we may re-use the finite volume
discretization of \cref{S:rxDiscretMethod} independently for
each of these terms.

We again let $\{V_i\}_{i=1}^K$ denote a polygonal mesh approximation
to $\Omega$, constructed as the dual mesh to a triangulation of
$\Omega$. Reusing the notation of \cref{S:diffApprox,S:methodSect} we
let $\bi^a = (i^a_1,\dots,i^a_a)$ denote the multi-index labeling the
hyper-voxel $\vV_{\bi^a}$
\begin{equation*}
  \vV_{\bi^a} := V_{i^a_1} \times \dots \times V_{i^a_a},
\end{equation*}
with $\bj^b$, $\bk^c$, $V_{\bj^b}$ and $V_{\bk^c}$ defined
similarly. Multi-species hypervoxels will be given by
$\vV_{\bi^a \bj^b} := \vV_{\bi^a} \times \vV_{\bj^b}$ and
$\vVabc := \vV_{\bi^a} \times \vV_{\bj^b} \times \vV_{\bk^c}$.  We
then make the piecewise constant (well-mixed) approximation that
\begin{subequations}
  \begin{align}
    \Fabcijk(t) &:= \prob \brac{A(t) = a, B(t) = b, C(t) = c \text{ and } 
                  \paren{\vQa(t), \vQb(t), \vQc(t)} \in \vVabc} \notag \\ 
                &\phantom{:}= \frac{1}{a! b! c!} \int_{\vVabc} \fabc(\vqa,\vqb,\vqc,t) \, d\vqc \, d\vqb \, d\vqa \notag \\
                &\phantom{:}\approx \fabc(\vqa_{\bi^a}, \vqb_{\bj^b}, \vqc_{\bk^c},t) \frac{\abs{\vVabc}}{a! \, b! \, c!}, \label{eq:FabcijkDef}
  \end{align}
\end{subequations}
where $\vqa_{\bi^a}$ denotes the vector of the centroids of the voxels
in $V_{\bi^a}$, with $\vqb_{\bj^b}$ and $\vqc_{\bk^c}$ defined
similarly. We collect the probabilities $\Fabcijk(t)$ into a state
vector
\begin{equation*}
  \vF(t) = \{ \Fabcijk(t) \}_{\bi^a,\bj^b,\bk^c},
\end{equation*}
so that the $(\bi^a,\bj^b,\bk^c)$ component of $\vF(t)$ is
$\Fabcijk(t)$.  By using the finite element discretization of
\cref{S:diffApprox} to approximate each Laplacian within the diffusion
operator, integrating the action of each of the reaction operators,
$\diffop$, $\Rp$, and $\Rm$ on $\fabc$ over $\vVabc$, and using the
well-mixed (piecewise constant) approximation~\cref{eq:FabcijkDef}, we
find that $\vF(t)$ satisfies the master equation
\begin{equation*}
  \D{\vF}{t}(t) = \paren{\diffoph + \Rph + \Rmh} \vF(t),
\end{equation*}
corresponding to the forward Kolmogorov equation for a jump process.
Here the discretized diffusion operator $\diffoph$ is given by
\begin{align*}
  (\diffoph \vF(t))_{\bi^a,\bj^b,\bk^c} 
  &= D^{\textrm{A}} \sum_{l=1}^{a} \sum_{i'= 1}^{K} \brac{ \paren{\Delta_h^T}_{i^a_l i'} F_{\bi^a\setminus i^a_l\cup i', \bj^b, \bk^c}(t)
    - \paren{\Delta_h^T}_{i' i^a_l } \Fabcijk(t)} \\
  &+ D^{\textrm{B}} \sum_{m=1}^{b} \sum_{j'= 1}^{K} \brac{ \paren{\Delta_h^T}_{j^b_m j'} F_{\bi^a, \bj^b\setminus j^b_m\cup j', \bk^c}(t)
    - \paren{\Delta_h^T}_{j' j^b_m } \!\Fabcijk(t)} \\
  & + D^{\textrm{C}} \sum_{n=1}^{c} \sum_{k'= 1}^{K} \brac{ \paren{\Delta_h^T}_{k^c_n k'} F_{\bi^a,\bj^b,\bk^c\setminus k^c_n\cup k'}(t)
    - \paren{\Delta_h^T}_{k' k^c_n } \!\Fabcijk(t)}.
\end{align*}
Similarly, the forward reaction association operator $\Rph$
is given by
\begin{equation*}
  (\Rph \vF(t))_{\bi^a,\bj^b,\bk^c} =
  -\paren{\sum_{l=1}^a \sum_{m=1}^b \kp_{i^a_l j^b_m} } \Fabcijk(t)  
  + \frac{(a+1)(b+1)}{c} \sum_{n=1}^{c} \sum_{i'=1}^{K} \sum_{j'=1}^{K}  
  \kp_{i' j' k_n^c} F_{\bi^a \cup i', \bj^b \cup j', \bk^c \setminus k_n^c}(t),
\end{equation*}
and the backward reaction dissociation operator $\Rmh$ is given by
\begin{equation*}
  (\Rmh \vF(t))_{\bi^a,\bj^b,\bk^c} = 
  - \paren{ \sum_{n=1}^{c} \km_{k_n^c}} \Fabcijk(t) 
  + \frac{c+1}{a b} \sum_{l=1}^a \sum_{m=1}^b \sum_{k'=1}^K \km_{i_l^a j_m^b k'} 
  F_{\bi^a \setminus i_l^a, \bj^b \setminus j_m^b, \bk^c \cup k'}(t). 
\end{equation*}

We wish to now convert from a representation where the state variables
are the numbers of molecules of each species and the indices of the
voxels that contain each molecule, to a representation where the state
variables are the numbers of molecules of each species at each lattice
site. While equivalent, this latter representation is more commonly
used for the reaction-diffusion master equation (RDME), and allows for
the easy identification of the effective transition rates listed in
\cref{tab:aAndBToCRxs}. Here we summarize how to convert between the
two forms, and refer readers interested in detailed derivations to the
near-identical conversion that we previously carried out
in~\cite{IsaacsonRDMENote}.

Denote by $A_i(t)$ the stochastic process for the number of molecules
of species \textrm{A} in voxel $V_i$ at time $t$, with $B_j(t)$ and
$C_k(t)$ defined similarly. By $a_i$, $b_j$ and $c_k$ we denote values
of $A_i(t)$, $B_j(t)$ and $C_k(t)$, i.e. $A_i(t) = a_i$. For each
species the new representation of the state at time $t$ is given by
the vector stochastic process
\begin{equation*}
  \vA(t) = \paren{A_1(t),\dots,A_K(t)},
\end{equation*}
with $\vB(t)$ and $\vC(t)$ defined similarly. The vectors $\va$, $\vb$
and $\vc$ will then denote corresponding values of these stochastic
processes. Finally, let
\begin{equation*}
 P(\va,\vb,\vc,t) := \prob \brac{ \vA(t)=\va, \vB(t) = \vb, \vC(t) = \vc}.
\end{equation*}
With these definitions, we now derive a master equation satisfied by
the probability distribution vector
$\vP(t) = \{P(\va,\vb,\vc,t)\}_{\va,\vb,\vc}$ from the master equation
for $\vF(t)$.

In the remainder, assume that the two representations of state are
chosen to be consistent, i.e. that $\va$ and $\bi^a$ are chosen so
that
\begin{equation*}
  a_i = \abs{ \{ i_l^a \vert i_l^a = i, l=1,\dots,a\}},
\end{equation*}
where $\abs{\cdot}$ denotes the cardinality of the set (with similar
relations for $b_j$ and $c_k$). With
\begin{equation} \label{eq:omegaDef}
  \omega_{\va,\vb,\vc} = \frac{a!\, b!\, c!}{\prod_{i=1}^K { a_i!\, b_i!\, c_i!}},
\end{equation}
we showed in~\cite{IsaacsonRDMENote} that
\begin{equation}
  \label{eq:partToNumRep}
  P(\va,\vb,\vc,t) = \omega_{\va,\vb,\vc} F_{\bi^a \bj^b \bk^c}(t).
\end{equation}
We therefore define the action of the diffusion and reaction operators
on $\vP(t)$ by
\begin{align*}
  (\diffoph \vP) (\va,\vb,\vc,t) &:= \omega_{\va,\vb,\vc} (\diffoph \vF(t))_{\bi^a,\bj^b,\bk^c}\\
  (\Rph \vP)(\va,\vb,\vc,t) &:= \omega_{\va,\vb,\vc} (\Rph \vF(t))_{\bi^a,\bj^b,\bk^c} \\
  (\Rmh \vP)(\va,\vb,\vc,t) &:= \omega_{\va,\vb,\vc} (\Rmh \vF(t))_{\bi^a,\bj^b,\bk^c},
\end{align*}
so that
\begin{equation} \label{eq:vPEq}
  \D{\vP}{t}(t) = \paren{\diffoph + \Rph + \Rmh} \vP(t).
\end{equation}

To explicitly characterize the action of each operator on
$P(\va,\vb,\vc,t)$, we use the symmetry of $\vF$ with respect to
components of each of $\bi^a$, $\bj^b$, and $\bk^c$
respectively. Symmetry implies that
\begin{equation} \label{eq:LhFrewrite}
  \begin{aligned} 
    (\diffoph \vF(t))_{\bi^a,\bj^b,\bk^c} 
    &= D^{\textrm{A}} \sum_{i= 1}^{K} \sum_{i'= 1}^{K} \brac{ \paren{\Delta_h^T}_{i i' } 
      a_{i} F_{\bi^a\setminus i \cup i', \bj^b, \bk^c}(t)
      - \paren{\Delta_h^T}_{i' i} a_{i} \Fabcijk(t)} \\
    &+ D^{\textrm{B}} \sum_{j= 1}^{K} \sum_{j'= 1}^{K} \brac{ \paren{\Delta_h^T}_{j j'} 
      b_{j} F_{\bi^a, \bj^b\setminus j \cup j', \bk^c}(t)
      - \paren{\Delta_h^T}_{j' j} b_{j} \Fabcijk(t)} \\
    & + D^{\textrm{C}} \sum_{k= 1}^{K} \sum_{k'= 1}^{K} \brac{ \paren{\Delta_h^T}_{k k'} 
      c_{k} F_{\bi^a,\bj^b,\bk^c\setminus k \cup k'}(t)
      - \paren{\Delta_h^T}_{k' k} c_{k} \Fabcijk(t)}.
  \end{aligned}
\end{equation}
Let $\ve_i$ denote the unit vector along the $i$th coordinate axis of
$\R^K$, and note we have a collection of identities relating
$\omega_{\va,\vb,\vc}$ values for different state vectors,
see~\cite{IsaacsonRDMENote}. For example,
\begin{equation*}
  a_{i} \, \omega_{\va,\vb,\vc} = (a_{i'}+1) \omega_{\va+e_{i'}-e_{i},\vb,\vc}.
\end{equation*}
Multiplying~\cref{eq:LhFrewrite} by $\omega_{\va,\vb,\vc}$ we then find
\begin{equation} \label{eq:LhvPDef}
  \begin{aligned}
    (\diffoph &\vP)
    (\va,\vb,\vc,t) \\
    &= D^{\textrm{A}} \sum_{i= 1}^{K} \sum_{i'= 1}^{K} \brac{ \paren{\Delta_h^T}_{i i'} 
      (a_{i'}+1) P(\va+\ve_{i'}-\ve_{i},\vb,\vc,t)
      - \paren{\Delta_h^T}_{i' i} a_{i} P(\va,\vb,\vc,t)} \\
    &+ D^{\textrm{B}} \sum_{j= 1}^{K} \sum_{j'= 1}^{K} \brac{ \paren{\Delta_h^T}_{j j'} 
      (b_{j'}+1) P(\va,\vb + \ve_{j'} - \ve_{j}, \vc, t) 
      - \paren{\Delta_h^T}_{j' j} b_{j} P(\va,\vb,\vc,t)} \\
    &+ D^{\textrm{C}} \sum_{k= 1}^{K} \sum_{k'= 1}^{K} \!\brac{ \paren{\Delta_h^T}_{k k'} 
      (c_{k'}+1) P(\va,\vb,\vc + \ve_{k'}\!-\ve_{k},t)
      - \paren{\Delta_h^T}_{k' k} c_{k} P(\va,\vb,\vc,t)}.
  \end{aligned}
\end{equation}
Similarly,
\begin{equation*}
  (\Rph \vF(t))_{\bi^a,\bj^b,\bk^c} =
  -\paren{\sum_{i=1}^K \sum_{j=1}^K \kp_{i j} a_{i} b_{j}} \Fabcijk(t) 
  + \frac{(a+1)(b+1)}{c} \sum_{i=1}^{K} \sum_{j=1}^{K} \sum_{k=1}^K
  \kp_{i j k} c_{k} F_{\bi^a \cup i, \bj^b \cup j, \bk^c \setminus k}(t),
\end{equation*}
so that 
\begin{multline} \label{eq:RphPDef}
  (\Rph \vP)(\va,\vb,\vc,t) 
  = -\paren{\sum_{i=1}^K \sum_{j=1}^K \kp_{i j} a_{i} b_{j}} \!\!P(\va,\vb,\vc,t) \\
  + \sum_{i=1}^{K} \sum_{j=1}^{K} \sum_{k=1}^K
  \kp_{i j k} (a_{i}+1)(b_{j}+1) P(\va+\ve_{i},\vb+\ve_{j},\vc-\ve_{k},t),
\end{multline}
and
\begin{equation*}
  (\Rmh \vF(t))_{\bi^a,\bj^b,\bk^c} = 
  - \paren{ \sum_{k=1}^{K} \km_{k} c_{k}} \Fabcijk(t) 
  + \frac{c+1}{a b} \sum_{i=1}^K \sum_{j=1}^K \sum_{k=1}^K \km_{i j k} a_{i} b_{j}
  F_{\bi^a \setminus i, \bj^b \setminus j, \bk^c \cup k}(t),
\end{equation*}
so that 
\begin{equation} \label{eq:RmhPDef}
  (\Rmh \vP(t))(\va,\vb,\vc,t) = 
  - \paren{ \sum_{k=1}^{K} \km_{k} c_{k}} P(\va,\vb,\vc,t) 
  + \sum_{i=1}^K \sum_{j=1}^K \sum_{k=1}^K \km_{i j k} (c_{k}+1)
  P(\va - \ve_{i}, \vb - \ve_{j}, \vc + \ve_{k}, t).
\end{equation}

\begin{table}[t]
  \centering
  \bgroup
  \renewcommand{\arraystretch}{1.2}
  \begin{tabular}{|l|c|c|c|}
    \hline
    & Transitions & Transition Rates & Upon Transition Event\\
    \hline 
    \multirow{6}{*}{\parbox{2cm}{Chemical Reactions:}}
    &\multirow{3}{*}{$\textrm{A}_i + \textrm{B}_j \to \textrm{C}_k$} & \multirow{3}{*}{$\kpijk a_i b_j$} 
                                   & $\textrm{A}_i := \textrm{A}_i - 1$\\  
    & & & $\textrm{B}_j := \textrm{B}_j - 1$\\
    & & & $\textrm{C}_k := \textrm{C}_k + 1$ \\ \cline{2-4}
    &\multirow{3}{*}{$\textrm{C}_k \to \textrm{A}_i + \textrm{B}_j$} & \multirow{3}{*}{$\kmijk c_k$} 
    & $\textrm{A}_i := \textrm{A}_i + 1$ \\
    & & & $\textrm{B}_j := \textrm{B}_j + 1$ \\
    & & & $\textrm{C}_k := \textrm{C}_k - 1$\\
    \hline
  \end{tabular}
  \egroup
  \caption{ 
    Summary of possible \emph{reactive} transitions for the general multi-particle
    $\textrm{A} + \textrm{B} \leftrightarrows \textrm{C}$ reaction in the CRDME~\cref{eq:vPEq}. Here
    $a_i$ denotes the number of \textrm{A} molecules in voxel $V_i$, with $b_j$ and $c_k$ defined
    similarly. Transition rates give the probability per time for a reaction to occur (i.e. the propensities). 
    The final column explains how to update the system state upon occurrence of the reaction.
  }
  \label{tab:aAndBToCRxsV2}
\end{table}

The coupled system of linear ordinary differential-difference
equations given by~\cref{eq:vPEq} with the operator
definitions~\cref{eq:LhvPDef}, \cref{eq:RphPDef}
and~\cref{eq:RmhPDef} gives a master equation for the probability the
vector jump process $(\vA(t),\vB(t),\vC(t))$ has the value
$(\va,\vb,\vc)$ at time $t$. By analogy with the RDME, which is
usually written in terms of these state variables, we
call~\cref{eq:vPEq} the convergent reaction-diffusion master equation
(CRDME).  The possible diffusive transitions for the jump process the
CRDME describes are identical to those enumerated in
\cref{tab:aAndBToCRxs}, while the possible reactive transitions
are given in \cref{tab:aAndBToCRxsV2}. Note, as we described in
\cref{S:rxDiscretMethod}, simulating the set of reactive
transitions in \cref{tab:aAndBToCRxs} with the associated update
rules upon reaction events is statistically equivalent to simulating
the set of reactions in \cref{tab:aAndBToCRxsV2}. As such, we
have shown that the transitions in \cref{tab:aAndBToCRxs} give
the generalization of the two-particle transitions in
equation~\cref{eq:twopartCRDME} to the general multi-particle case.


\section{Evaluating $\kpij$ and $\kpijk$ integrals for the Doi Volume-Reactivity Model}
\label{ap:hypervolCalc}
In \cref{S:implementation} we showed that $\kpij$ in convex
domains, and $\kpijk$ in general domains, could be given in terms of
area fractions $\phi_{i j}$ and $\phi_{i j k}$:
\begin{align*}
  \kpij &= \lambda \phi_{i j}, &
  \kpijk &= \lambda \phi_{i j k},
\end{align*}
where
\begin{align}
  \phi_{i j} &:
               = \frac{1}{\abs{V_{i j}}} \int_{V_i} \abs{B_{\rb}(\vx) \cap V_j} \, d\vx, \label{eq:phiij}\\
  \phi_{i j k} &:
                 = \frac{1}{\abs{V_{i j}}} \int_{V_i} \abs{B_{\rb}(\vx) 
                 \cap V_j \cap \hat{V}_k(\vx)} \, d\vx, \label{eq:phiijk}
\end{align}
and $\hat{V}_k(\vx)$ is defined by~\cref{eq:VhatDef}. Evaluating
these area fractions requires evaluation of the integrands, which
correspond to areas of intersection between disks of radius $\rb$ and
polygons. Our general approach to evaluating these integrals was to
develop fast and accurate routines to evaluate the integrands, and
then evaluate the outer integrals over $V_i$ using MATLAB's
\texttt{integral2} routine. The method we now describe is a
generalization of the method we outlined in~\cite{IsaacsonCRDME2013}
on Cartesian meshes (where the voxels $\{V_i\}_{i=1}^K$ were squares).

We begin by considering how to evaluate the area of intersection
$\abs{D \cap P}$ between a disk $D$ and a polygon $P$. 
The boundaries of each set will be given by $\partial D$ and
$\partial P$ respectively, with $\partial (D \cap P)$ denoting the
boundary curve(s) of intersection between the two sets.  Assume the
disk has radius $r$, with $\vC = (c_0,c_1)$ labeling its center. We
denote by $\{\vP_i\}_{i=1}^n$ a counter-clockwise (CCW) ordering of
the polygon's vertices, and for convenience define
$\vP_{n+1} := \vP_1$. Let $S_i$ label the side of the polygon given
by the line $\vL_i(s)$ connecting $\vP_i$ to $\vP_{i+1}$, for
$i = 1,\dots,n$. In parametric coordinates,
\begin{equation*}
  S_i = \{ \vL_i(s) \mid \vL_i(s) = \vP_i + s \vT_i, \, s \in \brac{0,1} \},
\end{equation*}
where $\vT_i = (\vP_{i+1} - \vP_{i})$ denotes the (unnormalized)
tangent vector to the line. 
Similarly, we describe the circle $\partial D$ by the parametric
curve $\vgam(t)$,
\begin{equation*}
  \partial D = \{  \vgam(t) \mid \vgam(t) = \vC + r \veta(t), \, t \in \left[0,2 \pi\right) \},
\end{equation*}
where $\veta(t) = (\cos(t),\sin(t))$ denotes the unit normal to the
circle. The (unnormalized) tangent to the circle is then
$\vT(t) = r \D{\veta}{t}$.

With these definitions, the area of intersection between $D$ and $P$ is given by
\begin{align*}
  \abs{D \cap P} &= \iint_{D \cap P} \, dx \, dy, \\
                 &= \frac{1}{2} \iint_{D \cap P} \nabla \cdot (x,y) \, dx \, dy,\\
                 &= \frac{1}{2} \int_{\partial (D \cap P)} (x,y) \cdot \veta(x,y) \, dl,
\end{align*}
where we have used the divergence theorem. Here $\veta(x,y)$ denotes
the unit outward normal at the point $(x,y)$, and $dl$ denotes the
differential along the boundary curve at the point $(x,y)$. Using
indicator functions we may split up the last integral as
\begin{align}
  \abs{D \cap P} &= \frac{1}{2} \int_{\partial P} \paren{(x,y) \cdot \veta(x,y)} \ind_{D}(x,y) \, dl
                   + \frac{1}{2} \int_{\partial D} \paren{(x,y) \cdot \veta(x,y)}  \ind_{P}(x,y) \, dl \notag \\
                 &= \frac{1}{2} \sum_{i=1}^{n} \brac{\int_{S_i} \paren{(x,y) \cdot \veta_i} \ind_{D}(x,y) \, dl}
                   + \frac{1}{2} \int_{\partial D} \paren{(x,y) \cdot \veta(x,y)}  \ind_{P}(x,y) \, dl. \label{eq:finalAreaForm}
\end{align}
Here $\veta_i$ denotes the \emph{constant} unit outward normal to $P$
on $S_i$. To calculate $\abs{D \cap P}$ we then need to evaluate these
$n+1$ integrals. Our strategy will be to determine the points of
intersection between $\partial D$ and $\partial P$. These divide each
line segment $S_i$, and the circle $\partial D$, into a collection of
sub-arcs along which the preceding integrals can be evaluated
analytically.

\subsection{Finding Points of Intersection of $\partial D$ and $\partial P$}
\label{S:findIsectPts}
We first calculate the locations where the polygon is intersected by
the circle. Let $s \in \brac{0,1}$ denote the parametric coordinate
describing the line segment, $S_i$.  Suppose the circle intersects
$S_i$ at the (parameterization) points $\{s_i^j\}_{j=1}^{N_i}$ (where
$N_i$ can be $0$, $1$, or $2$). Let $s_i^0 = 0$ and $s_i^{N_i+1} = 1$.

The $\{s_i^j\}$ are determined by the solutions for $s \in \brac{0,1}$
to the equation
\begin{equation} \label{eq:SqPtsFromIsect}
  \abs{ \vL_i(s) - \vC}^2 = r^2.
\end{equation}
Let 
\begin{align*}
  b := \frac{\vT_i \cdot (\vP_i - \vC)}{\abs{\vT_i}^2}, &&\text{and} 
  &&d := \frac{\abs{\vP_i - \vC}^2 - r^2}{\abs{\vT_i}^2}.
\end{align*}
Expanding out~\cref{eq:SqPtsFromIsect} and using the quadratic formula,
we find the circle crosses $S_i$ only for $b^2-d > 0$. The points
of intersection are given by 
\begin{equation*}
  s^* = -b \pm \sqrt{b^2 - d},
\end{equation*}
wherever $s^* \in \brac{0,1}$. These give the values of the
$\{s_i^j\}$.  (Note, in practice we do not include points where the
circle touches but does not cross $S_i$, since they are not needed to
calculate the area of intersection.)

We now calculate the locations where the circle is intersected by the
polygon. We denote by $\{t^j\}_{j=1}^M$ the parameterization values
(on the circle) of the points of intersection with the polygon.  Let
$t^{M+1} = t^1 + 2 \pi$. If there are no points of intersection
$M = 0$, and calculating the area of intersection is trivial since we
need only check if one point of the polygon is inside the circle, or
one point of the circle is inside the polygon. We ignore this special
case.

Each $t^j$ corresponds to some $s_i^{j'}$, $i\in\{1,\dots,n\}$,
$j'\in\{1,\dots,N_i\}$. As such, for each $s_i^{j'}$ we solve
\begin{equation*}
\vC + r (\cos(t^*),\sin(t^*)) = \vL_i(s_i^{j'})
\end{equation*}
for $t^*$.  Defining the components of $\vP_i = (P_i^0,P_i^1)$ and $\vT_i
= (T_i^0,T_i^1)$, the solution to the preceding equation is given by
\begin{equation*}
  t^* = \arctan \paren{ \frac{ P_i^1 - c_1 + s_i^{j'} T_i^1}{P_i^0 - c_0 + s_i^{j'} T_i^0}},
\end{equation*}
with the range of $\arctan$ taken in $\left[0,2 \pi\right)$. The
values of $t^*$ then determine the $\{t^j\}_{j=1}^M$.

\subsection{Evaluating the line integrals for the area of intersection}
Knowing the parametric intersection points
$\{s_i^j\}_{i=1,\,j=1}^{i=n,\,j=N_i}$ and $\{t^j\}_{j=1}^M$, we can
evaluate the line integrals along $S_i$ and $\partial D$
in~\cref{eq:finalAreaForm}.  We first explain how to evaluate the
line integral along the side $S_i$ of the polygon. Let $(x(s),y(s))$
denote the parametric curve over which a given line integral is
defined.  Since $dl = \abs{\vL_i'(s)} ds = \abs{\vT_i} ds$, converting
to parametric coordinates we find that
\begin{align*}
  \int_{S_i} \paren{(x,y) \cdot \veta_i} \ind_{D}(x,y) \, dl
  &= \sum_{j=0}^{N} \int_{s_i^j}^{s_i^{j+1}} \paren{ \vL_i(s) \cdot
   \veta_i} \ind_{D}(x(s),y(s)) \abs{\vT_i} ds, \\
 &= \sum_{j=0}^{N} \int_{s_i^j}^{s_i^{j+1}} \paren{\vP_{i}  \cdot \veta_i} 
 \ind_{D}(x(s),y(s)) \abs{\vT_i} ds.
\end{align*}
Here the last equation follows using the definition of $\vL_i(s)$ and
that $\vT_i$ is perpendicular to $\veta_i$.

Let $L_i^j$ denote the subsegment of $S_i$ for
$s \in (s_i^{j},s_i^{j+1})$. $L_i^j$ is either completely inside or
completely outside $D$. This can be determined by checking where the
midpoint of $L_i^j$ lies.  As such, in the preceding equation each
indicator function is constant within each individual integral in the
sum over $j$. We then find
\begin{equation} \label{eq:SqSegInt}
\int_{s_i^j}^{s_i^{j+1}} \paren{\vP_{i}  \cdot \veta_i}  \ind_{D}(x(s),y(s)) \abs{\vT_i} ds
= \begin{cases}
  \paren{\vP_{i} \cdot \veta_i} \abs{\vT_i} (s_i^{j+1}-s_i^{j}), &L_i^j \in D,\\
  0, &L_i^j \notin D.
\end{cases}
\end{equation}
From this formula the line integrals along each $S_i$ are completely
specified once the points of intersection, $\{s_i^j\}_{j=1}^{N_i}$ are
known.

We use a similar approach to evaluate the integral over the circle
$\partial D$.  Since $dl = \abs{\vT(t)} dt$, converting to parametric
coordinates we find
\begin{align*}
\int_{\partial C} \paren{(x,y) \cdot \veta(x,y)}  \ind_{P}(x,y) \, dl 
  &=  \sum_{j=1}^{M} \int_{t^j}^{t^{j+1}} \paren{\vgam(t) \cdot \veta(t)}
\ind_{P}(x(t),y(t)) \abs{\vT(t)} \, dt, \\
&= \sum_{j=1}^{M} \int_{t^j}^{t^{j+1}} \paren{\vC \cdot \veta(t) + r}
\ind_{P}(x(t),y(t)) \abs{\vT(t)} \, dt, \\
&= \sum_{j=1}^{M} \int_{t^j}^{t^{j+1}} \!\!\!\!\! r \paren{c_0 \cos(t) + c_1 \sin(t) + r} \ind_{P}(x(t),y(t)) \, dt.
\end{align*}
Each integral within the sum corresponds to a sub-arc of the circle
that is entirely within or outside the polygon. As such, the indicator
function is either identically one or zero within each integral, which
can be determined by testing if the midpoint of the arc is within $P$.
We therefore conclude
\begin{multline} \label{eq:CircArcInt}
  \int_{t^j}^{t^{j+1}} r \paren{c_0 \cos(t) + c_1 \sin(t) + r} \ind_{P}(x(t),y(t)) \, dt = \\
  \begin{cases}
     r \brac{c_0 \paren{\sin(t^{j+1}) - \sin(t^j)}
       - c_1 \paren{\cos(t^{j+1}) - \cos(t^j)} + r (t^{j+1}-t^j)}, &\text{arc $\in P$,}\\
     0, &\text{arc $\notin P$.}
   \end{cases}
\end{multline}
From this formula the line integral around $\partial D$ is fully
determined once the points of intersection $\{t^j\}_{j=1}^M$ are
known.

\subsection{Overall Algorithm}
\begin{algorithm}[tbp]
  \caption{Evaluating the area of intersection $\abs{D \cap P}$.}
  \label{alg:areaOfIsect}
  \begin{algorithmic}[1]
    \Function{AreaOfIntersection}{disk $D$, polygon $P$}
    \State{$A = 0$.}
    \For{$i = 1,\dots,n$}
    \State{For side $S_i$ of polygon $P$ calculate the intersection points $\{s_i^j\}_{j=1}^{N_i}$.}
    \For{$j = 0,\dots,N_i$}
    \State{$A = A \,+ $ value of~\cref{eq:SqSegInt}}
    \EndFor
    \State{Calculate the intersection points on the circle $\partial D$, $\{t^j\}_{j=1}^M$.}
    \For{$j = 1,\dots,M$}
    \State{$A = A \,+ $ value of~\cref{eq:CircArcInt}}
    \EndFor    
    \EndFor
    \State{\Return $\abs{D \cap P} = \frac{1}{2} A$}
    \EndFunction
  \end{algorithmic}
\end{algorithm}
Using~\cref{eq:SqSegInt} and~\cref{eq:CircArcInt} we can rapidly and
accurately determine the area of intersection $\abs{D \cap P}$. The
overall algorithm we use is summarized in \cref{alg:areaOfIsect}.  Our
implementation is in C++, and where possible uses the CGAL
library~\cite{CGAL} to accurately handle various geometry operations
(triangulation of polygons, calculation of intersection points,
testing if points are inside polygons, etc).

To evaluate the integral~\cref{eq:phiij} for $\phi_{i j}$ we call
MATLAB's numerical integration routine \texttt{integral2}, using our
C++ version of \cref{alg:areaOfIsect} to evaluate the
integrand $\abs{B_{\rb}(\vx) \cap V_j}$. To do the integration over a
given polygon, we triangulate the polygon and sum the results of
integration over each sub-triangle. To calculate $\phi_{i j k}$ requires
one extra step as the integrand
$\abs{B_{\rb}(\vx) \cap V_j \cap \hat{V}_k(\vx)}$ in~\cref{eq:phiijk}
involves the translated and scaled set $\hat{V}_k(\vx)$. Here our
basic approach is to pass a MATLAB wrapper routine that evaluates the
integrand to \texttt{integral2}.  This routine first calculates the
polygon(s) of intersection $\{P_{\ell}\}$ corresponding to
$V_j \cap \hat{V}_k(\vx)$ using MATLAB's \texttt{polybool}
function. For each resulting polygon $P_{\ell}$ the intersection area
$\abs{B_{\rb}(\vx) \cap P_{\ell}}$ is then calculated using our C++
implementation of \cref{alg:areaOfIsect}.  The integrand is
then calculated as
\begin{equation*}
  \abs{B_{\rb}(\vx) \cap V_j \cap \hat{V}_k(\vx)} = \sum_{\ell} \abs{B_{\rb}(\vx) \cap P_{\ell}}.
\end{equation*}

\section*{Acknowledgments}
SAI would like to thank Steven Andrews, Aleksander Donev and David
Isaacson for helpful discussions related to this work.  SAI and YZ
were supported by National Science Foundation award DMS-1255408. SAI
was partially supported by a grant from the Simons Foundation, and
thanks the Isaac Newton Institute of Mathematical Sciences for hosting
him as a visiting Simons Fellow for the programme on Stochastic
Dynamical Systems in Biology while completing this work.

\bibliographystyle{elsarticle-num.bst} 
\bibliography{lib}

\end{document}

%% file: my_commands.tex
\usepackage{bbm}

\renewcommand{\vec}[1]{\boldsymbol{#1}}
\newcommand{\vecs}[1]{\boldsymbol{#1}}
\newcommand{\paren}[1]{\left(#1\right)}
\newcommand{\brac}[1]{\left[#1\right]}
\newcommand{\E}{\mathbb{E}}
\newcommand{\avg}[1]{\E[#1]}

\newcommand{\D}[2]{\frac{d#1}{d#2}}
\newcommand{\PD}[2]{\frac{\partial#1}{\partial#2}}

\newcommand{\lap}[1]{\Delta#1}

\newcommand{\abs}[1]{\left|#1\right|}

\DeclareMathOperator{\prob}{Pr}

\newcommand{\DA}{D^{\textrm{A}}}
\newcommand{\DB}{D^{\textrm{B}}}
\newcommand{\DC}{D^{\textrm{C}}}

\newcommand{\bi}{\vec{i}}
\newcommand{\bj}{\vec{j}}
\newcommand{\bk}{\vec{k}}

\newcommand{\ve}{\vec{e}}

\newcommand{\veta}{\vecs{\eta}}

\newcommand{\vp}{\vec{p}}
\newcommand{\vP}{\vec{P}}
\newcommand{\vC}{\vec{C}}
\newcommand{\vT}{\vec{T}}
\newcommand{\vL}{\vec{L}}
\newcommand{\vgam}{\vec{\gamma}}
\newcommand{\va}{\vec{a}}
\newcommand{\vb}{\vec{b}}
\newcommand{\vc}{\vec{c}}
\newcommand{\vx}{\vec{x}}
\newcommand{\vy}{\vec{y}}
\newcommand{\vz}{\vec{z}}

\newcommand{\vf}{\vec{f}}
\newcommand{\vF}{\vec{F}}

\newcommand{\vV}{\mathcal{V}}

\newcommand{\vq}{\vec{q}}

\newcommand{\vO}{\vec{0}}

\newcommand{\vA}{\vec{A}}
\newcommand{\vB}{\vec{B}}
\newcommand{\rb}{r_{\textrm{b}}}

\newcommand{\vqa}{\vec{q}^{a}}
\newcommand{\vqb}{\vec{q}^{b}}
\newcommand{\vqc}{\vec{q}^{c}}
\newcommand{\vQ}{\vec{Q}}
\newcommand{\vQa}{\vec{Q}^{a}}
\newcommand{\vQb}{\vec{Q}^{b}}
\newcommand{\vQc}{\vec{Q}^{c}}

\newcommand{\ind}{\mathbbm{1}}

\def\R{\mathbb{R}}

\newcommand{\pb}{p_{\textrm{b}}}

\newcommand{\Tb}{T_{\textrm{bind}}}

\newcommand{\Rset}{\mathcal{R}}

\renewcommand{\epsilon}{\varepsilon}
\renewcommand{\rb}{\epsilon}

\newcommand{\ballunbind}{B_{(1-\gamma)\rb}}
\newcommand{\kp}{\kappa^{+}}
\newcommand{\km}{\kappa^{-}}
\newcommand{\Rp}{\mathcal{R}^{+}}
\newcommand{\Rm}{\mathcal{R}^{-}}
\newcommand{\diffop}{\mathcal{L}}
\newcommand{\Rph}{\mathcal{R}^{+}_h}
\newcommand{\Rmh}{\mathcal{R}^{-}_h}
\newcommand{\diffoph}{\mathcal{L}_h}
\newcommand{\fabc}{f^{(a,b,c)}}
\newcommand{\Fabcijk}{F_{\bi^a \bj^b \bk^c}}
\newcommand{\vVabc}{\vV_{\bi^a \bj^b \bk^c}}

\newcommand{\kpijk}{\kappa_{i j k}^{+}}
\newcommand{\kpkij}{\kappa_{k\vert i j}^{+}}
\newcommand{\kpij}{\kappa_{i j}^{+}}
\newcommand{\kmk}{\kappa_{k}^{-}}
\newcommand{\kmijk}{\kappa_{i j k}^{-}}
\newcommand{\kmkij}{\kappa_{i j \vert k}^{-}}

\newcommand{\kprod}{k_{\textrm{1}}}
\newcommand{\kdecay}{k_{\textrm{2}}}
